\documentclass[a4paper,oneside,final]{amsart}

\usepackage{geometry}

\usepackage[english]{babel}
\usepackage{graphicx}
\usepackage{wrapfig}
\usepackage{mathtools} 		
\usepackage{amsmath, amssymb, amsfonts, amsthm} 
\usepackage{hyperref}
\usepackage{url}            
\usepackage[dvipsnames]{xcolor}
\usepackage{color}         
\usepackage{changes}		
\usepackage{pdfpages}
\usepackage{enumerate}

\usepackage{cleveref}

\usepackage{ifdraft}

\usepackage{tikz}
\usetikzlibrary{calc,arrows.meta,bending}
\ifoptionfinal{}{%
    \usetikzlibrary{external}\tikzexternalize
}
\usepackage{pgfplots}
\pgfplotsset{compat=newest}

\def\centerarc[#1](#2)(#3:#4:#5)(#6);%
{
    \draw[#1]([shift=(#3:#5)]#2) arc (#3:#4:#5) node[below] { #6 };
}

\usepackage[subrefformat=parens]{subcaption}

\usepackage{booktabs}

\usepackage{rotating}

\usepackage{autonum}

\usepackage[nice]{nicefrac}
\hypersetup{colorlinks=true,urlcolor=blue,linkcolor=black}
\theoremstyle{plain}
\newtheorem{lem}{Lemma}[section]
\newtheorem{thm}[lem]{Theorem}
\newtheorem{cor}[lem]{Corollary}
\newtheorem{prop}[lem]{Proposition}

\theoremstyle{definition}
\newtheorem{defn}[lem]{Definition}
\newtheorem{rem}[lem]{Remark}
\newtheorem{ex}[lem]{Example}

\DeclareMathOperator*{\vspan}{span}
\DeclareMathOperator*{\argmin}{arg\,min}

\newcommand{\NN}{\mathbb N}
\newcommand{\ZZ}{\mathbb Z}

\newcommand{\RR}{\mathbb R}
\newcommand{\R}{\mathbb R}

\newcommand{\sE}{\mathcal{E}}

\newcommand{\sA}{\mathcal{A}}
\newcommand{\sG}{\mathcal{G}}

\newcommand{\sU}{\mathcal{U}}

\newcommand{\intd}[1]{\, \mathrm{d}#1}

\numberwithin{equation}{section}

\newcommand{\hl}[1]{\textcolor{black}{#1}}

\usetikzlibrary{angles,quotes}

\makeatletter
\@namedef{subjclassname@2020}{%
  \textup{2020} Mathematics Subject Classification}
\makeatother

\title[On evolving heterogeneous elastic wires]{Conservation, convergence, and computation \\ for evolving heterogeneous elastic wires}

\author[A.~Dall'Acqua]{Anna Dall'Acqua}
\address[A.~Dall'Acqua]{Institute of Applied Analysis, Ulm University, Helmholtzstra\ss e 18, 89081 Ulm, Germany.}
\email{anna.dallacqua@uni-ulm.de}
\author[G.~Jankowiak]{Gaspard Jankowiak}
\address[G.~Jankowiak]{Department of Mathematics and Statistics, University of Konstanz, 78457 Konstanz, Germany.}
\email{gaspard.jankowiak@uni-konstanz.de}
\author[L.~Langer]{Leonie Langer}
\address[L.~Langer]{Institute of Applied Analysis, Ulm University, Helmholtzstra\ss e 18, 89081 Ulm, Germany.}
\email{leonie.langer@uni-ulm.de}
\author[F.~Rupp]{Fabian Rupp}
\address[F.~Rupp]{Faculty of Mathematics, University of Vienna, Oskar-Morgenstern-Platz 1, 1090 Vienna, Austria.}
\email{fabian.rupp@univie.ac.at}

\begin{document}
\maketitle
\begin{abstract}
The elastic energy of a bending-resistant interface depends both on its geometry and its material composition. We consider such a heterogeneous interface in the plane, modeled by a curve equipped with an additional density function. The resulting energy captures the complex interplay between curvature and density effects, resembling the Canham--Helfrich functional.
We describe the curve by its inclination angle, so that the equilibrium equations reduce to an elliptic system of second order.
After a brief variational discussion, we investigate the associated nonlocal $L^2$-gradient flow evolution, a coupled quasilinear parabolic problem. We analyze the (non)preservation of quantities such as convexity, positivity, and symmetry, as well as the asymptotic behavior of the system. 
The results are illustrated by numerical experiments.
\end{abstract}





\section{Introduction}
In shape optimization, the energy and thus the equilibrium 
configuration of a bending-resistant surface can depend both on its geometry as well as the distribution of some physical densities, representing, 
for example mass, electrostatic charge, temperature, etc. For instance,  the shape can depend on the material composition and vice versa, as it is the case for certain biomembranes \cite{JulicherLipowsky93,BHW2003,MG2005}. This phenomenon is not accounted for in the classical Canham--Helfrich model describing the characteristic biconcave shape of red blood cells \cite{Canham,Helfrich}. If the densities are discrete, a variational existence theory has been established, relying on either rotational symmetry 
\cite{CMV13,Helmers15}
or weak formulations using curvature varifolds \cite{BLS20}.

A one-dimensional model which attributes for the aforementioned interplay between curvature and density-related effects has been introduced in \cite{BJSS2020}, where closed planar curves are equipped with a (nondiscrete) density function. For the discrete setting, we refer to \cite{Helmers11} and also mention a related discrete-to-continuum $\Gamma$-limit result \cite{dondl2023gammaconvergence}. Inspired by the Canham--Helfrich model depending on a spontaneous curvature, we define a generalized Euler--Bernoulli energy for a planar heterogeneous elastic wire $\gamma$ with density function $\rho$ by
\begin{align}
\label{eq:Egamma}
    \sE_\mu(\gamma,\rho)&=\frac12\int_\gamma\left(\beta(\rho)( \kappa-c_0)^2+\mu\,(\partial_s\rho)^2\right)\intd s.
\end{align}
Here,  
$\beta$ is a smooth positive function, describing the density-modulated bending stiffness of the wire, the parameter $c_0 \in \RR$ denotes the spontaneous curvature, and 
$\mu>0$ models the diffusivity of the density. Further, $\kappa$ is the signed curvature of $\gamma$, $\intd s=\vert\partial_x\gamma(x)\vert\intd x$ is the arc-length element,
and $\partial_s=\left\vert\partial_x\gamma(x)\right\vert^{-1}\partial_x$. Consequently, the curve strives for a preferred curvature $c_0\in \RR$ (determined by the material, see \cite{MG2005}), and the defect is penalized depending on the continuous density distribution. 
In the special case where $\rho$ is constant and $c_0=0$, we essentially retrieve the classical Euler--Bernoulli elastic energy given by
\begin{align}
\label{eq:classicalE}
    \sE(\gamma)=\frac12\int_\gamma\kappa^2\intd s.
\end{align}
Critical points of \eqref{eq:classicalE} with prescribed length are called elasticae and have been classified in several previous works. In particular, the only closed elasticae are multifold coverings of the circle and of the figure eight elastica (depicted in \Cref{fig:omega 0 figure 8}), see for instance \cite{LS1984}, \cite{DHMV2008}, or \cite[Lemma 5.4]{MR2021}. Such a classification is, of course, not available for \eqref{eq:Egamma}, but, as we shall see, elasticae play an important role also in this general case.

Actually, for planar curves, the order of the energy \eqref{eq:Egamma} may be reduced as follows.
Since \eqref{eq:Egamma} is invariant under orientation preserving reparametrisations, we assume that the planar curve $\gamma$ with prescribed length $L$ is parametrized by arc-length. 
Then there exists an (inclination) angle function $\theta\colon [0,L]\to\RR$ such that $\partial_s\gamma=(\cos\theta,\sin\theta)$. 
Modulo isometries of $\RR^2$, the curve $\gamma$  together with its orientation is uniquely determined by $\theta$. More precisely, for some $\gamma(0)\in\RR^2$ we have
\begin{align}
        \gamma\colon[0,L]\to\RR^2,\quad\gamma(s)=\int_0^s\begin{pmatrix}
            \cos\theta\\ \sin\theta
        \end{pmatrix}\intd r+\gamma(0).\label{eq:gamma}
    \end{align}
With $\theta$ and the density function $\rho\colon[0,L]\to\RR$, the energy \eqref{eq:Egamma} can be expressed by
\begin{align}
\label{Eallg}
\sE_\mu(\theta,\rho)=\frac{1}{2}\int_0^L\left(\beta(\rho)(\partial_s\theta-c_0)^2+\mu\,(\partial_s\rho)^2\right)\intd s,
\end{align}
using that $\kappa=\partial_s \theta$, so that $\sE_\mu(\gamma,\rho)=\sE_\mu(\theta,\rho)$. Naturally, to achieve compactness, the variational discussion of \eqref{Eallg} involves prescribing the length $L$ and the rotation index $\omega$ of the curve, and the integral of the density, i.e.\
\begin{align}\label{eq:fixedmass}
    \int_0^L \rho\intd s &= \nu L.
\end{align}
The length constraint and the preferred curvature $c_0\in\RR$ are competing forces for minimizing $\sE_\mu$, in general, since prescribing the length may not allow for a curve with constant curvature $\kappa\equiv c_0$.
In the sequel, the constants $
    L>0$, $\nu\in\RR$, $\mu>0$, $\omega \in \ZZ$, and $c_0 \in \RR,
$
are fixed, and we will refer to them as model parameters.
For the bending stiffness $\beta$, we assume
$
    \beta\in C^\infty(\RR) $ and $ \beta>0
$.

\subsection{Previous work}
The constrained minimization problem $\min \sE_\mu(\theta,\rho)$ with zero spontaneous curvature and rotation index equal to one 
was studied in \cite{BJSS2020}. 
In \cite{DLR2022}, some of the authors followed a dynamic approach and introduced the constrained $L^2$-gradient flow associated to \eqref{Eallg}.
It can be used to describe, in the simplified quasistatic and viscous regime, the continuous deformation towards an energetically more favourable state. In addition to the constraints on length, rotation index, and the integral of the density, we need to ensure that the angle function describes a closed curve.  
Overall, this results in the initial boundary value problem
\begin{align}
\label{eq:flow equation}
\begin{split}
\qquad\begin{cases}
\begin{tabular}{l l l} 
 \multicolumn{2}{l }{$\displaystyle\partial_t\theta=\partial_s\big(\beta(\rho)(\partial_s\theta-c_0)\big)+\lambda_{\theta1}\sin\theta-\lambda_{\theta2}\cos\theta$\hphantom{-----}}  
 & in $(0,T) \times [0,L]$, \\
 \multicolumn{2}{l }{$\displaystyle\partial_t\rho=\mu\partial_s^2\rho-\frac{1}{2}\beta'(\rho)(\partial_s\theta-c_0)^2-\lambda_\rho$}
 & in $(0,T) \times [0,L]$, \\
 $\theta(\cdot, L)-\theta(\cdot,0)=2\pi \omega$,\hphantom{-----}
 & $\rho(\cdot,L)=\rho(\cdot,0)$
 & on $[0,T)$,\\
 $\partial_s \theta(\cdot,L)=\partial_s\theta(\cdot,0)$,
 & $\partial_s \rho(\cdot,L)=\partial_s \rho(\cdot,0)$
 & on $[0,T)$,\\
 $\theta(0,\cdot)=\theta_0$,
 & $\rho(0,\cdot)=\rho_0$ 
 & on $[0,L]$,
\end{tabular} 
\end{cases}
\end{split}
\end{align}
with $(\theta_0,\rho_0)\in C^1([0,L])$ satisfying the boundary conditions. Since an angle function $\theta\colon[0,L]\to\RR$ represents a 
zeroth order closed curve if and only if 
\begin{align}
\label{eq:closedcurve}
    \int_0^L\sin\theta\intd s=\int_0^L\cos\theta\intd s=0, 
\end{align} it is required that $\theta_0$ satisfies \eqref{eq:closedcurve}.
To ensure \eqref{eq:closedcurve} along the flow, the nonlocal Lagrange multipliers $\lambda_{\theta1}$ and $\lambda_{\theta2}$ are given by
\begin{align}
\label{eq:lambdatheta}
\begin{pmatrix}
\lambda_{\theta1}(t )\\ \lambda_{\theta2}(t)
\end{pmatrix}
:=\Pi^{-1}(\theta)\int_0^L\begin{pmatrix}
\cos\theta \\ \sin\theta
\end{pmatrix}
\partial_s\theta \beta(\rho)(\partial_s\theta-c_0)\intd s,
\end{align}
where $\Pi^{-1}(\theta)(t)$ denotes the inverse of the matrix
\begin{align}\label{eq:def Pi}
\Pi(\theta)(t):=\begin{pmatrix}
\int_0^L\sin^2\theta\intd s & -\int_0^L\cos\theta\sin\theta\intd s \\
-\int_0^L\cos\theta\sin\theta\intd s & \int_0^L\cos^2\theta\intd s
\end{pmatrix}.
\end{align}

The matrix $\Pi$ is invertible
as long as the angle function $\theta$ describes a closed curve, see \cite[Remark 2.1]{DLR2022}.
Moreover, the nonlocal Lagrange multiplier $\lambda_\rho$ is chosen as
\begin{align}\label{eq:lambdarho}
\lambda_\rho(t):=-\frac{1}{2L}\int_0^L\beta'(\rho)(\partial_s\theta-c_0)^2\intd s.
\end{align}
This ensures conservation of the total mass, i.e.\ \eqref{eq:fixedmass} is preserved along the evolution.
For the convenience of the reader, we recall the previous results on existence and convergence which we build upon, see \cite[Theorems 1.2--1.4, Lemma 3.5, and Remark 3.7]{DLR2022}.
\begin{thm}
\label{thm:zsf artcl1}
Suppose the initial datum $(\theta_0,\rho_0)\in C^{\infty}([0,L])$ satisfies \eqref{eq:fixedmass}, \eqref{eq:closedcurve}, and 
\begin{align}\label{eq:bcid}
\begin{array}{lllll}
    \theta_0( L)-\theta_0(0)=2\pi \omega, &  & \partial_s \theta_0(L)=\partial_s\theta_0(0), &  & \partial_s^2\theta_0(L)=\partial_s^2\theta_0(0),\\
     \rho_0(L)=\rho_0(0), &  & \partial_s \rho_0(L)=\partial_s \rho_0(0),&  &\partial_s^2 \rho_0(L)=\partial_s^2\rho_0(0).
     \end{array}
\end{align}
Then, there exists a unique global solution $(\theta, \rho)\in C^\infty((0,\infty)\times[0,L])\cap C^0([0,\infty);C^2([0,L]))$ of \eqref{eq:flow equation}, depending continuously on the initial datum $(\theta_0,\rho_0)$. For all $t>0$, $\kappa(t,\cdot)=\partial_s\theta(t,\cdot)$ and $\rho(t,\cdot)$ can be extended to smooth $L$-periodic functions on $\RR$. Moreover, the solution satisfies 
\begin{align}\label{eq:globalW32bounds}
    \limsup_{t\to\infty}\left(\Vert \theta(t)\Vert_{W^{3,2}(0,L)}+\Vert \rho(t)\Vert_{W^{3,2}(0,L)}\right) <\infty,
\end{align}
and subconverges, as $t\to\infty$, in $C^2([0,L])$ to a stationary solution, i.e.\ a solution of
\begin{align}\label{eq:stationary}
\begin{cases}
\begin{tabular}{l l l} 
 \multicolumn{2}{l }{$\displaystyle 0=\partial_s\big(\beta(\rho)(\partial_s\theta-c_0)\big)+\lambda_{\theta1}\sin\theta-\lambda_{\theta2}\cos\theta$\hphantom{-----}}  
 & in $[0,L]$, \\
 \multicolumn{2}{l }{$\displaystyle 0=\mu\partial_s^2\rho-\frac{1}{2}\beta'(\rho)(\partial_s\theta-c_0)^2-\lambda_\rho$}
 & in $[0,L]$, \\
 $\theta(L)-\theta(0)=2\pi \omega$,\hphantom{-----}
 & $\rho(L)=\rho(0)$,
 &\\
 $\partial_s \theta(L)=\partial_s\theta(0)$,
 & $\partial_s \rho(L)=\partial_s \rho(0)$,
 &\\
 $\partial^2_s \theta(L)=\partial^2_s\theta(0)$,
 & $\partial^2_s \rho(L)=\partial^2_s \rho(0)$
 &
\end{tabular} 
\end{cases}
\end{align}
for some $\lambda_{\theta1},\lambda_{\theta2},\lambda_\rho\in\RR$.
If, in addition, $\beta$ is real analytic, then we have full convergence $(\theta(t),\rho(t))\to(\theta_\infty,\rho_\infty)$ in $C^2([0,L])$ as $t\to\infty$, for some $(\theta_\infty,\rho_\infty)$ solving \eqref{eq:stationary}.
\end{thm}
We will refer to a smooth function $(\theta_0,\rho_0)$ satisfying the assumptions of \Cref{thm:zsf artcl1} as an \emph{admissible initial datum} in the sequel.

The global existence and convergence of the system \eqref{eq:flow equation} is in accordance with previous work on the elastic flow, i.e.\  the $L^2$-gradient flow of \eqref{eq:classicalE}, both in its fourth order version for curves \cite{DKS2002,MantegazzaPozzetta21,Lin,DLP17,DPS16,DP14,MR4277362,MR4080255} and its second order flow for the angle function in the planar case \cite{W1993,NP2020,LLS15,OPW}. Due to a lack of maximum principle, along the fourth order flow properties like convexity or embeddedness of the initial datum do not need to be preserved \cite{Blatt,MMR2021}, see also \cite{L1989}.
As we shall see, for the second order system \eqref{eq:flow equation} such properties depend delicately on the choice of the model parameters.

Lastly, we mention that a similar system relating mean curvature flow and diffusion of a density on a hypersurface has been studied in \cite{MR4530425,ABG2022}.


\subsection{Main results and structure of the article}

This article extends the results obtained in \cite{DLR2022} and investigates several properties of solutions to \eqref{eq:flow equation}, a quasilinear coupled parabolic system of second order involving nonlocal Lagrange multipliers. 

Our work is inspired by the analysis of preserved quantities for curvature flows of hypersurfaces by Escher--Ito \cite{EI2005} and
by Wen's article \cite{W1993} on the gradient flow of \eqref{eq:classicalE} in terms of $\theta$.
We focus mainly on two aspects: the (non)preservation of several properties of the initial datum along the evolution and the discussion of conditions under which the limit and the rate of convergence of the system can be determined.
Since the limit configuration is a constrained critical point of \eqref{Eallg}, this also provides a partial classification of constrained critical points of $\sE_\mu$.
We complement the analysis by numerical experiments that have motivated our results. 

In \Cref{sec:minprob}, as a first step towards understanding the asymptotic behavior, we examine minimizers and constrained critical points of the functional, which are precisely the stationary solutions. As indicated by the form of the energy \eqref{Eallg}, a global minimizer has to have constant density if $\mu$ is sufficiently large, cf.\ also \cite{BJSS2020} for the case $\omega=1$, $c_0=0$. Remarkably, this is generically not true on the level of critical points, see \Cref{ex:non const rho crit point}. We provide sharp sufficient conditions for constrained critical points to be homogeneous elasticae, i.e.\ elasticae with constant density. 

In Section \ref{sec:qualprop}, we study the (non)preservation of properties of the initial datum along \eqref{eq:flow equation}.
The decisive advantage of working with the angle function $\theta$ 
is that the equation 
is
of second order. In contrast, working with the curve $\gamma$ yields a fourth order equation, like the classical elastic flow. 
Due to the reduction to second order, parabolic maximum principles are available for both evolution equations in \eqref{eq:flow equation} individually, but of course not for the full system.
A key difficulty in applying maximum principles is the structure of the relevant evolution equations, explaining the fundamentally different behavior for $c_0=0$ compared to $c_0\neq 0$.

First, we adapt the methods in \cite{A1988} to study the 
inflection points of the curve, i.e.\ the sign changes of the curvature. 
\begin{thm}\label{thm:kappa_zeros}
    Let $c_0=0$ and let $(\theta,\rho)$ be a global solution of \eqref{eq:flow equation}. Then both the number of zeros of $\kappa=\partial_s \theta$ and the number of inflection points of the associated curve are nonincreasing in time.
\end{thm}

Combining this result with further maximum principle arguments, we examine (strict) convexity along the evolution.

\begin{thm}[Preservation of convexity for $c_0=0$]
\label{prop:convex}
Let $(\theta,\rho)$ be the global solution of \eqref{eq:flow equation} with admissible initial datum  $(\theta_0,\rho_0)$ and $c_0=0$. Then 
\begin{enumerate}[(i)]
\item $\kappa_0\geq 0$ ($\kappa_0\leq 0$) on $[0,L]$ implies $\kappa\geq0$ ($\kappa\leq 0$) on $[0,\infty)\times[0,L]$; 
\item $\kappa_0>0$ ($\kappa_0<0$) on $[0,L]$ implies $\kappa>0$ ($\kappa<0$) on $[0,\infty)\times [0,L]$.
\end{enumerate}
\end{thm}

Both \Cref{thm:kappa_zeros} and \Cref{prop:convex} rely heavily on the vanishing of $c_0$. Indeed, even for $\vert c_0\vert\neq0$ small,
convexity is not preserved in general (see Example \ref{ex:counterconvex}).
This behavior is not easily predicted from the energy, since e.g.\ a large positive $c_0$ clearly favors positive curvature.
Moreover, despite the nonlinear structure of the equation for the density, under appropriate sharp assumptions on $\beta$ we are still able to apply maximum principles to examine sign-preservation of the density, see \Cref{prop:rhopos}. 

For the curve shortening flow, a classical application of the maximum principle is the preservation of embeddedness, cf.\ \cite{Gage_Hamilton_86}. While the evolution of the curves obtained via \eqref{eq:gamma} does not allow for these methods, we are still able to find an embeddedness-preserving energy threshold in \Cref{prop:li_yau_emb}.

Further, we show in Section \ref{sec:symmetry} that both $k$-fold rotational symmetry and axial symmetry of the initial datum are preserved along the evolution (see Propositions \ref{prop:rotsym preserved} and \ref{prop:presaxsym}), for simplicity restricting to the case $\omega=1$. 
In the rotationally symmetric case, the Lagrange multipliers $\lambda_{\theta1}$ and $\lambda_{\theta2}$ vanish, which enables us to generalize \Cref{prop:convex} for a $k$-fold rotationally symmetric initial datum as follows.

\begin{thm}
    \label{thm:pisym kappa geq c0} 
    Let $\omega=1$, $k\geq2$ and $c_0\in\R$. Let $(\theta_0,\rho_0)$ be an admissible initial datum corresponding to a $k$-fold rotationally symmetric heterogeneous curve 
    with $\kappa_0\geq c_0$ on $[0,L]$ and let $(\theta,\rho)$ be the solution of \eqref{eq:flow equation}. Then $\kappa\geq c_0$ on $[0,\infty)\times[0,L]$. Similarly, if 
    $\kappa_0\leq c_0$ on $[0,L]$, then $\kappa\leq c_0$ on $[0,\infty)\times[0,L]$.
\end{thm}

The convergence result in \cite{DLR2022} naturally raises the question of a characterization of the limit, which we can answer despite the large number of selectable parameters under suitable assumptions on $\beta$, $\nu$, and $\mu$ in \Cref{sec:proplim}. To that end, we rely on the properties of constrained critical points that we discussed in \Cref{sec:minprob}.

\begin{thm}[Asymptotic behavior under growth assumptions on $\beta$
] \label{thm:conv special beta neu}
    Suppose there exists $\bar{C}\geq 0$ such that
    \begin{align} \label{eq:beta'bound}
        \beta'(x)(\nu-x)\leq \bar{C}\beta(x)(\nu-x)^2\quad \text{ for all }x\in\RR.
    \end{align}
    Let $(\theta_0, \rho_0)$ be an admissible initial datum with $\bar{C}L\sE_\mu(\theta_0, \rho_0)<\mu$. Then
    the density $\rho$ of the solution $(\theta, \rho)$ to \eqref{eq:flow equation} converges exponentially fast to $\rho_\infty\equiv \nu$ in $C^2([0,L])$ as $t\to\infty$. Moreover,
    \begin{enumerate}[(i)]
        \item if $\omega\neq 0$, then $\theta(t) \to \theta_\infty$ in $C^2([0,L])$, where $\theta_\infty$ describes a $\omega$-fold covered circle;
        \item if $\omega=0$ and $\beta$ is analytic, then $\theta(t) \to \theta_\infty$ in $C^2([0,L])$, where $\theta_\infty$ describes a multifold covered figure eight elastica.
    \end{enumerate}  
\end{thm}

Assumption \eqref{eq:beta'bound} is clearly satisfied if we have $\beta'(\nu)=0$, $||\beta''||_\infty<\infty$, and $\inf_\R\beta>0$.
If $\beta$ is such that $\beta^\prime(x)(x-\nu)\leq0$ for $x\in\R$, we may choose $\bar C=0$ in \eqref{eq:beta'bound}. In this case, there is no assumption on the initial energy. We highlight that, remarkably, analyticity is not needed for the case $\omega\neq 0$, see the discussion after the proof of \Cref{thm:conv special beta neu} in \Cref{sec:convassbeta}.

For rotationally symmetric initial data ($\omega=1$) and $c_0=\frac{2\pi}{L}$ we prove exponential convergence to a circle with constant density without further assumptions, see \Cref{prop:conv_symm_c02Pi/L}.
Furthermore, we can determine the limit for large values of $\mu$ and constant initial density.

\begin{thm}[Asymptotic behavior for large $\mu$]
\label{prop:convlargemu}
Let $\omega\neq0$ and suppose that $\beta$ is real analytic. Let $(\theta_0,\rho_0)$ be an admissible initial datum with $\rho_0\equiv \nu$. 
There exists $\mu_0\in (0,\infty)$ such that if $\mu\geq \mu_0$, 
then the limit $(\theta_\infty,\rho_\infty)$ of Theorem \ref{thm:zsf artcl1} describes an $\omega$-fold covered circle with constant density.
\end{thm}

A major difficulty here is that the two parts of the energy do not decrease individually, so that the density does not remain constant, in general.

In Section \ref{sec:numerics}, we propose a simple numerical scheme approximating solutions to \eqref{eq:flow equation} which takes advantage of the gradient flow structure of the problem and is based on De Giorgi's Minimizing Movements and finite differences, extending the idea of \cite{BJSS2020} to the more general and time-dependent system. Our analysis is substantially guided by the resulting computations, especially concerning the long-term behavior of the system. Numerical experiments allow us to explore qualitative properties of the flow beyond the scope covered in the previous sections.
We consider a number of examples, some of which are:
the loss of embeddedness in two cases, first in the case of some $c_0 > 2\pi / L$,
and second in the more subtle case of $c_0 = 0$, with a careful choice of the initial data;
the convergence to nontrivial states for $\mu$ small, which hints at the existence of nontrivial
critical points, even for rotationally symmetric initial data;
the impact of the choice of parameters on the limit configuration $(\theta_\infty,\rho_\infty)$;
and finally the instability of multiple coverings of the figure eight, in the case $\omega = 0$.

For a related problem on surfaces, we would like to point out that two finite element methods are proposed and analyzed in \cite{elliott_numerical_2022}, based on a more geometric approach.

\section{The static problem}\label{sec:minprob}

\subsection{The Euler--Lagrange equations}\label{subsec:abstract_framework}

For $k\in \NN$, we denote the Hilbert space of periodic Sobolev functions on $[0,L]$ by
\begin{align}
    W^{k,2}_{\mathrm{per}}(0,L) := \{u \in W^{k,2}(0,L) : \partial_s^\ell u (L)= \partial_s^\ell u(0) \text{ for }\ell=0,\dots,k-1\}.
\end{align}
The standard angle function of an $\omega$-fold covering of the circle 
is given by
\begin{align}\label{eq:defphi}
    \phi(s)=\frac{2\pi\omega s}{L}\quad \text{ for } s\in[0,L],
\end{align}
see \cite[Section 3.1]{DLR2022}. 
If $\theta \in W^{k,2}(0,L)$ 
is an angle function for a $C^{k}$-closed curve with rotation index $\omega\in \ZZ$, 
there exists a unique $u\in W^{k,2}_{\mathrm{per}}(0,L)$ such that
$
    \theta = \phi + u.    
$
Consequently, it is natural to study the energy $\sE_\mu$ defined in \eqref{Eallg} on the set
\begin{align}
    \sU &:= W^{1,2}_{\mathrm{per}}(0,L;\RR^2) + (\phi,0) 
    = \left\lbrace(\theta,\rho)\in W^{1,2}(0,L;\RR^2):\;\theta(L)-\theta(0)=2\pi\omega,\; \rho(L)=\rho(0)\right\rbrace.
\end{align}
By introducing the constraint functional
\begin{align}
\label{eq:defG}
    \mathcal{G}(\theta, \rho) = \Big( \int_0^L \cos\theta\intd s,\; \int_0^L \sin \theta\intd s,\; \int_0^L \rho \intd s - \nu L\Big),
\end{align}
we see that $(\theta, \rho)\in \sU$ corresponds to a $C^1$-closed curve with rotation index $\omega\in \ZZ$ and a density with total mass $\nu L$ if and only if $\sG(\theta, \rho)=0$. 
We thus define the side condition
$\sA:=\left\lbrace (\theta,\rho)\in \sU:\;\mathcal{G}(\theta,\rho)=0\right\rbrace$.
The set $\sU$ is not a vector space (unless $\omega=0$), but only an affine subspace of the Hilbert space $W^{1,2}(0,L;\R^2$). However, this causes only some minor technical difficulties 
which can be resolved by working in the periodic setting with the shifted functionals
\begin{align}
    &E_\mu: W^{1,2}_{\mathrm{per}}(0,L;\RR^2)\to\RR,\; E_{\mu}(u,\rho)=\sE_\mu(u+\phi,\rho),\\
    &G:W^{1,2}_{\mathrm{per}}(0,L;\RR^2)\to\RR^3,\; G(u,\rho)=\sG(u+\phi,\rho), \label{eq:shifted}
\end{align}
with
$\phi$ as in \eqref{eq:defphi}. 

It can be checked that $\sG^\prime(\theta,\rho)\colon W^{1,2}_{\textup{per}}(0,L;\RR^2)\to\RR^3$, for $(\theta,\rho)\in\sA$, is surjective. Hence, applying \cite[Proposition 43.21]{Z1985}, we see that 
the energy $\sE_\mu$ has a critical point subject to the constraint $\sG=0$ at some $(\theta,\rho)\in \sA$ if and only if there exist $\lambda_{\theta1},\lambda_{\theta2},\lambda_\rho\in\RR$ such that 
\begin{align}
    &0=\int_0^L\big(\beta(\rho)(\partial_s\theta-c_0)\partial_s v -\lambda_{\theta1}\sin\theta\, v +\lambda_{\theta2}\cos\theta\, v \big)\intd s,\label{eq:EL1}\\
    &0=\int_0^L\big(\mu\partial_s\rho\partial_s\sigma +\frac12\beta^\prime(\rho)(\partial_s\theta-c_0)^2\sigma +\lambda_\rho\sigma \big)\intd s\label{eq:EL2}
\end{align}
 for all $( v ,\sigma )\in W^{1,2}_{\mathrm{per}}(0,L;\RR^2)$. 
Choosing appropriate test functions shows that 
$\lambda_{\theta1}$, $\lambda_{\theta2}$ and $\lambda_\rho$ are given as in \eqref{eq:lambdatheta}, and \eqref{eq:lambdarho}.
If the constrained critical point $(\theta,\rho)$ 
is more regular, precisely $(\theta,\rho)\in  W^{2,2}_{\mathrm{per}}(0,L;\RR^2)+(\phi,0)
$, 
\eqref{eq:EL1} and \eqref{eq:EL2} yield the \textit{Euler--Lagrange equations}
\begin{align}
    0&=\partial_s\left(\beta(\rho)(\partial_s\theta-c_0)\right)+\lambda_{\theta1}\sin\theta-\lambda_{\theta2}\cos\theta,\label{eq:ELp1}\\
    0&=\mu\partial_s^2\rho-\frac12\beta'(\rho)(\partial_s\theta-c_0)^2-\lambda_\rho\label{eq:ELp2}.
\end{align}

\subsection{Existence of minimizers and smoothness of critical points}\label{sec:exmin}

The existence of a solution to the minimization problem
\begin{align}\label{eq:minprob}
    \inf_{(\theta,\rho)\in\sA}\sE_\mu(\theta,\rho)\longrightarrow \min!
\end{align}
can be shown via the direct method 
following the arguments in \cite[Prop.\ 3.2]{BJSS2020}, also in the case of
general winding number $\omega\in \ZZ$, and with spontaneous curvature $c_0\in\RR$.

\begin{prop}[Existence of a minimizer]
\label{prop:minprobex}
There exists 
$(\theta^\ast,\rho^\ast)\in\sA$ such that 
$
\sE_\mu(\theta^\ast,\rho^\ast)=\inf_{(\theta,\rho)\in\sA}\sE_\mu(\theta,\rho).
$
\end{prop}

We prove now that constrained critical points (and in particular minimizers) are smooth.

\begin{prop}[Smoothness of critical points]
\label{prop:minprobex2} 
If $(\theta,\rho)$ is a constrained critical point, 
the $L$-periodic extension of $(\partial_s\theta,\partial_s\rho)$ to $\RR$ is smooth. 
In particular,
$(\theta,\rho)\in C^\infty([0,L])$ and $\partial_s\theta(L)=\partial_s\theta(0)$, $\partial_s\rho(L)=\partial_s\rho(0)$.
\end{prop}

\begin{proof}
Let $(\theta,\rho)\in \sU$
be a constrained critical point. 
Then there exist $\lambda_{\theta1},\lambda_{\theta2},\lambda_\rho\in\RR$
such that $(\theta,\rho)$ satisfies 
\eqref{eq:EL1} and \eqref{eq:EL2}. \\
\textit{Step 1: $(\theta,\rho)\in W^{2,2}(0,L)\subset C^1([0,L])$.}
Since $\rho\in W^{1,2}(0,L)\subset C([0,L])$, 
$\left\Vert\beta^\prime\circ\rho\right\Vert_{C([0,L])}$ and $\left\Vert\partial_s\rho\right\Vert_{L^2(0,L)}$ are bounded. Thus, \eqref{eq:EL1} implies that there is $C=C(\theta,\rho)$ such that
\begin{align}
    &\bigg\vert\int_0^L\beta(\rho)\partial_s\theta\partial_s v \intd s\bigg\vert
    =\bigg\vert\int_0^L(\lambda_{\theta1}\sin\theta-\lambda_{\theta2}\cos\theta) v \intd s-\int_0^L\beta^\prime(\rho)\partial_s\rho c_0 v \intd s\bigg\vert
    \leq C \left\Vert v \right\Vert_{L^2(0,L)}
\end{align}
for all $ v \in C_c^1((0,L))\subset W^{1,2}_{\mathrm{per}}(0,L)$.
So, $\beta(\rho)\partial_s\theta\in W^{1,2}(0,L)$.
Hence, $\left\Vert\beta(\rho)(\partial_s\theta-c_0)\right\Vert_{C([0,L])}$ and also $\left\Vert\beta'(\rho)(\partial_s\theta-c_0)\right\Vert_{C([0,L])}$ are bounded, so \eqref{eq:EL2} implies that
\begin{align}
    &2\mu\bigg\vert\int_0^L\partial_s\rho\partial_s\sigma \intd s\bigg\vert
    =\bigg\vert\int_0^L\left(-\beta^\prime(\rho)(\partial_s\theta-c_0)^2\sigma -2\lambda_\rho\sigma \right)\intd s\bigg\vert
    \leq C \left\Vert\sigma \right\Vert_{L^2(0,L)}
\end{align}
for all $\sigma \in C_c^1((0,L))$. 
It follows that $\rho\in W^{2,2}(0,L)\subset C^1([0,L])$. We thus obtain 
$\partial_s (\beta(\rho) \partial_s \theta) = \beta'(\rho)\partial_s \rho\partial_s \theta + \beta(\rho)\partial_s^2\theta$ in the sense of distributions, so $\partial_s^2\theta\in L^2(0,L)$ as $\inf_{[0,L]}\beta(\rho)>0$.\\
\textit{Step 2: $(\theta,\rho)\in W^{3,2}(0,L)\subset C^2([0,L])$.} The increased regularity
yields that $(\theta,\rho)$ satisfies 
\eqref{eq:ELp1} and \eqref{eq:ELp2}
in $L^2(0,L)$.
Using
the same ideas as in Step 1, we can deduce that $(\theta, \rho)\in W^{3,2}(0,L)$.
\\[0.1em]
\textit{Step 3:  
Smooth $L$-periodic extension.} 
Testing \eqref{eq:EL1} and \eqref{eq:EL2} with $W^{1,2}_{\mathrm{per}}$-functions 
not vanishing at the boundary results in the natural boundary conditions 
\begin{align}
    \partial_s\theta\big\vert_0^L=0 \quad\text{ and }\quad \partial_s\rho\big\vert_0^L=0.\label{eq: nat bc}
\end{align}
Since $(\theta,\rho)$ satisfies \eqref{eq:ELp1} and \eqref{eq:ELp2} 
pointwise we conclude with \eqref{eq: nat bc} that $\partial_s^2\theta(L)=\partial_s^2\theta(0)$ and $\partial_s^2\rho(L)=\partial_s^2\rho(0)$.
The claim follows by bootstrapping.
\end{proof}

\subsection{Homogeneous elastica}\label{sec:homelastica}

The structure of the energy functional $\sE_\mu$ suggests that for large values of $\mu$, minimizers favor almost constant density, cf.\ \cite{BJSS2020}. For constant density, $\sE_\mu$ is essentially the elastic energy whose critical points are called elasticae. 
For $\mu$ large, these elasticae also play an important role for the heterogeneous elastic energy \eqref{Eallg}. 

\begin{defn}
Let $(\theta,\rho)\in\sA$. We say that $\theta$ describes a \textit{(length-constrained) elastica} if the curvature $\kappa=\partial_s\theta$ is smooth and satisfies the \textit{constrained elastica equation}
\begin{align}
\label{eq:elasticae}
    \partial_s^2\kappa+\frac12\kappa^3-\lambda\kappa=0 \quad \text{ for some }\lambda\in\RR.
\end{align}
If further $\rho$ is constant, we say that $(\theta,\rho)$ describes a \textit{homogeneous elastica}.
\end{defn}

Solutions of \eqref{eq:elasticae} have been classified explicitely in several previous works, see for example \cite{LS1984}, \cite{DHMV2008}, or \cite[Lemma 5.4]{MR2021}. 
In the case of closed curves, the 
elasticae can be characterized as follows.

\begin{lem}[{\cite{LS1984}}]
\label{lem:charelasticae}
The only closed constrained elasticae are 
multifold coverings of circles and multifold coverings of the figure eight (elastica).
\end{lem}

The figure eight is illustrated in \Cref{fig:omega 0 figure 8} on page \pageref{fig:omega 0 figure 8}.  

\begin{rem}
\label{rem:kappaelastica}
    Let $(\theta,\rho)\in\sA$.
    If $\omega\neq0$ and $\theta$ describes an 
    elastica, then $\theta$ is the angle function of an $\omega$-fold covering of a circle with curvature $\kappa=\frac{2\pi\omega}{L}$. Thus, $\theta$ is determined up to an additive constant. 
    If $(\theta,\rho)$ describes a homogeneous elastica and 
    we require that $\theta(0)=0$ or that $\int_0^L\theta(s)\intd s=\pi\omega L$, then 
    \begin{align}
\label{eq:deftrivsol}
    (\theta,\rho)=(\theta_c, \rho_c):=\left(\frac{2\pi\omega}{L}s,\nu\right)\in \sA.
\end{align}
\end{rem} 

\subsubsection{Minimizers and critical points for large $\mu$}
\label{subsec:min for large mu}

We now show that for large values of $\mu$ and $\omega\neq0$, a minimizer $(\theta,\rho)$ in \eqref{eq:minprob} describes a homogeneous elastica. 
To state a uniqueness result, we define $\mathcal{I}(\theta):=\int_0^L\theta(s)\intd s$ and fix $\mathcal{I}(\theta)=\pi\omega L$.

\begin{prop}[Unique minimizer for large $\mu$]
\label{prop:minproblargemu}
Let $\omega\neq0$. 
Then there exists $\mu_0\in (0,\infty)$ such that if $\mu\geq \mu_0$, 
$\sE_\mu(\theta_c,\rho_c)<\sE_\mu(\theta,\rho)$ for all $(\theta,\rho)\in\sA\cap\left\lbrace\mathcal{I}(\theta)=\pi\omega L\right\rbrace\setminus \left\lbrace(\theta_c,\rho_c)\right\rbrace$.
\end{prop}

The proof can essentially be done as in \cite[Prop.\ 3.3]{BJSS2020} which is why we only outline the idea here.

\begin{proof}[Idea of the proof.]
In a first step, one shows that for $\mu$ large enough, the second variation $\sE''_\mu(\theta_c,\rho_c)$ is strictly positive. It follows (for example with \cite[Theorem 43.D]{Z1985}) that 
$(\theta_c,\rho_c)$ is a strict local minimizer in the sense that there exists $\delta=\delta(\mu)>0$ such that $\sE_\mu(\theta_c,\rho_c)<\sE_\mu(\theta,\rho)$ for all $(\theta,\rho)\in\sA\cap\left\lbrace\mathcal{I}(\theta)=\pi\omega L\right\rbrace\setminus \left\lbrace(\theta_c,\rho_c)\right\rbrace$ with norm $\left\Vert(\theta,\rho)-(\theta_c,\rho_c)\right\Vert_{W^{1,2}(0,L)}<\delta$. Now, note that $\sE_\mu(\theta_c,\rho_c)$ is independent of $\mu$ whereas $\sE_\mu(\theta,\rho)$ is increasing in $\mu$. Consequently, the neighborhood in which $(\theta_c,\rho_c)$ is a strict local minimizer can be chosen only depending on a lower bound $\mu_1$ on $\mu$.
In a second step, one proves that 
there exists $\mu_2$ such that for $\mu>\mu_2$, all global minimizers are contained in $B_\delta(\theta_c,\rho_c)\subset W^{1,2}(0,L)$.
Together with the first step, for $\mu>\max\left\lbrace\mu_1,\mu_2\right\rbrace$, $(\theta_c,\rho_c)$ is the unique global minimizer with $\mathcal{I}(\theta)=\pi\omega L$. 
\end{proof}

This result does not extend to $\omega=0$, see Remark \ref{rem:omega0}. 
The next example shows that 
$(\theta_c,\rho_c)$ 
does not always need to be a global minimizer.
\begin{ex} \label{ex:smaller E than trivial}
Consider the double-well potential $\beta(x):=(x^2-1)^2+c$ with $c>0$. Let $L=2\pi$, $\omega=1$, $c_0=0$ and $\nu=0$. Consider $\theta_c=s$, $\rho_c=0$ and $\rho:=\sin(2s)$.
Then 
$(\theta_c,\rho)\in\sA$ and
$
    \sE_\mu(\theta_c,\rho_c)=\left(1+c\right)\pi>\left(\frac38+c+\frac\mu2\right)\pi=\sE_\mu(\theta_c,\rho)
$
for $\mu<\frac54$.
\end{ex}

For large $\mu$, $(\theta_c,\rho_c)$ is also locally the unique constrained critical point. This follows from
the continuity of 
$\sE''_\mu$ and $\sG''$ and the positive definiteness of $\sE''_\mu(\theta_c,\rho_c)$. 

\begin{cor}
\label{cor:C1neigh}
Let $\omega\neq0$. There is $\mu_0>0$ and a $C^1$-neighborhood $\mathcal{O}$ of $(\theta_c,\rho_c)$ 
such that if $(\theta,\rho)\in\mathcal{O}$ is a constrained critical point 
with $\mathcal{I}(\theta)=\pi\omega L$ and $\mu>\mu_0$, then $(\theta,\rho)=(\theta_c,\rho_c)$.
\end{cor}

\subsubsection{Conditions for homogeneous elastica}
If a constrained critical point has constant density, this already implies that it describes a homogeneous elastica.

\begin{lem}
\label{lem:critpointconstrho}
If $(\theta,\rho)$ describes a constrained critical point 
and $\rho\equiv \nu$, 
then $\theta$ describes an 
elastica.
\end{lem}

\begin{proof}
By \Cref{prop:minprobex2}, 
$(\theta,\rho)$ 
is smooth. 
Since $\rho$ is constant,  
the Euler--Lagrange equation \eqref{eq:ELp1} reads
\begin{align}
\label{eq:elasticaeq1}
    0=\beta(\nu)\partial_s\kappa+\lambda_{\theta1}\sin\theta-\lambda_{\theta2}\cos\theta,
\end{align}
with $\kappa=\partial_s\theta$. As in \cite[Remark 2.2]{NP2020}, we multiply \eqref{eq:elasticaeq1} with $\kappa$. This yields
    $0=\frac12\beta(\nu)\partial_s(\kappa^2)+\partial_s\left(-\lambda_{\theta1}\cos\theta-\lambda_{\theta2}\sin\theta\right)$.
So there is $\tilde\lambda\in\RR$ such that 
\begin{align}
\label{eq:elasticaeq2}
    \tilde\lambda=\frac12\beta(\nu)\kappa^2-\lambda_{\theta1}\cos\theta-\lambda_{\theta2}\sin\theta.
\end{align}
Differentiation of \eqref{eq:elasticaeq1} and inserting \eqref{eq:elasticaeq2} leads to
\eqref{eq:elasticae} with $\lambda=\nicefrac{\tilde\lambda}{\beta(\nu)}$. 
\end{proof}

\begin{rem}
\label{rem:omega0}
In the case $\omega=0$, a constrained critical point 
with constant density exists only if $\beta'(\nu)=0$. 
Indeed, if $\rho\equiv \nu$ and $\beta^\prime(\nu)\neq0$, it follows from \eqref{eq:ELp2} that $(\partial_s\theta-c_0)^2$ is constant. For $\omega=0$, this contradicts the closedness of the curve described by $\theta$. 
\end{rem}

For $\omega\neq0$ and under suitable assumptions on $\beta$, 
the converse implication of Lemma \ref{lem:critpointconstrho} also holds.

\begin{lem}
\label{lem:critpointconstkappa}
    Let $\omega\neq0$ and $c_0\neq\frac{2\pi\omega}{L}$. 
    Let $(\theta,\rho)$ be a constrained critical point 
    and suppose $\theta$ describes an 
    elastica.
    If $\beta$ is such that 
    \begin{enumerate}[(a)]
        \item $\beta$ is convex or
        \item $\displaystyle\left\vert\beta'(x)\right\vert<2\mu\,\Big(\frac{2\pi}{2\pi\omega-Lc_0}\Big)^2\left\vert \nu-x\right\vert$ or
        \item $\displaystyle\sup\left\vert\beta''\right\vert<2\mu\,\Big(\frac{2\pi}{2\pi\omega-Lc_0}\Big)^2$,
    \end{enumerate}
    then $\rho$ is constant. In particular, $(\theta,\rho)$ describes a homogeneous elastica.
\end{lem}

\begin{proof}
 Since $\theta$ describes a constrained closed elastica and $\omega\neq0$, we have 
 $\kappa= \partial_s\theta\equiv \frac{2\pi\omega}{L}$ by Remark \ref{rem:kappaelastica}. The Euler--Lagrange equation for $\rho$ (cf.\ \eqref{eq:ELp2}) simplifies to 
 \begin{align}
 \label{eq:intrhocrit}
     \partial_s^2\rho=\frac{(\kappa-c_0)^2}{2\mu}\bigg(\beta'(\rho)-\frac1L\int_0^L\beta'(\rho)\intd s\bigg).
 \end{align}
 
\textit{(a)} 
Using integration by parts, \eqref{eq:intrhocrit}, $\int_0^L\rho\intd s=\nu L$, and the convexity of $\beta$, we have
\begin{align}
     \int_0^L(\partial_s\rho)^2\intd s 
     =-\frac{(\kappa-c_0)^2}{2\mu}\int_0^L\beta'(\rho)(\rho-\nu)\intd s
     =-\frac{(\kappa-c_0)^2}{2\mu}\int_0^L\left(\beta'(\rho)-\beta'(\nu)\right)(\rho-\nu)\intd s\leq 0.
\end{align}

\textit{(b)} First, we proceed as in (a), then we obtain with $(\kappa-c_0)^2=\left(\frac{2\pi}{L}\right)^2\left(\frac{2\pi\omega-Lc_0}{2\pi}\right)^{2}$, the assumption on $\beta'$, and the Wirtinger inequality 
that
\begin{align}
     \int_0^L(\partial_s\rho)^2\intd s
     =-\frac{(\kappa-c_0)^2}{2\mu}\int_0^L\beta'(\rho)(\rho-\nu)\intd s 
     <\left(\frac{2\pi}{L}\right)^2\int_0^L(\rho-\nu)^2\intd s
     \leq \int_0^L(\partial_s\rho)^2\intd s.
\end{align}
 
 \textit{(c)} We write $\overline{\beta'(\rho)}=\frac1L\int_0^L\beta'(\rho)\intd s$ and use the Wirtinger inequality twice to get
 \begin{align}
     \int_0^L(\partial_s^2\rho)^2\intd s 
     &=\left(\frac{(\kappa-c_0)^2}{2\mu}\right)^2\int_0^L\left(\beta'(\rho)-\overline{\beta'(\rho)}\right)^2\intd s 
     \leq \frac{(\kappa-c_0)^4}{(2\mu)^2}\left(\frac{L}{2\pi}\right)^2\int_0^L\left(\beta''(\rho)\partial_s\rho\right)^2\intd s \\
     &< \left(\frac{2\pi}{L}\right)^2\int_0^L(\partial_s\rho)^2\intd s
     \leq \int_0^L(\partial_s^2\rho)^2\intd s. 
 \end{align}
In all cases, the periodic boundary conditions imply that $\rho\equiv \nu$.
\end{proof}
\begin{cor}[of \Cref{thm:conv special beta neu}]
    Let $(\theta,\rho)$ be a constrained critical point. 
    If $\beta$ is such that $\beta'(x)\leq0$ for $x<\nu$ and $\beta'(x)\geq0$ otherwise, then $(\theta,\rho)$ describes a homogeneous elastica.
\end{cor}

Without additional assumptions on $\beta$, the density of a constrained critical point describing an 
elastica might be nonconstant.

\begin{ex}
\label{ex:non const rho crit point}
    Let $L=2\pi$, $\omega=2$, $\nu=0$ and let $\mu>0$ and $c_0\neq2$ be chosen such that $(2-c_0)^2=2\mu$. Let $\beta(x)=-\frac{x^2}{2}+1$ for $-1\leq x\leq 1$. 
    Then both $(\theta= \theta_c, \rho\equiv \nu)\in \sA$ and $(\theta=\theta_c, \rho=\sin s)\in \sA$ are constrained critical points.
    In this case, inequalities (b) and (c) in Lemma \ref{lem:critpointconstkappa} are attained with equality, so the assumptions are sharp.
\end{ex}

\begin{rem}
    For $\omega\neq0$, $c_0=\frac{2\pi\omega}{L}$ and $(\theta,\rho)$ a constrained critical point 
    with $\theta$ describing an 
    elastica, it follows directly from \eqref{eq:intrhocrit} and the periodic boundary conditions that $\rho$ is constant.
\end{rem}

\section{Qualitative properties of solutions}\label{sec:qualprop}

\subsection{Decrease of the energy}
\label{sec:decreaseE}

The $L^2$-gradient structure of the flow equations in \eqref{eq:flow equation} ensures that the energy $\sE_\mu$ decreases along the evolution. 
On the other hand, 
the two parts of the energy, 
\begin{align}
    \sE^\theta(\theta,\rho):=\frac12\int_0^L\beta(\rho)(\partial_s\theta-c_0)^2\intd s \quad\text{ and }\quad \sE_\mu^\rho(\rho):=\frac{\mu}{2}\int_0^L(\partial_s\rho)^2\intd s
\end{align}
are not monotonically decreasing individually as the computation 
\begin{align}
    \frac{\intd }{\intd t}\sE_\mu(\theta, \rho) 
    &=\int_0^L\left(\frac12\beta'(\rho)(\partial_s\theta-c_0)^2\partial_t\rho+\beta(\rho)(\partial_s\theta-c_0)\partial_t\partial_s\theta\right)\intd s
    +\int_0^L\mu\,\partial_s\rho\,\partial_s\partial_t\rho\intd s\\
    &=-\int_0^L\partial_s\left(\beta(\rho)(\partial_s\theta-c_0)\right)\partial_t\theta\intd s+\int_0^L\left(\frac12\beta^\prime(\rho)(\partial_s\theta-c_0)^2-\mu\partial_s^2\rho\right)\partial_t\rho\intd s\\
    &=-\int_0^L\partial_t\theta\left(\partial_t\theta-\lambda_{\theta1}\sin\theta+\lambda_{\theta2}\cos\theta\right)\intd s - \int_0^L\partial_t\rho\left(\partial_t\rho+\lambda_\rho\right)\intd s\\
    &=-\int_0^L(\partial_t\theta)^2\intd s-\int_0^L(\partial_t\rho)^2\intd s
    \leq0 \label{eq:decreaseE}
\end{align}
already suggests.
This fact significantly complicates the discussion of the limit in Section \ref{sec:proplim}. 
We give concrete examples where either $\sE^\theta$ or $\sE_\mu^\rho$ grows. 
\begin{ex}
\label{ex:E2grows}
Consider the double-well potential $\beta(x):=(x^2-1)^2+c$ with 
$c>0$. 
Let $L=2\pi$, $\omega=1$, $\nu=0$, and $c_0\neq1$. 
Take
$
    \theta_0:=\theta_c=
    s$ and $
    \rho_0:=
    \sin s$.
Then, for the solution $(\theta,\rho)$ of \eqref{eq:flow equation}, an elementary computation yields $\lambda_{\theta1}(0)=\lambda_{\theta2}(0)=\lambda_\rho(0)=0$ and
\begin{align}
    \frac{\intd}{\intd t}\sE^\theta(\theta,\rho)\Big\vert_{t=0}
    &=(1-c_0)^2\int_0^{2\pi}\left(\frac{\mu}{2}\,\beta'(\rho_0)\,\partial_s^2\rho_0-\big(\partial_s(\beta(\rho_0))\big)^2\right)\intd s-\frac{(1-c_0)^4}{4}\int_0^{2\pi}(\beta'(\rho_0))^2\intd s
     \\
     &=\frac{\pi}{2}(1-c_0)^2\left(\mu-(1-c_0)^2-\frac52\right),
\end{align}
which is positive for $\mu>(1-c_0)^2+\frac52$. On the other hand, 
\begin{align}
    \frac{\intd}{\intd t}\sE_\mu^\rho(\rho)\Big\vert_{t=0}
    &=(1-c_0)^2\,\frac\mu2\int_0^{2\pi}\beta'(\rho_0)\,\partial_s^2\rho_0\intd s-\mu^2\int_0^{2\pi}(\partial_s^2\rho_0)^2\intd s=\mu\pi\left(\frac12(1-c_0)^2-\mu\right)
\end{align}
is positive for $\mu<\frac12(1-c_0)^2$. For numerical examples, see Figures \ref{fig:energy high c0}, \ref{fig:energy low curvature} and \ref{fig:energy high curvature}.
\end{ex} 

\begin{rem}\label{rem:Erho_grow}
In view of \Cref{prop:minproblargemu} and \Cref{cor:C1neigh}, for $\mu$ large, $\sE_\mu^\rho$ can be seen as the dominant term in $\sE_\mu$. However, even for arbitrary large $\mu$, $\sE_\mu^\rho$ might still not be monotonically decreasing.
This can be seen with the evolution equations in \eqref{eq:flow equation}. Indeed, taking 
$\rho_0\equiv \nu$ yields $\sE^\rho_\mu(\rho_0)=0$. On the other hand, $\partial_t\rho\vert_{t=0}\neq0$ for all $\mu>0$ as long as $\theta$ is not constant and $\beta^\prime(\nu)\neq0$. 
\end{rem}

\subsection{Zeros of the curvature}
\label{sec:zeroset}

Differentiating \eqref{eq:flow equation} we find that the curvature $\kappa=\partial_s\theta$ satisfies
\begin{align}
    \partial_t\kappa
    &=\beta(\rho)\partial_s^2\kappa+2\,\partial_s(\beta(\rho))\partial_s\kappa+\partial_s^2(\beta(\rho))\kappa+\left(\lambda_{\theta1}\cos\theta+\lambda_{\theta2}\sin\theta\right)\kappa-\partial_s^2(\beta(\rho))c_0.\label{eq:ev eq kappa}
\end{align}
The structure of this evolution equation already indicates that the behavior of $\kappa$ strongly depends on $c_0$. In case $c_0=0$, \eqref{eq:ev eq kappa} may be written as a linear second order parabolic equation for $\kappa$. Indeed, we have
\begin{align}
\label{eq:dtkappa}
    \partial_t\kappa=a\partial_s^2\kappa+b\partial_s\kappa+c\kappa,
\end{align}
where we define the nonconstant coefficients  $a:=\beta(\rho)>0$, 
    $b:=2\partial_s(\beta(\rho))$ and 
    $c:=
    \partial_s^2(\beta(\rho))
    +\lambda_{\theta1}\cos\theta+\lambda_{\theta2}\sin\theta$.
This allows us to use the techniques in \cite{A1988}
to study the evolution of 
the zeroset of the curvature. 
Sign-changing zeros of the curvature are \textit{inflection points} of the curve, i.e.\ 
points where the curve locally changes from being convex to being concave or vice versa.  Zeros at which the curvature does not change sign are called \textit{undulation points} of the curve.
First, we do not distinguish between inflection points and undulation points and show that the total number of zeros of $\kappa$ decreases.
In \cite[Remark 3.2]{W1993}, this idea was also indicated, without proof, for the $L^2$-gradient flow of the angle function of the classical elastic energy without density-modulated stiffness.

For all $t\geq 0$, we denote by $z_\kappa(t)\in\NN_0\cup\lbrace\infty\rbrace$  the number of zeros  in $[0,L)$ of the curvature $\kappa(t,\cdot)$ of the global solution of \eqref{eq:flow equation}. First we note the following. 

\begin{lem}[{\cite[Theorem C, (a)]{A1988}}]
\label{lem:zeroset discrete}
Let $c_0=0$ and $(\theta,\rho)$ be the global solution of \eqref{eq:flow equation}. Then 
$z_\kappa(t)$ 
is finite for any $t>0$.
\end{lem}

Further, 
the number of zeros of the curvature (inflection points and undulation points) does not increase along the evolution.

\begin{prop}
\label{prop:zeroset non-incr}
Let $c_0=0$ and $(\theta,\rho)$ be the global solution of \eqref{eq:flow equation}. Then $z_\kappa$  
is a nonincreasing function on $[0,\infty)$.
\end{prop}

\begin{proof}
Consider the smooth $L$-periodic extension of $\kappa=\partial_s\theta$ to $\RR$, which we do not rename for simplicity. The function $z_\kappa$ still denotes the number of zeros of the curvature in the interval $[0,L)$. 

Let $t_2>0$. By Lemma \ref{lem:zeroset discrete}, there exists $n\in\NN_0$ such that 
$z_\kappa(t_2)=n$. Without loss of generality, we assume that $n\geq1$. Let $0\leq s_1<s_2< \dots <s_n<L$ be the zeros of $\kappa$ at time $t_2$.  
By \cite[Lemma 5.5]{A1988}, there are continuous curves $x_i(t)$, $i=1,\dots,n$ in the zeroset $Z_\kappa=\left\lbrace(t,s)\in(0,\infty)\times\RR: \kappa(t,s)=0\right\rbrace$ defined for $t\in(0,t_2]$ such that $x_i(t_2)=s_i$.
Moreover, define $x_{n+1}(t):=x_1(t)+L$, $t\in(0,t_2]$. Then $x_{n+1}(t)$ is also in $Z_\kappa$.

For $t_1\in(0,\infty)$ with $t_1 < t_2$, \cite[Lemma 5.3]{A1988} tells us that if $x_i(t_2)<x_j(t_2)$, then $x_i(t)<x_j(t)$ for all $t\in[t_1,t_2]$, $i,j=1,\dots n+1$.
Hence, considering $I_t:=[x_1(t),x_{n+1}(t))$ for $t\in[t_1,t_2]$, 
we find $x_i(t)\in I_t$ for $t\in[t_1,t_2]$ and $i=1,...,n$. Thus, 
there are at least 
$n$ zeros of $\kappa$ on $\lbrace t_1\rbrace\times I_t$.
Since by periodicity, the number of zeros on $\lbrace t_1\rbrace\times [0,L)$ equals the number of zeros on $\lbrace t_1\rbrace\times I_t$, it follows that $z_\kappa$ is nonincreasing on $(0,\infty)$.

It remains to consider the transition from $t=0$ to positive times. 
Due to \Cref{lem:zeroset discrete} we assume without loss of generality that $z_\kappa(0)<\infty$. Similarly as in \cite[Lemma 5.2]{A1988} it follows that $\lim_{t\searrow 0}x_i(t)$ exists for $i=1,\dots n$. Suppose that $\lim_{t\searrow 0}x_i(t)=\lim_{t\searrow 0}x_j(t)$ for $i<j$ and consider the nonempty open set $G:=\lbrace(t,s):0<t<t_2,\, x_i(t)<s<x_j(t)\rbrace$. Since $\kappa\equiv0$ on the parabolic boundary of $G$, the parabolic maximum principle implies that $\kappa\equiv0$ in $G$. This contradicts \Cref{lem:zeroset discrete}. Thus, $\lim_{t\searrow 0}x_i(t)\neq\lim_{t\searrow 0}x_j(t)$ for $i\neq j$ and it follows that $z_\kappa(0)\geq n$.
\end{proof}

In the following, we specifically consider the inflection points and show that the number of sign-changing zeros of $\kappa$ does not increase. This means geometrically that the number of `dents' of a curve like in Figure \ref{fig:dents} is not increasing along the evolution. This is supported by numerical experiments, see \Cref{fig:energy high curvature}, while for $c_0\neq 0$, \Cref{fig:energy high c0} gives an example for growing number of inflection points.\\
We denote the number of sign-changing zeros in $[0,L)$ of the $L$-periodic extension of $\kappa(t,\cdot)$ by $\bar z_\kappa(t)\in\NN_0\cup\lbrace\infty\rbrace$, $t\in[0,\infty)$.

\begin{figure}
\begin{center}
\begin{tikzpicture}

    \newcommand*\radius{1.2}
    \newcommand*\amplitude{.2}
    \newcommand*\schnittpunkte{14}

    \pgfmathsetmacro\winkel{360/\schnittpunkte}
    \pgfmathsetmacro\abschnitte{\schnittpunkte/2}

    \draw[black,very thick]
      (0:\radius-\amplitude)
      foreach \i in {1,...,\abschnitte}{
        to[out={2*\winkel*\i+90-2*\winkel},in={2*\winkel*\i-90-\winkel}]
        ({2*\winkel*\i-\winkel}:{\radius+\amplitude})
        to[out={2*\winkel*\i+90-\winkel},in={2*\winkel*\i-90}]
        ({2*\winkel*\i}:{\radius-\amplitude})
      };
\end{tikzpicture}
\end{center}
\caption{Curve with dents.}
\label{fig:dents}
\end{figure}
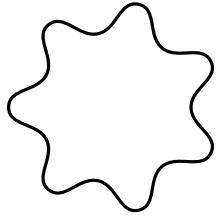

\begin{prop}
    \label{prop:inflpoints}
    Let $c_0=0$ and $(\theta,\rho)$ be the global solution of \eqref{eq:flow equation}. Then $\bar z_\kappa$ is a nonincreasing function on $[0,\infty)$.
\end{prop}

\begin{proof}
    Consider the smooth $L$-periodic extension of $(\kappa,\rho)$ to $\RR$.
    First we observe, that the zeros of $\tilde\kappa(t,s):=e^{-\lambda t}\kappa(t,s)$, $\lambda\in\RR$, $(t,s)\in[0,\infty)\times\RR$, coincide with the zeros of $\kappa(t,\cdot)$. By choosing $\lambda$ sufficiently large, we can thus assume that $\kappa$ satisfies \eqref{eq:dtkappa} with $c(t,s)<0$ on $[0,\infty)\times\RR$.\\
    Let $0\leq t_1< t_2<\infty$. By \Cref{lem:zeroset discrete}, there is $n\in\NN_0$ such that $z_\kappa(t_2)=n$. In view of \Cref{prop:zeroset non-incr} and due to periodicity of $\kappa$, we may assume without loss of generality that $n>1$. As in the proof of \Cref{prop:zeroset non-incr}, let $0\leq s_1<\dots<s_n<L$ be such that $\kappa(t_2,s_i)=0$, $i=1,\dots,n$. By the arguments in the 
    proof of \Cref{prop:zeroset non-incr}, there exist continuous curves $x_i(t)$ in the zeroset of $\kappa$ defined for $t\in[0,t_2]$ such that $x_i(t_2)=s_i$, $i=1,\dots,n$, and $x_i(t)<x_{i+1}(t)$ for $t\in[0,t_2]$, where $x_{n+1}:=x_1+L$.\\
    Let $i\in\lbrace1,\dots,n\rbrace$. Either $\kappa(t_2,\cdot)<0$ or $\kappa(t_2,\cdot)>0$ on $(x_i(t_2),x_{i+1}(t_2))$. We assume the latter. 
    Consider the open set $G:=\lbrace(t,s):t_1<t<t_2,\, x_i(t)<s<x_{i+1}(t)\rbrace$, whose boundary is composed of the curves $x_i, x_{i+1}$ and two vertical lines at $t=t_1, t_2$.
    By assumption and continuity, $\kappa$ attains a positive maximum on $\bar{G}$, more precisely at the parabolic boundary by the maximum principle. But since $\kappa(t,x_i(t))=0$ for $t\in[t_1,t_2]$ and similarly for $x_{i+1}$, the maximum is attained at the vertical line $t=t_1$. Consequently, there exists an interval $I\subset (x_i(t_1),x_{i+1}(t_1))$ with $\kappa(t_1,s)>0$ for $s\in I$.
    Thus, for each of the disjoint intervals $(x_i(t_2),x_{i+1}(t_2))$, $i=1,\dots,n$, with $\kappa>0$ (or $\kappa<0$), there is an interval in $(x_i(t_1),x_{i+1}(t_1))$ with $\kappa>0$ (or $\kappa<0$). Since $x_i(t)<x_{i+1}(t)$, $t\in[t_1,t_2]$, there are at least as many sign changes of $\kappa$ in $[x_1(t_1),x_1(t_1)+L)$ as in $[x_1(t_2),x_1(t_2)+L)$. By periodicity, it follows that $\bar z_\kappa(t_1)\geq \bar z_\kappa(t_2)$.    
\end{proof}

\begin{proof}[Proof of \Cref{thm:kappa_zeros}]
    The result directly follows from \Cref{prop:zeroset non-incr} and \Cref{prop:inflpoints}.
\end{proof}


\subsection{Convexity}
\label{sec:presconvex}

In the following, we will use \eqref{eq:dtkappa} to examine whether nonnegativity (or nonpositivity) of the curvature 
of the initial curve is preserved along the evolution. This is closely related to convexity of the associated curve, see Remark \ref{rem:convex in general} below. 

\begin{proof}[Proof of Theorem \ref{prop:convex}]
Case (ii) follows from \Cref{prop:zeroset non-incr}.
For case (i), we now show that if $\partial_s\theta_0\geq0$, then $\partial_s\theta\geq0$ for all $t\in(0,\infty)$. The proof for the nonpositive case works analogously.
We consider the $L$-periodic extension of the global solution $(\theta,\rho)$ to all of $\RR$, which we do not rename for simplicity.
Note that $\partial_s\theta\in C^\infty((0,\infty)\times\RR)$ and $\rho\in C^\infty((0,\infty)\times\RR)$ by \Cref{thm:zsf artcl1}. 
Since $c_0=0$, $\kappa$ satisfies \eqref{eq:dtkappa} on $(0,\infty)\times\RR$.
Even though $\theta$ itself is not even continuous on $\RR$, due to the boundary condition $\theta(L)-\theta(0)=2\pi\omega$, the coefficients $a$, $b$ and $c$ in \eqref{eq:dtkappa} are smooth on $\RR$ for all $t\geq0$. 
Moreover, since the initial datum is attained in the $C^2([0,L])$-norm 
(see Theorem \ref{thm:zsf artcl1}) and $c$ is bounded globally in 
$(0,\infty)\times\RR$ (see \cite[Section 4]{DLR2022}),
there is $K>0$ such that 
$
    \sup_{{t\in[0,\infty),\,s\in\RR}}c(t,s)<K.
$
Defining $\kappa_\varepsilon:=\kappa+\varepsilon \exp(Kt)$ for $\varepsilon>0$ it follows that 
\begin{align}
\label{eq:dtkappa2}
    \partial_t\kappa_\varepsilon>a(t,s)\partial_s^2\kappa_\varepsilon+b(t,s)\partial_s\kappa_\varepsilon+c(t,s)\kappa_\varepsilon\quad \text{ on }(0,\infty)\times\RR
\end{align}
and $\min_{s\in\RR}\kappa_\varepsilon(0,s)=\min_{s\in[0,L]}\partial_s\theta_0(s)+\varepsilon>0$. 
Now we set
\begin{align}
\label{eq:tstar2}
    t^{\ast}:=\sup\big\lbrace t\geq0\colon \min_{s\in\RR}\kappa_{\varepsilon}(\tau,s)> 0 \text{ for all } \tau\in(0,t)\big\rbrace.
\end{align}
Clearly, $t^\ast>0$. 
We assume that $t^\ast<\infty$. This implies that there is $s^\ast\in\RR$ such that 
\begin{align}
\label{eq:keps(tstar,sstar)}
0=\min_{s\in\RR}\kappa_{\varepsilon}(t^\ast,s)=\kappa_{\varepsilon}(t^\ast,s^\ast).
\end{align}
The necessary conditions for a local minimum yield $\partial_s^2\kappa_\varepsilon(t^\ast,s^\ast)\geq0$, $\partial_s\kappa_\varepsilon(t^\ast,s^\ast)=0$ and $\partial_t\kappa_\varepsilon(t^\ast,s^\ast)\leq0$. Together with \eqref{eq:keps(tstar,sstar)} this is a contradiction to \eqref{eq:dtkappa2}. Hence, 
$\kappa_\varepsilon>0$ on $(0,\infty)\times\RR$.
Sending $\varepsilon\to0$ yields $\partial_s\theta=\kappa\geq0$ on $(0,\infty)\times\RR$. 
\end{proof}

\begin{rem}
\label{rem:convex in general}
A planar closed curve $\gamma$ is called convex if it is \emph{simple} and its curvature $\kappa =\partial_s \theta$ does not change sign. This is equivalent to $\gamma$ parametrizing the boundary of a convex subset of $\R^2$. 
If the initial datum $\theta_0$ of the flow \eqref{eq:flow equation} describes a convex curve and if $c_0=0$, then the corresponding curve remains convex for all times. Indeed, by \Cref{prop:convex}, the sign of $\kappa_0=\partial_s \theta_0$ is preserved. Since the initial curve is simple 
and the winding number is preserved along the flow, we have $\omega=\pm1$ for all $t\geq0$ by Hopf's Umlaufsatz.  
Thus, $\int_0^L |\kappa(t,s)|\intd s = 2\pi$ for any $t\geq 0$. Fenchel's Theorem yields that the evolving curve remains convex for all $t\geq 0$, in particular it remains simple.  \\
For $c_0=0$, this is a distinctive advantage of the second order system \eqref{eq:flow equation} over the classical fourth order elastic flow, as simple curves generically can become nonsimple, see \cite{MMR2021}. 
\end{rem}

Remarkably, the statement of \Cref{prop:convex} can, in general, not be extended to the case $c_0\neq 0$. 
We give an example with $\left\vert c_0\right\vert>0$ arbitrarily small. 
For a numerical example, see \Cref{fig:energy high c0}.

\begin{ex}
\label{ex:counterconvex}
Let $c_0\neq0$. If $c_0>0$, let $\beta''>0$ on $[r,R]$ for some $r,R\in\RR$. Otherwise, let $\beta''<0$ on $[r,R]$.
We consider an initial curve consisting of a straight line smoothly connected to the rest of the curve in such a way that $\partial_s\theta_0\geq0$ on $[0,L]$. For example, a cigar-shaped curve as shown in \Cref{fig:cigarre:A}. 
Let $[a,b]\subset[0,L]$, $a<b$, such that
\begin{align}
    \kappa_0(s)=\partial_s\theta_0(s)=0,\quad \rho_0(s)\in(r,R), \quad \partial_s\rho_0(s)\neq0,\quad \partial_s^2\rho_0(s)=0 \quad\text{ for } s\in[a,b],
\end{align} 
cf.\ \Cref{fig:cigarre:B}. Then \eqref{eq:ev eq kappa} gives 
    $\partial_t\kappa(t,s)\big\vert_{
    t=0}=-\beta''(\rho_0(s))(\partial_s\rho_0(s))^2c_0
    <0$ for $s\in[a,b]$.
Thus, $\kappa$ becomes negative on $[a,b]$ for $t>0$ instantaneously and for any $c_0$.
\end{ex}

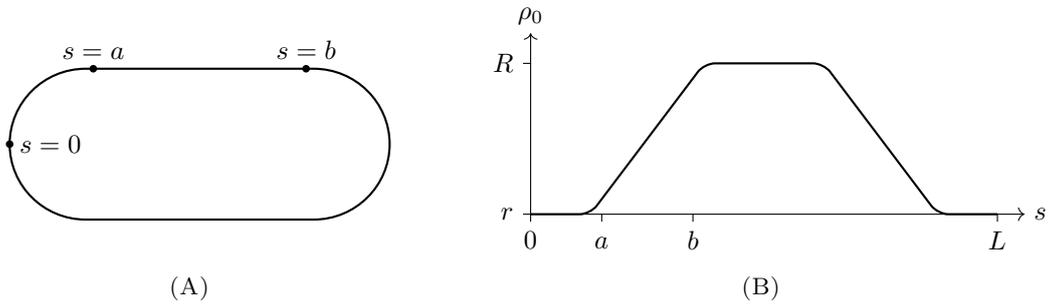
\begin{figure}[h]
  \centering
  \subfloat[][]{\label{fig:cigarre:A}
  \begin{tikzpicture}[color=black,style=thick]
\draw (1,0) -- (4,0)
(4,0) arc (-90:90:1)
(4,2) -- (1,2)
(1,0) arc (90:-90:-1) 
;
\draw[color=white](1,-0.5) -- (1.001,-0.5);
\fill (0,1) circle (0.05) node[right]{$s=0$};
\fill[color=black] (1.1,2) circle (0.05) node[above]{\textcolor{black}{$s=a$}};
\fill[color=black] (3.9,2) circle (0.05) node[above]{\textcolor{black}{$s=b$}};
\end{tikzpicture}
    }
  \qquad\quad
  \subfloat[][]{\label{fig:cigarre:B}
  \begin{tikzpicture}
\draw [->] (0,0) -- (6.5,0) node[right]{$s$};
\draw [->] (0,0) -- (0,2.4) node[above]{$\rho_0$};
\draw [rounded corners, style=thick]
(0,0) -- (pi/4,0)
 -- (pi/4+1.5,2)
 -- (3*pi/4+1.5,2)
 -- (3*pi/4+3,0)
 -- (pi+3,0);
 \draw (0,0) -- (0,-0.1) node[below]{$0$};
 \draw (pi+3,0) -- (pi+3,-0.1) node[below]{$L$};
 \draw (0,0) -- (-0.1,0) node[left]{$r$};
 \draw (0,2) -- (-0.1,2) node[left]{$R$};
 \draw[color=black] (pi/4+0.15,0) -- (pi/4+0.15,-0.1);
 \draw[color=white] (pi/4+0.15,-0.1) -- (pi/4+0.15,-0.18) node[below]{\textcolor{black}{$a$}};
 \draw[color=black] (pi/4+1.35,0) -- (pi/4+1.35,-0.1) node[below]{\textcolor{black}{$b$}};
\end{tikzpicture}}
  \caption{Cigar-shaped curve \subref{fig:cigarre:A} with linear density on $[a,b]$ \subref{fig:cigarre:B}.}
  \label{fig:cigarre}
\end{figure}

\subsection{Nonnegativity of the density}
\label{sec:densnonneg}

For a physical mass density $\rho$, 
only a positive sign is meaningful. We state a condition on $\beta$ which gives a lower bound on $\rho$ and particularly enables us to control its sign. 
Compared to the situation of \Cref{prop:convex}, the evolution equation for $\rho$ in \eqref{eq:flow equation} is not linear in $\rho$. This asks for stricter assumptions on $\beta$, which we again show to be sharp, see \Cref{ex:rhobound}.
\begin{prop}
\label{prop:rhopos}
Let $x_0\in\RR$.
Let $\beta\in C^\infty(\RR)$ be such that $\beta^\prime\geq0$ on $[x_0,\infty)$ and $\beta^\prime(x_0)=0$.
Let $(\theta,\rho)$ be the global solution of \eqref{eq:flow equation} with admissible initial datum $(\theta_0,\rho_0)$. 
If $\rho_0\geq x_0$ on $[0,L]$, then $\rho\geq x_0$ on $(0,\infty)\times[0,L]$.
\end{prop}
\begin{proof}
As in the proof of \Cref{prop:convex}, consider the $L$-periodic extension of the global solution $(\theta,\rho)$ to all of $\RR$. Since $\beta'(x_0)=0$, the function 
\begin{align}
    f(x):=\begin{cases}
    \frac{\beta^\prime(x)}{x-x_0}\quad &\text{ for } x\neq x_0,\\
    \beta^{\prime\prime}(x)\quad &\text{ for } x= x_0
    \end{cases}
\end{align}
is continuous. Choose $R>0$ such that $\sup_{t\in[0,\infty),s\in\RR}\left\vert\rho\right\vert<R$ and $K>0$ such that
\begin{align}
\label{eq:defK}
    \sup_{t\in[0,\infty),\,s\in\RR}\Big(\Big(\sup_{x\in[-R,R]}\left\vert\beta''(x)\right\vert-f(\rho(t,s))\Big)(\partial_s\theta(t,s)-c_0)^2\Big)<2K.
\end{align} 
Let $\varepsilon, T>0$ be arbitrary. Define $\rho_\varepsilon:=\rho-x_0+\varepsilon\exp(Kt)$ 
and
\begin{align}
\label{eq:tstar3}
    t^{\ast}:=\sup\big\lbrace t\in[0,T]\colon \min_{s\in\RR}\rho_{\varepsilon}(\tau,s)> 0 \text{ for all } \tau\in(0,t)\big\rbrace.
\end{align}
We assume that $t^\ast<T$. 
Observe that with \eqref{eq:flow equation} and \eqref{eq:lambdarho},
\begin{align}
    \partial_t\rho_\varepsilon
    =&\;\mu\partial_s^2\rho_\varepsilon-\frac12f(\rho)(\partial_s\theta-c_0)^2\rho_\varepsilon+\frac{1}{2}f(\rho)(\partial_s\theta-c_0)^2\varepsilon e^{Kt}
    +K\varepsilon e^{Kt}\\
    &+\frac{1}{2L}\int_0^L\beta'(\rho_\varepsilon+x_0)(\partial_s\theta-c_0)^2\intd s+\frac{1}{2L}\int_0^L\left(\beta'(\rho)-\beta'(\rho_\varepsilon+x_0)\right)(\partial_s\theta-c_0)^2\intd s.\label{eq:dtrhoepsilon}
\end{align}
For $t\in[0,t^\ast]$, $\rho_\varepsilon(t,s)+x_0\geq x_0$ on $\RR$ and hence $\beta'(\rho_\varepsilon(t,s)+x_0)\geq0$. Thus, the first integral in \eqref{eq:dtrhoepsilon} has positive sign. Choosing $\varepsilon$ so small that $\sup_{t\in[0,T],s\in\RR}(\rho+\varepsilon e^{Kt})\leq R$, we estimate the second integral in \eqref{eq:dtrhoepsilon} for $t\in[0,t^\ast]$ by 
\begin{align}
    \frac{1}{2L}\int_0^L\left(\beta'(\rho)-\beta'(\rho_\varepsilon+x_0)\right)(\partial_s\theta-c_0)^2\intd s
    \geq -\frac{1}{2}\sup_{[-R,R]}\left\vert\beta''\right\vert\varepsilon e^{Kt}\sup_{t\in[0,\infty),s\in\RR}(\partial_s\theta-c_0)^2.
\end{align}
Thus it follows with the definition of $K$ in \eqref{eq:defK} that for $t\in[0,t^\ast]$,
\begin{align}
    \partial_t\rho_\varepsilon>\mu\partial_s^2\rho_\varepsilon-\frac{1}{2}f(\rho)(\partial_s\theta-c_0)^2\rho_\varepsilon.
\end{align}
Arguing as in the proof of \Cref{prop:convex} we find 
$t^\ast=T$, which implies that $\rho_\varepsilon\geq0$ on $[0,T]\times\RR$. The limit $\varepsilon\to0$ yields $\rho\geq x_0$ on $[0,T]\times\RR$. Since $T$ was chosen arbitrarily, the claim follows.
\end{proof}

\begin{rem}
    With the same arguments it can be shown that for $\beta\in C^\infty(\RR)$ such that $\beta'\leq0$ on $(-\infty,x_0]$ and $\beta'(x_0)=0$ for some $x_0\in\RR$, $\rho_0\leq x_0$ on $[0,L]$ implies that $\rho\leq x_0$ on $(0,\infty)\times[0,L]$.
\end{rem}

The following example shows that it is necessary to impose some conditions on $\beta$ in order to control the sign of $\rho$.

\begin{ex} \label{ex:rhobound}
Let $x_0, R\in \RR$, $x_0>R$ and suppose 
that $\beta'(x_0)\geq \beta'(x)$ for all $x\in [x_0,R]$. Assume $\omega\neq 0$ and $c_0\neq \frac{2\pi\omega}{L}$. Consider an admissible initial datum with
\begin{align}
    \kappa_0\equiv \frac{2\pi\omega}{L}, \quad x_0\leq\rho_0\leq R, \quad \rho_0\equiv x_0 \text{ on } [a,b],\quad \rho\not\equiv x_0 \text{ on } [0,L]
\end{align}
for some nontrivial $[a,b]\subset [0,L]$. For $s\in [a,b]$, \eqref{eq:flow equation} and \eqref{eq:lambdarho} imply
\begin{align}
    \partial_t\rho(t,s)\big\vert_{t=0} 
    =-\frac12\left(\frac{2\pi\omega}{L}-c_0\right)^2\bigg(\beta^\prime(x_0)-\frac1L\int_0^L\beta^\prime(\rho_0(s))\intd s\bigg)
    <0.
\end{align}
Thus, on $[a,b]$ the density is not bounded from below by $x_0$ for $t>0$.
\end{ex}


\subsection{Embeddedness}
\label{sec:embedded}

We now discuss the (non)preservation of embeddedness. 
The curve $\gamma$ described by the angle function $\theta$ will no longer satisfy a second order parabolic equation (not even one with Lagrange multipliers), but a nonstandard integro-differential equation instead. Indeed, if $\theta$ evolves according to \eqref{eq:flow equation}, we may integrate \eqref{eq:gamma}
to describe the evolution of $\gamma$. Denoting by 
$n=(-\sin\theta,\cos\theta)$
the usual normal vector field along $\gamma$, by integration by parts we find
\begin{align}
    \partial_t \gamma &= \beta(\rho)(\partial_s^2 \gamma- c_0 n) + \int_0^s \partial_s \gamma\,\beta(\rho)\kappa(\kappa-c_0)\intd r + \int_0^s 
        n (\lambda_{\theta1}\sin\theta-\lambda_{\theta2}\cos\theta)
    \intd{r} + v(t).
\end{align}
Here $v(t)\in \R^2$ attributes for $\partial_t\gamma(t,0)$ and the boundary term arising at $s=0$. Note that the second term on the right hand side still contains $\kappa$, a term of order two, in a nonlocal way. While it is possible to express $\theta$ and thus the right hand side entirely in terms of $\gamma$ and its derivatives, the resulting evolution equation is rather complicated and of nonstandard structure. 

In particular, in contrast to the curve shortening flow (cf.\ \cite{Gage_Hamilton_86}), we cannot rely on classical maximum principles to show the preservation of embeddedness. Instead, we use a recent energy-based argument \cite{MMR2021} to conclude preservation of embeddedness for explicitly small initial energy. Note that $\omega=\pm 1$ for any embedded curve by Hopf's Umlaufsatz.

\begin{prop}\label{prop:li_yau_emb}
    Let $C_{2T}\approx 146.628$ be as in \cite[Theorem 1.1]{MMR2021}. Let $(\theta,\rho)$ be the global solution of \eqref{eq:flow equation} with admissible initial datum $(\theta_0,\rho_0)$. Suppose that $\theta_0$ describes an embedded curve with rotation index $\omega=1$ and $(\theta_0,\rho_0)$ satisfies
    \begin{align}\label{eq:energy_bound_emb}
        \sE_\mu(\theta_0,\rho_0) \leq \frac{\inf_\RR\beta}{2}\Big(\frac{C_{2T}}{L}-4\pi c_0 + Lc_0^2\Big).
    \end{align}
    Then $\theta(t,\cdot)$ describes an embedded curve for all $t> 0$.
\end{prop}

\begin{rem}
    If $\inf_{\R} \beta>0$, then the above threshold is nontrivial, i.e.\ there exist an admissible initial datum $(\theta_0,\rho_0)$ satisfying \eqref{eq:energy_bound_emb}. Indeed let $\omega=1$, $\varepsilon>0$ and let $\nu\in \R$ such that $\beta(\nu)\leq \inf_{\R}\beta + \varepsilon$. Consider $\rho_0\equiv \nu$ and let $\theta_0(s)=\frac{2\pi s}{L}$. Then, we have
    \begin{align}
        \sE_\mu(\theta_0,\rho_0) = \frac{\beta(\nu)}{2} \Big( \frac{4\pi^2}{L} -4\pi c_0+ Lc_0^2\Big).
    \end{align}
    Using that $4\pi^2<C_{2T}$ it follows that the assumptions of \Cref{prop:li_yau_emb} are satisfied for $\varepsilon>0$ small enough. Moreover, $4\pi^2<C_{2T}$ also yields that the right hand side of \eqref{eq:energy_bound_emb} is positive if $\inf_\RR \beta>0$.
\end{rem}

\begin{proof}[Proof of \Cref{prop:li_yau_emb}]
    Let $t> 0$. If $(\theta_0, \rho_0)$ is stationary, then there is nothing to show. Thus, by \eqref{eq:decreaseE}, we may assume that the energy is strictly decreasing, i.e.\  $\sE_\mu(\theta(t,\cdot), \rho(t,\cdot))<\sE_\mu(\theta_0,\rho_0)$. Moreover, \eqref{eq:decreaseE} and the assumption yield that 
    \begin{align}
        \int_0^L \kappa(t,s)^2 \intd{s} &\leq \frac{2}{\inf_\RR \beta} \sE_\mu(\theta(t,\cdot), \rho(t,\cdot))+ 4\pi c_0 -Lc_0^2< \frac{2}{\inf_\RR \beta} \sE_\mu(\theta_0, \rho_0)+ 4\pi c_0 -Lc_0^2
        \leq\frac{C_{2T}}{L}.
    \end{align}
    Now, $\theta(t,\cdot)$ describes a closed curve with rotation index one, and hence this curve must be embedded by \cite[Theorem 1.4]{MMR2021}.
\end{proof}
    

\subsection{Symmetry}
\label{sec:symmetry}

The following section is inspired by \cite{L1989}. We restrict ourselves here to $\omega=1$.

\subsubsection{Rotational symmetry}
\label{sec:rotsym}
Without further comment, in this section, we frequently identify $(\theta-\phi,\rho)$ with its $L$-periodic extension to $\RR$. Here, $\phi$ as in \eqref{eq:defphi}. 

\begin{defn}
\label{def:rotsym}
    Let $\omega=1$,  $(\theta,\rho)\in C^\infty([0,L])$ 
    and $k\in\NN$ with $k\geq2$. 
    If $(\theta-\phi,\rho)$ 
    is $\frac{L}{k}$-periodic,  
    we call the heterogeneous curve described by $(\theta,\rho)$ \textit{$k$-fold rotationally symmetric}.
\end{defn} 

\begin{rem}\label{rem:rotsym_all} We gather some immediate consequences of rotational symmetry.
    \begin{enumerate}[(i)]
        \item\label{rem:rotsym} The $\frac{L}{k}$-periodicity of $\theta-\phi$ is equivalent to demand that
    \begin{align}
        \theta(s)
        =\theta\left(s-\tfrac{nL}{k}\right)+\tfrac{2\pi n}{k}
        \label{eq:rotsym theta}
            \quad\text{ for } s\in\big[\tfrac{nL}{k},\tfrac{(n+1)L}{k}\big] \text{ and } n=0,1,\dots, k-1.
    \end{align}
    \item\label{rem:rotsym closed curve}
        If $\theta-\phi$ 
        is $\frac{L}{k}$-periodic, then $\theta$ describes a closed curve. Indeed, using \eqref{eq:rotsym theta} we have
        \begin{align}
            \int_0^L\sin\theta(s)\intd s
            &= \sum_{n=0}^{k-1}\int_\frac{nL}{k}^\frac{(n+1)L}{k}\sin\theta(s)\intd s
            = \sum_{n=0}^{k-1}\int_\frac{nL}{k}^\frac{(n+1)L}{k}\sin\big(\theta\big(s-\tfrac{nL}{k}\big)+\tfrac{2\pi n}{k}\big)\intd s\\
            &= \int_0^\frac{L}{k}\sum_{n=0}^{k-1}\sin\big(\theta(s)+\tfrac{2\pi n}{k}\big),
        \end{align}
        which is zero for $k\geq2$ since 
        \begin{align}
        \label{eq:sum sin zero}
            \sum_{n=0}^{k-1}\sin\!\left(\theta+\tfrac{2\pi n}{k}\right)
            &=\textup{Im}\Big(\exp(i\theta)\sum_{n=0}^{k-1}\exp(i\tfrac{2\pi n}{k})\Big)
            =\textup{Im}\Big(\exp(i\theta)\frac{1-\exp\left(2\pi i\right)}{1-\exp\left(\tfrac{2\pi i}{k}\right)}\Big)
            =0.
        \end{align}
        Analogously, we obtain $\int_0^L\cos\theta(s)\intd s=0$.
        \item If $(\theta,\rho)\in C^\infty([0,L])$ describes a $k$-fold rotationally symmetric heterogeneous curve and we choose 
    \begin{align}\label{eq:gamma(0)}
        \gamma(0):=\frac1L\int_0^Ls\begin{pmatrix}
            \cos\theta(s)\\ \sin\theta(s)
        \end{pmatrix}\intd s
    \end{align}
    (which ensures that $\int_0^L\gamma(s)\intd s=0$), a computation shows that
    for $s\in\big[\nicefrac{nL}{k},\nicefrac{(n+1)L}{k}\big]$ with $n=0,1,\dots, k-1$, we have
    \begin{align}
        \textup{Rot}\big(\tfrac{2\pi n}{k}\big)\gamma\big(s-\tfrac{nL}{k}\big)
        =\gamma(s).
    \end{align}
    Here, $\textup{Rot}(\alpha)$ is the 
    counterclockwise rotation by the angle $\alpha$. 
    Since additionally $\rho\big(s-\tfrac{nL}{k}\big)=\rho(s)$, the graph of the heterogeneous curve described by $(\theta,\rho)$ indeed possesses a $k$-fold rotational symmetry, see \Cref{fig:ex rot sym}.
    \end{enumerate}
    \end{rem}
    
\pgfplotsset{samples=200}
\begin{figure}[h]
  \centering
  \raisebox{8mm}{\includegraphics[width=0.4\textwidth]{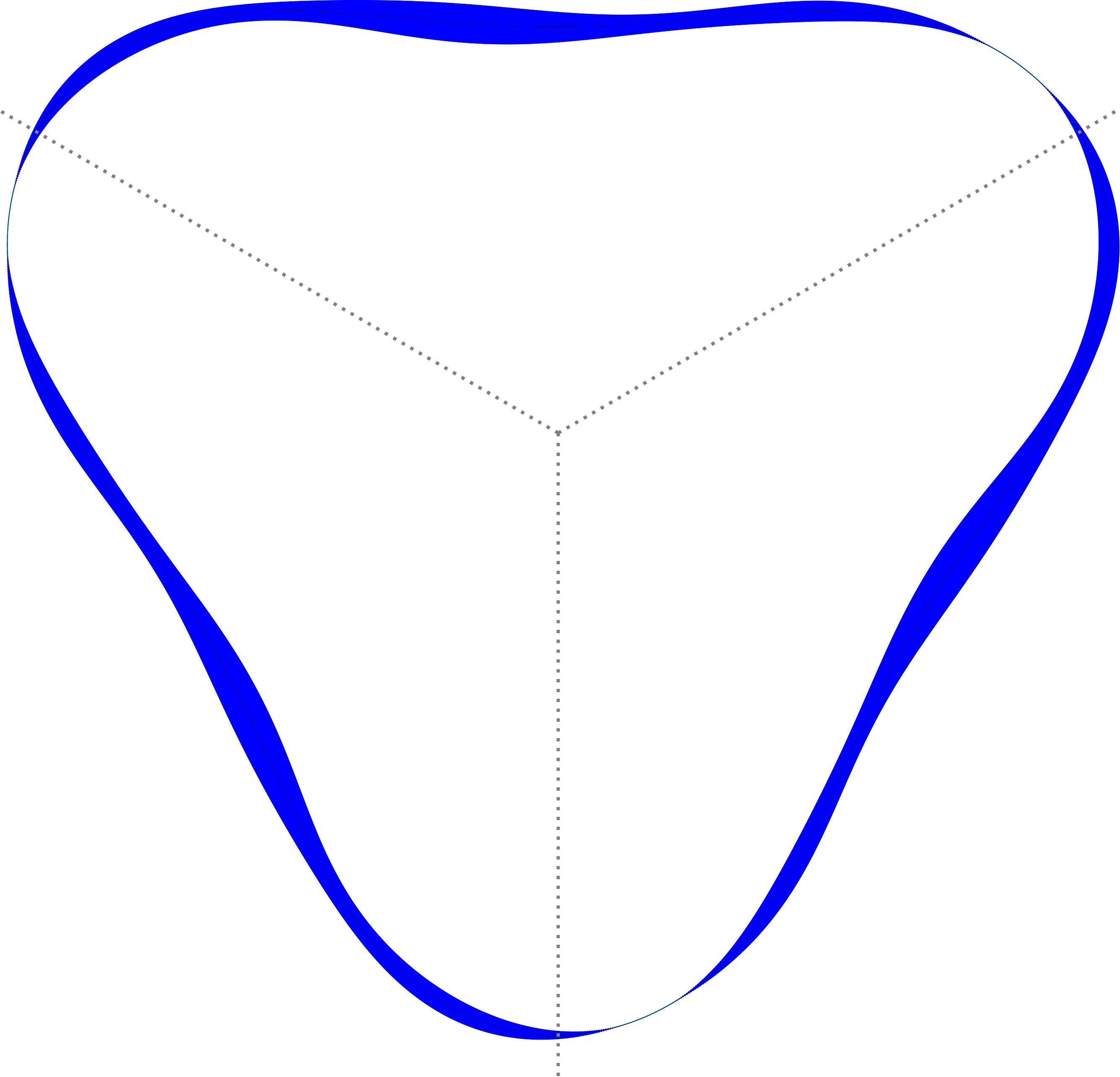}}
\hfill
\begin{tikzpicture}[scale=0.9]
    \begin{axis}[
        xlabel = $s$,
        xtick = {0, 2*pi/3, 4*pi/3, 2*pi},
        xticklabels = {$0$, $\frac{2\pi}{3}$, $\frac{4\pi}{3}$, $2\pi$},
        ytick = {0, 2*pi/3, 4*pi/3, 2*pi, 8*pi/3},
        yticklabels = {$0$, $\frac{2\pi}{3}$, $\frac{4\pi}{3}$, $2\pi$, $\frac{8\pi}{3}$},
        grid = both,
        enlarge x limits=false,
        clip = false,
        ]
        \addplot[domain=0:2*pi, color=gray] {0.5*sin(3*x * 180/pi)+x} node [right] {$\theta$};
        \addplot[domain=0:2*pi, color=red] {0.5*sin(3*x * 180/pi)} node [right] {$\theta-\phi$};
        \addplot[domain=0:2*pi, color=orange] {3*(cos(6*(x+7*pi/9) *180/pi)-0.4)*(sin(3*(x+7*pi/9) *180/pi))+4*pi/3} node [right] {$\rho$};
        \addplot[domain=0:2*pi, color=teal] {1.5*cos(3*x * 180/pi) + 1} node [right] {$\kappa$};
    \end{axis}
\end{tikzpicture}

  \caption{Example of a $3$-fold rotationally symmetric configuration $(\theta,\rho)$.}
  \label{fig:ex rot sym}
\end{figure}
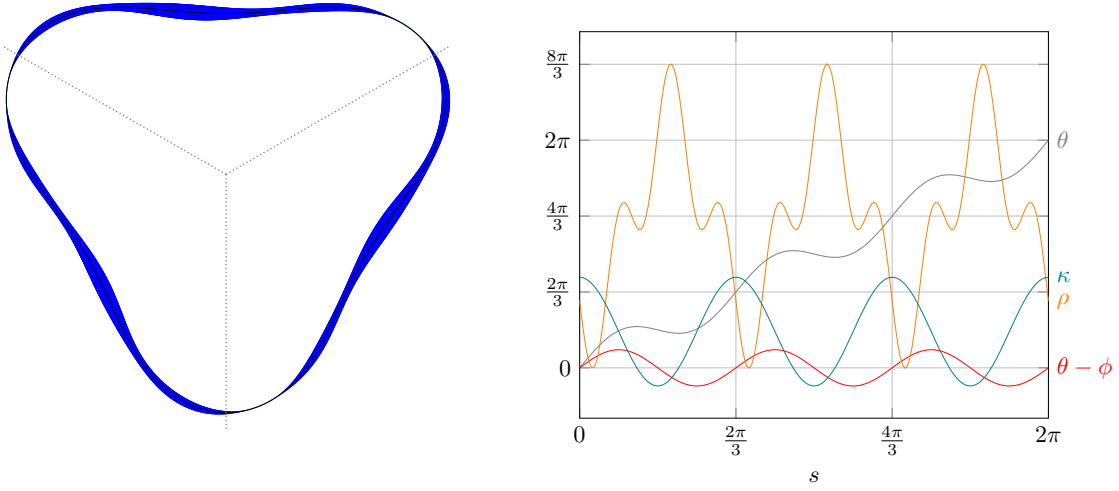

We show that the flow \eqref{eq:flow equation} preserves $k$-fold rotational symmetry for all $k\in\NN$, $k\geq2$.

\begin{prop}
    \label{prop:rotsym preserved}
    Let $\omega=1$ and $k\geq2$. Let $(\theta,\rho)$ be the solution of \eqref{eq:flow equation} with admissible initial datum $(\theta_0,\rho_0)$. 
    If $(\theta_0,\rho_0)$ describes a $k$-fold rotationally symmetric heterogeneous curve, then so does $(\theta,\rho)$ for all $t\in(0,\infty)$.
\end{prop}

\begin{proof} 
    We consider the $L$-periodic extension of $(u,\rho):=(\theta-\phi,\rho)$ to $\RR$. Notice that $(u,\rho)\in C^\infty((0,\infty)\times\RR)$. 
    We define $(\tilde u,\tilde\rho)\in C^\infty((0,\infty)\times\RR)$ by 
    $
        (\tilde u,\tilde\rho)(t,s):=(u,\rho)\big(t,s-\tfrac{L}{k}\big)
    $
    and $(\tilde\theta,\tilde\rho)\in C^\infty((0,\infty)\times\RR)$ by $(\tilde\theta,\tilde\rho):=(\tilde u+\phi,\tilde\rho)$. Notice, that $\tilde\theta\in C^\infty((0,\infty)\times\RR)$ is not periodic. Our intention now is to show that the restriction of $(\tilde\theta,\tilde\rho)$ to $(0,\infty)\times[0,L]$ solves \eqref{eq:flow equation} with initial datum $(\theta_0,\rho_0)$. By uniqueness of the solution (see \Cref{thm:zsf artcl1}), it then follows that $(\tilde\theta-\phi,\tilde\rho)=(\theta-\phi,\rho)$ 
    and thus
    \begin{align}
        (\theta-\phi,\rho)(t,s)
        =(\tilde u,\tilde\rho)(t,s)
        &=(u,\rho)\big(t,s-\tfrac{L}{k}\big)
        =(\theta-\phi,\rho)\big(t,s-\tfrac{L}{k}\big)
    \end{align}
    for $s\in\RR$.
    Thus, the heterogeneous curve described by $(\theta,\rho)$ is $k$-fold rotationally symmetric. 
    \\
    It is clear that $(\tilde\theta,\tilde\rho)\in C^\infty((0,\infty)\times[0,L])$. With \eqref{eq:defphi} and the $L$-periodicity of $u$, we have
    \begin{align}\label{eq:thetatildetheta}
    \begin{split}
    \tilde\theta(t,s)
    &=\tilde u(t,s)+\phi(s)
    =u\big(t,s-\tfrac{L}{k}\big)+\phi\big(s-\tfrac{L}{k}\big)+\phi\big(\tfrac{L}{k}\big)
    =\theta\big(t,s-\tfrac{L}{k}\big)+\tfrac{2\pi }{k}
        \\ 
    & \text{ for } (t,s)\in(0,\infty)\times\big(\tfrac{L}{k},L\big],\\
    \tilde\theta(t,s)&=\theta\big(t,s-\tfrac{L}{k}+L\big)+\tfrac{2\pi }{k}-2\pi
        \quad \text{ for } (t,s)\in(0,\infty)\times\big[0,\tfrac{L}{k}\big].
    \end{split}
    \end{align}
    Using this and trigonometric identities, we see that for $t\geq0$ and $s\in[0,L]$, $(\tilde\theta,\tilde\rho)$ solves
    \begin{align}
    \label{eq:eveqtilde}
        \partial_t\tilde\theta
        =
        \partial_s\big(\beta(\tilde\rho)(\partial_s\tilde\theta-c_0)\big)+\alpha_1\sin\tilde\theta-\alpha_2\cos\tilde\theta
    \end{align}
    with
    $
        \alpha_1:=\lambda_{\theta1}(\theta,\rho)\cos\big(\tfrac{2\pi }{k}\big)-\lambda_{\theta2}(\theta,\rho)\sin\big(\tfrac{2\pi }{k}\big)
    $, $\alpha_2:=\lambda_{\theta1}(\theta,\rho)\sin\big(\tfrac{2\pi }{k}\big)+\lambda_{\theta2}(\theta,\rho)\cos\big(\tfrac{2\pi }{k}\big)$.
    We use \eqref{eq:thetatildetheta} and the fact that $\theta$ describes a closed curve to obtain
    \begin{align}
        \int_0^L\sin\tilde\theta\intd s
        &=\int_0^{\frac{L}{k}}\sin\big(\theta\big(s-\tfrac{L}{k}+L\big)+\tfrac{2\pi }{k}\big)\intd s
        +\int_\frac{{L}}{k}^L\sin\big(\theta\big(s-\tfrac{L}{k}\big)+\tfrac{2\pi}{k}\big)\intd s\\
        &=\int_0^L\sin\big(\theta+\tfrac{2\pi}{k}\big)\intd s
        =\cos\big(\tfrac{2\pi}{k}\big)\int_0^L\sin\theta\intd s+\sin\big(\tfrac{2\pi}{k}\big)\int_0^L\cos\theta\intd s
        =0 
    \end{align}
    (cf.\ \eqref{eq:closedcurve}) and $\int_0^L\cos\tilde\theta\intd s=0$ for all $t\geq0$. In particular, this implies
    \begin{align}
        \int_0^L\sin\tilde\theta\,\partial_t\tilde\theta\intd s=\int_0^L\cos\tilde\theta\,\partial_t\tilde\theta\intd s
        &=0.\label{eq:above}
    \end{align}
    Inserting \eqref{eq:eveqtilde} into \eqref{eq:above} and comparing to \eqref{eq:lambdatheta}, a short computation yields $\alpha_1=\lambda_{\theta1}(\tilde\theta,\tilde\rho)$ and $\alpha_2=\lambda_{\theta2}(\tilde\theta,\tilde\rho)$. Moreover, we check that $\lambda_\rho(\tilde\theta,\tilde\rho)=\lambda_\rho(\theta,\rho)$ and 
    \begin{align}
        \partial_t\tilde\rho=\mu\partial_s^2\tilde\rho-\frac12\beta'(\tilde\rho)(\partial_s\tilde\theta-c_0)^2+\lambda_\rho(\tilde\theta,\tilde\rho).
    \end{align}
    Next, we notice that $(\tilde\theta,\tilde\rho)$ satisfies the boundary conditions 
    $$\tilde\theta(t,L)-\tilde\theta(t,0)
    =\tilde u(t,L)+\phi(L)-\tilde u (t,0)
    =u\big(t,L-\tfrac{L}{k}\big)+2\pi-u\big(t,-\tfrac{L}{k}\big)
    =2\pi,$$
    $\tilde\rho(t,L)=\tilde\rho(t,0)$, and $\partial_s(\tilde\theta,\tilde\rho)(t,L)=\partial_s(\tilde\theta,\tilde\rho)(t,0)$ for all $t>0$.
    Since $(\theta_0-\phi,\rho_0)$ describes a $k$-fold rotationally symmetric heterogeneous curve, we know with \eqref{eq:thetatildetheta} and \Cref{rem:rotsym_all}  \eqref{rem:rotsym} that
    \begin{align}
        \lim_{t\to0}(\tilde\theta,\tilde\rho)(t,s)
        =\lim_{t\to0}\big(\theta\big(t,s-\tfrac{L}{k}\big)+\tfrac{2\pi}{k},\rho\big(t,s-\tfrac{L}{k}\big)\big)=(\theta_0,\rho_0)(s)
    \end{align}
    (in $C^{2+\alpha}([0,L])$ for all $\alpha\in(0,1)$) for $s\in\big[\frac{L}{k},L\big]$ and analogously for $s\in\big[0,\tfrac{L}{k}\big]$.
    Hence, $(\tilde\theta,\tilde\rho)$ solves \eqref{eq:flow equation} with initial datum $(\theta_0,\rho_0)$. 
\end{proof}

A distinctive feature of rotationally symmetric configurations is that the Lagrange multipliers $\lambda_{\theta1}$ and $\lambda_{\theta2}$ vanish. 

\begin{lem}
    \label{lem:Lag mult vanish}
    Let 
    $k\geq2$ and let 
    $(\theta,\rho)\in C^\infty([0,L])$
    describe a $k$-fold rotationally symmetric heterogeneous curve. Then $\lambda_{\theta1}(\theta,\rho)=\lambda_{\theta2}(\theta,\rho)=0$.
\end{lem}

\begin{proof}
    With 
    the $\frac{L}{k}$-periodicity of $\rho$ and with \eqref{eq:rotsym theta} 
    we obtain that 
    \begin{align}
        \int_0^L
        \begin{pmatrix}
            -\sin\theta \\ \cos\theta
        \end{pmatrix}
        \partial_s\big(\beta(\rho)(\partial_s\theta-c_0)\big)\intd s
        =\sum_{n=0}^{k-1}\int_0^\frac{L}{k}
        \begin{pmatrix}
            -\sin\!\left(\theta+\tfrac{2\pi n}{k}\right) \\ \cos\!\left(\theta+\tfrac{2\pi n}{k}\right)
        \end{pmatrix}
        \partial_s\big(\beta(\rho)(\partial_s\theta-c_0)\big)\intd s.
    \end{align}
    This sum is zero as we have seen in \Cref{rem:rotsym_all} \eqref{rem:rotsym closed curve}.
    With 
    \eqref{eq:lambdatheta}, this yields $\lambda_{\theta1}=\lambda_{\theta2}=0$.
\end{proof}

The vanishing of $\lambda_{\theta1}$ and $\lambda_{\theta2}$ 
gives the following extension
of \Cref{lem:critpointconstkappa}. 
If $\omega=1$, $c_0\neq\frac{2\pi}{L}$ and $(\theta,\rho)$ is a constrained critical point with constant curvature, which describes a $k$-fold rotationally symmetric configuration, then $\mathcal{L}^1(\{x : \beta'(x)=0\})=0$ implies $\rho\equiv\nu$.

Furthermore, \Cref{lem:Lag mult vanish} together with \Cref{prop:rotsym preserved} enables us to prove \Cref{thm:pisym kappa geq c0}.

\begin{proof}[Proof of \Cref{thm:pisym kappa geq c0}.]
    With \Cref{prop:rotsym preserved} we know that the solution $(\theta-\phi,\rho)$ remains $\frac{L}{k}$-periodic for all $t\geq0$. 
    Hence, \Cref{lem:Lag mult vanish} implies that $\lambda_{\theta1}(t)=\lambda_{\theta2}(t)=0$ for all $t\geq0$. 
    This is why \eqref{eq:ev eq kappa} can be written as
    \begin{align}
    \label{eq:dt kappa-c0}
        \partial_t(\kappa-c_0)=\beta(\rho)\partial_s^2(\kappa-c_0)+2\partial_s(\beta(\rho))\partial_s(\kappa-c_0)+\partial_s^2(\beta(\rho))(\kappa-c_0).
    \end{align}
    Since \eqref{eq:dt kappa-c0} has the same structure as \eqref{eq:dtkappa}, we can proceed as in the proof of \Cref{prop:convex} to show the claim. 
\end{proof}
In \Cref{ex:counterconvex}, we saw that convexity may not be preserved for $c_0\neq 0$. With a small modification of the density, the example can be made rotationally symmetric, so that convexity ($\kappa_0\geq 0$) is still lost, while the bound $\kappa_0\geq c_0$ is preserved by \Cref{thm:pisym kappa geq c0}.

\subsubsection{Axial symmetry}
\label{sec:axsym}

We characterize axial symmetry of heterogeneous curves (after a possible shift in the $s$-argument) as follows.

\begin{defn}
    \label{def:axsym}
    Let $(\theta,\rho)\in C^\infty([0,L])$. 
    We call the heterogeneous curve described by $(\theta,\rho)$ \textit{axially symmetric}, if
    \begin{align}\label{eq:axsymkappa}
        (\partial_s\theta,\rho)(s)=(\partial_s\theta,\rho)(L-s) \quad \text{ for } s\in[0,L].
    \end{align}
\end{defn}

\begin{rem}
    Axial symmetry implies the following properties.
    \begin{enumerate}[(i)]
        \item Integrating \eqref{eq:axsymkappa} we obtain the equivalent condition
    \begin{align}
        \theta(s)=\theta(L)+\theta(0)-\theta(L-s) \quad \text{ for } s\in[0,L].\label{eq:axsymtheta}
    \end{align}
    \item Let $(\theta,\rho)\in C^\infty([0,L])$ satisfy \eqref{eq:axsymkappa}  
    and $\gamma(0)\in\RR^2$ be 
    as in \eqref{eq:gamma(0)}. 
    A computation shows that
    \begin{align}\label{eq:axsym}
        \textup{Ref}\big(\theta(0)-\tfrac{\pi}{2}\big)\,\gamma(s)
        =\gamma(L-s),
    \end{align}
    where $\textup{Ref}(\alpha)$ 
    represents the reflection about an axis through the origin 
    with angle $\alpha$ to the $x$-axis.
    Since additionally $\rho(s)=\rho(L-s)$, \eqref{eq:axsymkappa} indeed implies that the corresponding heterogeneous curve is axially symmetric, see \Cref{fig:axsym}.
    \item A short computation, relying on \eqref{eq:axsymtheta} and the preservation of the integral of the inclination angle \cite[Lemma 2.3]{DLR2022}, implies that for a solution $(\theta, \rho)$ describing an axially symmetric curve for all times $t\geq 0$, we have $\theta(t,0)=\theta_0(0)$.
    \end{enumerate}
\end{rem}

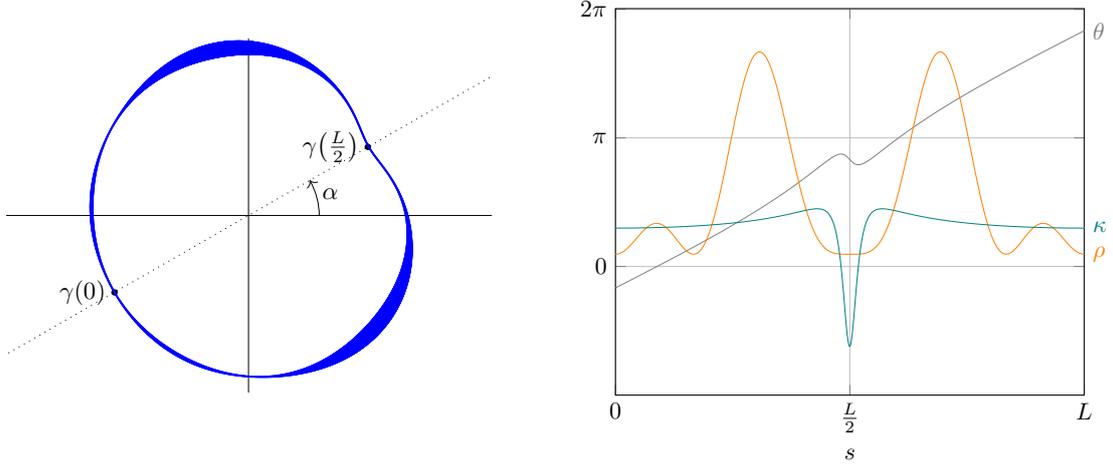
\begin{figure}[h]
  \centering

  \raisebox{1cm}{
  \begin{tikzpicture}[scale=0.93]
      \begin{axis}[
          xmin=-1.5,
          xmax=1.5,
          ymin=-1.1,
          ymax=1.1,
          hide axis,
          unit vector ratio*={1 1 1},  
          ]
          \addplot[samples=50, smooth] coordinates {(0,-1.5)(0,2)};
          \addplot[samples=50, smooth] coordinates {(-2,0)(3,0)};
          \addplot[samples=50, smooth, dotted] coordinates {(-5.196,-3)(5.196,3)};
          \addplot+[
              scatter,
              mark=*,
              color=blue,
              point meta=explicit symbolic,
              scatter/@pre marker code/.style={/tikz/mark size=0.2+\pgfplotspointmeta/5},
              scatter/@post marker code/.style={}
              ]
              table[header=false, meta index=2]
              {imgs/symmetry/symmetry_picture2.csv};
          \fill (-0.83,-0.48) circle (0.02) node[left]{$\gamma(0)$};
          \fill (0.739,0.428) circle (0.02) node[left]{$\gamma\big(\frac{L}{2}\big)$};
          \coordinate (xaxis) at (\pgfkeysvalueof{/pgfplots/xmax},0);
          \coordinate (origin) at (0,0);
          \addplot [blue,opacity=0] coordinates {(0, 0) (0.86, 0.5)}
              coordinate [at end] (A)
              ;
          \path (xaxis) -- (origin) -- (A)
              pic [
              draw,
              ->,
              angle radius=10mm,
              angle eccentricity=1.2,
              "$\alpha$",
              ] {angle = xaxis--origin--A}
              ;
      \end{axis}
\end{tikzpicture}}
\hfill
\begin{tikzpicture}[scale=0.9]
    \begin{axis}[
        ymin = -3.1415,
        ymax = 6.3,
        xlabel = $s$,
        xtick = {0, pi, 2*pi},
        xticklabels = {$0$, $\frac{L}{2}$, $L$},
        ytick = {0, pi, 2*pi},
        yticklabels = {$0$, $\pi$, $2\pi$},
        grid = both,
        enlarge x limits=false,
        enlarge y limits=false,
        clip = false,
        ]
        \pgfplotstableread{imgs/symmetry/symmetry_picture2.csv}\symmetrydata;
        \addplot[domain=0:2*pi, color=gray] table[x index=4, y index=3] \symmetrydata node [right] {$\theta$};
        \addplot[domain=0:2*pi, color=orange] {0.3+3*(cos(3*(x) * 180/pi)+1)*sin((x) * 180/pi)*sin((x) * 180/pi)} node [right] {$\rho$};
        \addplot[domain=0:2*pi, color=teal] table[x index=4, y index=5] \symmetrydata node [right] {$\kappa$};
    \end{axis}
\end{tikzpicture}
  \caption{Example of an axially symmetric configuration $(\theta,\rho)$.}
  \label{fig:axsym}
\end{figure}

We show that if the initial datum describes an axially symmetric heterogeneous curve, then this also applies to the solution for all $t>0$.

\begin{prop}
    \label{prop:presaxsym}
    Let $\omega=1$ and $(\theta,\rho)$ be the solution of \eqref{eq:flow equation} with admissible initial datum $(\theta_0,\rho_0)$. 
    If $(\theta_0,\rho_0)$ describes a heterogeneous curve which
    is axially symmetric, then so does $(\theta,\rho)$ 
    for all $t\in(0,\infty)$.
\end{prop}

\begin{proof}
   Let $(\theta,\rho)$ be the solution of \eqref{eq:flow equation} with initial datum $(\theta_0,\rho_0)$. We define
    \begin{align}
        (\tilde\theta,\tilde\rho)(t,s):=\big(
        \theta_0(L)+\theta_0(0)-\theta(t,L-s),\, 
        \rho(t,L-s)
        \big)
        \label{eq:tildeaxsym}
    \end{align}
    for $(t,s)\in(0,\infty)\times[0,L]$ and show that $(\tilde\theta,\tilde\rho)$ solves \eqref{eq:flow equation} with the same initial datum $(\theta_0,\rho_0)$. By uniqueness of the solution, it follows that $(\theta,\rho)=(\tilde\theta,\tilde\rho)$ in $[0,\infty)\times[0,L]$. 
    This yields $(\partial_s\theta,\rho)(t,s)=(\partial_s\theta,\rho)(t,L-s)$in $[0,\infty)\times[0,L]$. Thus, $(\theta,\rho)$ describes an axially symmetric heterogeneous curve for all $t>0$.
    \\
    It is clear that $(\tilde\theta,\tilde\rho)\in C^\infty((0,\infty)\times[0,L])$. Since $\partial_s\tilde\theta(s)=(\partial_s\theta)(L-s)$, $\partial_s^2\tilde\theta=-(\partial_s^2\theta)(L-s)$ and $\partial_s\tilde\rho(s)=-(\partial_s\rho)(L-s)$, we see that 
    \begin{align}
        \label{eq:eveqtilde2}
            \partial_t\tilde\theta
            =\beta(\tilde\rho)\partial_s^2\tilde\theta+\beta'(\tilde\rho)\partial_s\tilde\rho(\partial_s\tilde\theta-c_0)+\alpha_1\sin\tilde\theta-\alpha_2\cos\tilde\theta.
        \end{align}
    Here, we have $\alpha_1(t)=\lambda_{\theta1}(\theta,\rho)\cos\big(\theta_0(L)+\theta_0(0)\big)-\lambda_{\theta2}(\theta,\rho)\sin\big(\theta_0(L)+\theta_0(0)\big)$ and 
    $\alpha_2(t)=\lambda_{\theta1}(\theta,\rho)\sin\big(\theta_0(L)+\theta_0(0)\big)+\lambda_{\theta2}(\theta,\rho)\cos\big(\theta_0(L)+\theta_0(0)\big)$ (by using trignometric identities).
    Similarly,
    \begin{align}
        \int_0^L\sin\tilde\theta\intd s
        =\sin\big(\theta_0(L)-\theta_0(0)\big)\int_0^L\cos\theta\intd s-\cos\big(\theta_0(L)-\theta_0(0)\big)\int_0^L\sin\theta\intd s=0.\label{eq:axsymclosed}
    \end{align}
    Analogously, we obtain $\int_0^L\cos\tilde\theta\intd s=0$.
    Thus, $\tilde\theta$ describes a closed curve.
    As in the proof of \Cref{prop:rotsym preserved} one shows that this 
    implies that $\alpha_1=\lambda_{\theta1}(\tilde\theta,\tilde\rho)$ and $\alpha_2=\lambda_{\theta2}(\tilde\theta,\tilde\rho)$. 
    Moreover, 
    $\lambda_\rho(\tilde\theta,\tilde\rho)=\lambda_\rho(\theta,\rho)$ and 
    $$
        \partial_t\tilde\rho=\mu\partial_s^2\tilde\rho-\frac12\beta'(\tilde\rho)(\partial_s\tilde\theta-c_0)^2+\lambda_\rho(\tilde\theta,\tilde\rho).
    $$
    It is readily checked that $(\tilde\theta,\tilde\rho)$ satisfies the boundary conditions in \eqref{eq:flow equation}. 
    Using that $(\theta_0,\rho_0)$ describes an axially symmetric configuration, we further have
    \begin{align}
        \lim_{t\to0}\big(\tilde\theta,\tilde\rho\big)(t,s)=\big(\theta_0(L)+\theta_0(0)-\theta_0(L-s),\rho_0(L-s)\big)
        =(\theta_0,\rho_0)(s)
    \end{align}
    (in $C^{2+\alpha}([0,L])$ for all $\alpha\in(0,1)$). 
    Hence, $(\tilde\theta,\tilde\rho)$ is a solution of \eqref{eq:flow equation} with initial datum $(\theta_0,\rho_0)$ and the claim follows.
\end{proof}

\begin{rem}
    The proof of \Cref{prop:presaxsym} shows that for an axially symmetric initial datum $(\theta_0,\rho_0)$, the solution $(\theta,\rho)$ keeps $\theta(t,0)=\theta_0(0)$ for all $t\in[0,\infty)$.
\end{rem}

\section{Asymptotic behavior}
\label{sec:proplim}

Since stationary solutions of \eqref{eq:flow equation} are precisely the constrained critical points 
(compare 
\eqref{eq:ELp1}--\eqref{eq:ELp2} with 
\eqref{eq:stationary}), we can already derive some properties of the limit $(\theta_\infty,\rho_\infty)$ in Theorem \ref{thm:zsf artcl1} by using the classification of constrained critical points in Sections \ref{sec:exmin} and \ref{sec:homelastica}.

\subsection{Convergence to a homogeneous elastica under growth assumptions on \texorpdfstring{$\beta$}{β}}
\label{sec:convassbeta}

In this section, we impose some additional assumptions on the initial datum and the model parameters, under which the limit of \eqref{eq:flow equation} is a homogeneous elastica.

\begin{proof}[Proof of \Cref{thm:conv special beta neu}.]
    With \eqref{eq:flow equation}, \eqref{eq:lambdarho}, and integration by parts we have
    \begin{align}
        \frac{\intd}{\intd t}\int_0^L(\rho-\nu)^2\intd s
        &=2\int_0^L\rho\,\partial_t\rho\intd s
        =-2\mu\int_0^L(\partial_s\rho)^2\intd s+\int_0^L(\nu-\rho)\beta^\prime(\rho)(\partial_s\theta-c_0)^2\intd s.\label{eq:variancerho}
    \end{align}
    We use \eqref{eq:beta'bound} to estimate the second term on the right hand side of \eqref{eq:variancerho} and obtain
    \begin{align}
        \frac{\intd}{\intd t}\int_0^L(\rho-\nu)^2\intd s
        &\leq -2\mu\int_0^L(\partial_s\rho)^2\intd s+\bar{C}\int_0^L(\nu-\rho)^2\beta(\rho)(\partial_s\theta-c_0)^2\intd s \nonumber\\
        &\leq -2\mu\int_0^L(\partial_s\rho)^2\intd s+ \bar{C}\sup_{s\in[0,L]}(\nu-\rho)^2 \int_0^L\beta(\rho)(\partial_s\theta-c_0)^2\intd s\nonumber\\
        &\leq -2\mu\int_0^L(\partial_s\rho)^2\intd s+ 2L\bar{C} \sE_\mu(\theta_0, \rho_0) \int_0^L (\partial_s \rho)^2\intd s,\nonumber
    \end{align}
    using \eqref{eq:decreaseE}, $\sup_{s\in[0,L]}\vert\nu-\rho\vert\leq \int_0^L |\partial_s \rho|\intd s$ and Cauchy--Schwarz.
    This yields
    \begin{align}
    \frac{\intd}{\intd t}\int_0^L(\rho-\nu)^2\intd s
    \leq -2\bigg(\frac{2\pi}{L}\bigg)^2\big(\mu-L\bar C\sE_\mu(\theta_0,\rho_0)\big) \int_0^L(\rho-\nu)^2\intd s.
    \end{align}
    By Gronwall's inequality we conclude that
    \begin{align}
    \label{eq:rateofconv}
        \int_0^L\left(\rho-\nu\right)^2\intd s
        \leq \Big(\int_0^L\rho_0^2\intd s-\nu^2L\Big)\; \exp\Big(-\frac{8\pi^2}{L^2}\big(\mu-L\bar C\sE_\mu(\theta_0,\rho_0)\big)t\Big).
    \end{align}
    
    Since $\bar{C}L\sE_\mu(\theta_0, \rho_0)<\mu$, it follows that $\rho\to\nu$ in $L^2(0,L)$ exponentially fast. By the subconvergence result in \Cref{thm:zsf artcl1}, there is a sequence $t_n\to\infty$ and $\theta_\infty\in C^\infty([0,L])$ such that $(\theta_\infty,\rho_\infty=\nu)$ is a solution of \eqref{eq:stationary} and $\theta(t_n)\to \theta_\infty$ in $C^2([0,L])$. By \Cref{lem:critpointconstrho}, we find that $(\theta_\infty,\rho_\infty)$ describes a homogeneous elastica.\\
    If $\omega\neq 0$, \Cref{lem:charelasticae} yields that $\theta_\infty$ describes an $\omega$-fold covering of a circle, so that necessarily $\partial_s \theta_\infty = \frac{2\pi\omega}{L}$. From \cite[Lemma 2.3]{DLR2022}, we conclude that $\int_0^L \theta_\infty \intd s = \int_0^L\theta_0\intd s$, which implies $\theta_\infty(s) = \phi(s)+ \frac{1}{L}\int_0^L\theta_0\intd s - \pi\omega L$, $s\in [0,L]$. In particular, $\theta_\infty$ does not depend on the sequence $(t_n)_{n\in\NN}$, and statement (i) follows from a subsequence argument. 
    \\
    In the case $\omega=0$, the analyticity assumption on $\beta$ and \Cref{thm:zsf artcl1} imply that $(\theta(t),\rho(t)) \to (\theta_\infty,\rho_\infty=\nu)$ in $C^2([0,L])$ as $t\to\infty$, where $(\theta_\infty,\rho_\infty)$ satisfies \eqref{eq:stationary}. Again, \Cref{lem:critpointconstrho} implies that $\theta_\infty$ describes an elastica, so necessarily a multifold covered figure eight elastica by \Cref{lem:charelasticae}. Statement (ii) follows.
\end{proof}

The necessity for stronger assumptions in the case $\omega=0$ arises from the parametrization invariance of the energy, a general issue for geometric flows which occurs here despite working only with arclength parametrizations.
Suppose that $(\theta,\rho)$ is a solution to \eqref{eq:flow equation} and $(\theta_\infty,\rho_\infty)$ is a solution to \eqref{eq:stationary} originating from the subconvergence result in \Cref{thm:zsf artcl1}, i.e.\ $(\theta_\infty,\rho_\infty) = \lim_{n\to\infty} (\theta(t_n), \rho(t_n))$ for some sequence $t_n\to\infty$.  With $\phi$ as in \eqref{eq:defphi},  we write $\theta_\infty= u_\infty + \phi$. Identifying $u_\infty, \rho_\infty$ with their smooth $L$-periodic extensions to $\RR$, we find that for any $s_0\in \RR$, the pair
\begin{align}\label{eq:reparam_s0}
	(\hat{\theta}_\infty,\hat{\rho}_\infty)(s) = (\phi(s) +u_\infty(s-s_0),\rho_\infty(s-s_0)), \quad s\in [0,L],
\end{align}
is also stationary for any $s_0\in \RR$. 
In fact, any other arclength parametrization of the corresponding curve leads to an angle function of this form.
In particular, the set of possible limits (i.e.\ solutions to \eqref{eq:stationary}) is nondiscrete, so that {\L}ojasiewicz--Simon gradient inequalities are generically needed for deducing convergence from subconvergence. 
Hence, it is somehow surprising that this argument is not needed in case (i) of \Cref{thm:conv special beta neu}. The reason for this is that if $(\theta_\infty,\rho_\infty)$ describes a circle, then any reparametrization of the form \eqref{eq:reparam_s0} with $s_0\neq 0$ will result in adding a constant to the original angle function since $\theta_\infty$ is affine. Since by \cite[Lemma 2.3]{DLR2022}, $\int_0^L\theta_\infty \intd s$ is determined by the initial datum, this degree of freedom is not present in the case $\omega\neq 0$, resulting in full convergence.
On the curve level, adding a constant to $\theta$ corresponds to a rotation of the associated curve about a fixed angle, i.e.\ for a circle there is a one-to-one correspondence between arclength reparametrizations and rotations. 
\\
On the other hand, in the case $\omega=0$, the classification of solutions to the elastica equation in \cite[Proposition 3.3]{Linner1996} allows us to determine all the parameters, except for the invariance due to \eqref{eq:reparam_s0}, see also \cite[Proposition B.8]{MR2021}. Hence the {\L}ojasiewicz inequality (and consequently analyticity of $\beta$, cf.\ \cite[Corollary 6.3]{ConstrLoja}) is necessary to ensure
convergence. 

For $\omega=1$, we have dealt with rotational symmetry of solutions. In this case we can prove exponential convergence of the curvature to a constant if the length allows for a circle with curvature $\kappa\equiv c_0$.

\begin{prop}\label{prop:conv_symm_c02Pi/L}
    Let $\omega=1$, $ c_0=\frac{2\pi}{L}$, and let $(\theta_0, \rho_0)\in C^\infty([0,L])$ be an admissible initial datum describing a $k$-fold rotationally symmetric heterogeneous curve for some $k\geq 2$. Then, as $t\to\infty$, the solution $(\theta,\rho)$ to \eqref{eq:flow equation} converges exponentially fast to $(\theta_\infty, \rho_\infty)$ with $\partial_s \theta_\infty\equiv c_0$, $\rho_\infty\equiv \nu$. In particular, the limit describes a circle with constant density.
\end{prop}

\begin{proof}
    Let $(\theta,\rho)$ be the solution to \eqref{eq:flow equation} and recall $\sE^\theta(t)= \frac{1}{2}\int_0^L \beta(\rho)(\kappa-c_0)^2\intd s$ with $\kappa=\partial_s\theta$. 
     Due to \Cref{prop:rotsym preserved} and \Cref{lem:Lag mult vanish}, $\lambda_{\theta1}(t)=\lambda_{\theta2}(t)=0$ for all $t\geq0$. Thus
     \begin{align}
         \frac{\intd}{\intd t}&\sE^\theta = \frac{1}{2}\int_0^L \beta'(\rho)\partial_t \rho (\kappa-c_0)^2\intd s + \int_0^L \beta(\rho)(\kappa-c_0)\partial_t \kappa \intd s \\
         &\leq \frac{\sup_{(t,s)}\left\vert\beta'(\rho)\right\vert}{2\inf_{(t,s)}(\beta(\rho))^2}\Vert \partial_t \rho\Vert_{L^\infty(0,L)}\int_0^L (\beta(\rho))^2(\kappa-c_0)^2\intd s - \int_0^L \big( \partial_s \big(\beta(\rho)(\kappa-c_0)\big)\big)^2\intd s \label{eq:conv_symm_1}
     \end{align}
     Note that $\sup_{(t,s)\in[0,\infty)\times[0,L]}\vert\beta'(\rho)\vert<\infty$ and $\inf_{(t,s)\in[0,\infty)\times[0,L]}\beta(\rho)>0$ since by convergence of the flow, $\rho(t,s)$ lies in a compact set for all $(t,s)$. 
     By the assumptions on $\omega$ and $c_0$, the function $\kappa(t)-c_0$ 
     has a zero in $[0,L]$ for all $t\geq 0$. 
     Therefore,
     \eqref{eq:conv_symm_1} and Wirtinger's inequality imply
     \begin{align}
         \frac{\intd}{\intd t}\sE^\theta &\leq \left(C\Vert \partial_t \rho\Vert_{L^\infty(0,L)}- \frac{4\pi^2}{L^2}\right)\int_0^L (\beta(\rho))^2(\kappa-c_0)^2\intd s,
     \end{align}
    where $C=C(\beta, \theta,\rho)\in(0,\infty)$ is a constant independent of $t\geq 0$.
     Now, $(\theta(t),\rho(t))\to (\theta_\infty,\rho_\infty)$ in $C^{2}([0,L])$ and \eqref{eq:flow equation} imply that $\Vert \partial_t \rho\Vert_{L^\infty(0,L)}\to 0$ as $t\to\infty$. Consenquently, we have
     \begin{align}
         \frac{\intd}{\intd t}\sE^\theta(t) \leq -\frac{4 \pi^2 \inf_{(t,s)}\beta(\rho)}{2L^2} \sE^\theta(t),
     \end{align}
     for $t\geq T$ large enough, whence Gronwall's lemma yields $\sE^\theta(t)\leq C e^{-\alpha t}$ for some appropriate $C,\alpha>0$.  
     It follows that $\kappa=\partial_s \theta\to c_0$ in $L^2(0,L)$ for $t\to\infty$ exponentially fast. For the exponential convergence of $\rho$, we use \eqref{eq:variancerho} to conclude
     \begin{align}
         \frac{\intd}{\intd t} \int_0^L(\rho-\nu)^2\intd s 
         \leq -2\mu\frac{4\pi^2}{L^2}\int_0^L(\rho-\nu)^2\intd s+ \sup_{(t,s)} \vert\nu-\rho\vert\frac{\sup_{(t,s)}\vert\beta'(\rho)\vert}{\inf_{(t,s)} \beta(\rho)} \sE^\theta(t).
     \end{align}
     Using $\sE^\theta(t)\leq C e^{-\alpha t}$, the exponential convergence $\rho\to\nu$ in $L^2(0,L)$ follows with a Gronwall argument. Since 
     $\int_0^L\theta\intd s$ is preserved (cf. \cite[Lemma 2.3]{DLR2022}), we have $\theta \to \phi+\frac{1}{L}\int_0^L\theta_0\intd s - \pi$ as $t\to\infty$ exponentially fast by the Poincar\'e--Wirtinger inequality. \Cref{thm:zsf artcl1} and an interpolation argument imply that $(\theta, \rho)\to(\phi+\frac{1}{L}\int_0^L\theta_0\intd s - \pi, \nu)$ exponentially fast in $C^{2+\tilde\alpha}([0,L])$ for all $\tilde\alpha\in (0,\frac{1}{2}).$
\end{proof}
\Cref{prop:conv_symm_c02Pi/L} implies that if $\omega=1$ and $c_0=\frac{2\pi}{L}$, there exists no nontrivial constrained critical point 
which is $k$-fold rotationally symmetric.
Moreover, in this setting, \Cref{prop:conv_symm_c02Pi/L} implies that for $t$ large enough, $\sE^\theta(\theta,\rho)$ is eventually monotonically decreasing (compare to \Cref{sec:decreaseE}).

\subsection{Convergence to a homogeneous elastica for large \texorpdfstring{$\mu$}{μ}}
\label{sec:convlargemu}

In Proposition \ref{prop:minproblargemu}, we have seen that for $\omega\neq 0$ and large $\mu$, the $\omega$-fold covering of the circle with constant density is the unique global minimizer. In \Cref{prop:convlargemu}, we present a time-dependent version of this result 
if $\rho_0\equiv\nu$. We point out that a constant initial density does not necessarily remain constant, see \Cref{rem:Erho_grow}, unless $\beta'(\nu)=0$ or $\partial_s\theta_0\equiv c_0$. 

\begin{proof}[Proof of \Cref{prop:convlargemu}] 
We assume that $\int_0^L\theta_0(s)\intd s=\pi\omega L$. This is no loss of generality because if $(\theta, \rho)$ is the solution to \eqref{eq:flow equation} with initial datum $(\theta_0,\rho_0)\in C^\infty([0,L])$, then by a direct computation using trigonometric identities, it is readily checked that $(\theta+r,\rho)$ is the solution to \eqref{eq:flow equation} with initial datum $(\theta_0+r,\rho)$ for $r\in\RR$. Consider 
$(\mu_j)_{j\in\mathbb{N}}$ such that $\mu_j\to\infty$ for $j\to\infty$. 
By assumption, $\sE_{\mu_j}(\theta_{0},\rho_{0})=:K$ is independent of $j\in\NN$.
For any $\mu_j$, 
there exists a unique global solution $(\theta_j,\rho_j)$ 
with initial datum $(\theta_{0},\rho_{0})$ and this solution converges to some $(\theta_{\infty,j},\rho_{\infty,j})\in C^\infty([0,L])$ in $C^2([0,L])$ for $t\to\infty$ (see Theorem \ref{thm:zsf artcl1} and Proposition \ref{prop:minprobex2}).
Since the integral of the angle is preserved (cf.\ \cite[Lemma 2.3]{DLR2022}), we have
\begin{align} \label{eq:int theta const}
    \int_0^L\theta_{\infty,j}(s)\intd s=\int_0^L\theta_j(t,s)\intd s=\int_0^L\theta_0(s)\intd s=\pi\omega L \quad\text{ for all } t\in(0,\infty),\; j\in\NN.
\end{align}
Thus, 
we want to show that 
\begin{align}
\label{eq:C1 conv j to infty}
    (\theta_{\infty,j},\rho_{\infty,j})\rightarrow(\theta_c,\rho_c) \quad\text{ in }C^1([0,L]) \quad\text{ for }j\to\infty,
\end{align}
cf.\ \Cref{rem:kappaelastica}.
\Cref{cor:C1neigh} and \eqref{eq:int theta const} then allow us to conclude that for $j$ large enough, $(\theta_{\infty,j},\rho_{\infty,j})=(\theta_c,\rho_c)$ and the statement follows. 
\\
\textit{Step 1: Uniform boundedness of $\left\Vert(\theta_{\infty,j},\rho_{\infty,j})\right\Vert_{W^{1,2}(0,L)}$, $\lambda_{\theta1}(\theta_{\infty,j},\rho_{\infty,j})$, $\lambda_{\theta2}(\theta_{\infty,j},\rho_{\infty,j})$.}
For this, we first observe that
\begin{align}
\label{eq:star}
    \int_0^L(\partial_s\rho_{\infty,j})^2\intd s\leq\frac{2}{\mu_j}K\to0, \quad j\to\infty. 
\end{align}
Since the integral of the density is fixed (see \eqref{eq:fixedmass}), 
this yields $\rho_{\infty,j}\to \nu$ in $W^{1,2}(0,L)$ and in particular uniform boundedness of $\left\Vert\rho_{\infty,j}\right\Vert_{W^{1,2}(0,L)}$ and $\Vert\rho_{\infty,j}\Vert_{C([0,L])}$. Thus, 
there is $M\in\RR$ (not depending on $j$) such that 
\begin{align}
\label{eq:mularge1}
    \int_0^L(\partial_s\theta_{\infty,j})^2\intd s\leq\frac{K}{\inf_{[-M,M]}\beta}+4\pi c_0\omega.
\end{align}
With \eqref{eq:int theta const}, we conclude 
that also
$\left\Vert\theta_{\infty,j}\right\Vert_{W^{1,2}(0,L)}$ is uniformly bounded.
By \cite[Lemma 4.1]{DLR2022} and \eqref{eq:mularge1}, the matrix $\Pi^{-1}(\theta_{\infty,j})$ is bounded uniformly in $j$. Hence the bounds on $\left\Vert\theta_{\infty,j}\right\Vert_{W^{1,2}(0,L)}$ and $\Vert\rho_{\infty,j}\Vert_{C([0,L])}$ imply that
\begin{align}
\label{eq:lambda rho part int}
\begin{pmatrix}
\lambda_{\theta1}\\ \lambda_{\theta2}
\end{pmatrix}(\theta_{\infty,j},\rho_{\infty,j})
=\Pi^{-1}(\theta_{\infty,j})\int_0^L\begin{pmatrix}
\cos\theta_{\infty,j} \\ \sin\theta_{\infty,j}
\end{pmatrix}
\partial_s\theta_{\infty,j} \beta(\rho_{\infty,j})(\partial_s\theta_{\infty,j}-c_0)\intd s
\end{align}
is bounded uniformly in $j\in \NN$.\\
\textit{Step 2: Uniform boundedness of $\left\Vert(\theta_{\infty,j},\rho_{\infty,j})\right\Vert_{W^{2,2}(0,L)}$.}
To show boundedness of the $L^2$-norm of the second derivatives, we use that for all $j\in\NN$,
$(\theta_{\infty,j},\rho_{\infty,j})$ is a stationary solution, i.e.\ a solution of \eqref{eq:stationary}. This allows to use similar arguments as in the proof of Proposition \ref{prop:minprobex2}. First, we observe that
\begin{align}
    &\left\Vert\partial_s\big(\beta(\rho_{\infty,j})(\partial_s\theta_{\infty,j}-c_0)\big)\right\Vert_{L^2(0,L)}
    =\left\Vert\lambda_{\theta1}
    \sin\theta_{\infty,j}-\lambda_{\theta2}
    \cos\theta_{\infty,j}\right\Vert_{L^2(0,L)}.
\end{align}
From Step 1, it follows that $\beta(\rho_{\infty,j})(\partial_s\theta_{\infty,j}-c_0)$ is bounded in $W^{1,2}(0,L)$ uniformly in $j\in\NN$ and hence also in $C([0,L])$. 
By \eqref{eq:lambdarho}, this implies boundedness of $\lambda_\rho(\theta_{\infty,j},\rho_{\infty,j})$ and with that, \eqref{eq:stationary} implies that $\partial_s^2 \rho_{\infty,j}$ is uniformly bounded in $L^2(0,L)$. Now, we know that $\rho_{\infty,j}$ is uniformly bounded in $C^1([0,L])$ and since $\beta(\rho_{\infty,j})\geq \inf_{[-M,M]}\beta$, this implies that $\partial_s^2 \theta_{\infty,j}$ is uniformly bounded in $L^2(0,L)$. It follows that $(\theta_{\infty,j}, \rho_{\infty,j})$ is bounded in $W^{2,2}(0,L)$, independently in $j\in \NN$.\\
\textit{Step 3: $(\theta_{\infty,j},\rho_{\infty,j})\to(\theta_c,\rho_c)$ in $C^1([0,L])$. 
}
Due to Step 2 and \eqref{eq:star}, there exists a (not relabeled) subsequence such that $(\theta_{\infty,j},\rho_{\infty,j})\rightharpoonup(\theta_\infty,\nu)$ in $W^{2,2}(0,L)$ and $(\theta_{\infty,j},\rho_{\infty,j})\rightarrow(\theta_\infty,\nu)$ in $C^1([0,L])$. 
The limit $(\theta_\infty,\nu)$ satisfies the Euler--Lagrage equations \eqref{eq:EL1} and \eqref{eq:EL2}.
Hence, $(\theta_\infty,\nu)$ is a constrained critical point. 
With \Cref{prop:minprobex2} it follows that $\theta_\infty\in C^\infty([0,L])$. 
Further, \Cref{lem:critpointconstrho} implies that $\theta_\infty$ describes an 
elastica. More precisely, \eqref{eq:int theta const} together with \Cref{rem:kappaelastica} yields $\theta_\infty=\frac{2\pi\omega}{L}s=\theta_c$. 
Finally, a standard subsequence argument yields \eqref{eq:C1 conv j to infty}.
\end{proof}

\section{Numerical experiments}
\label{sec:numerics}

\subsection{Newton's method for the gradient flow}
In the case $\omega = 1$ and $c_0 = 0$, a numerical scheme to solve the static minimization problem \eqref{eq:minprob} is proposed in \cite{BJSS2020}. We start by recalling the underlying idea and then explain how this can be extended to approximate solutions to \eqref{eq:flow equation}.

\medskip
\emph{Numerical approximation of the static minimization problem.}
The idea is to approximate the Euler--Lagrange equations \eqref{eq:EL1}-\eqref{eq:EL2} using finite differences, and to solve the resulting system with Newton's method.
We start by explaining the process formally: assuming that we have discretized space, we consider $\hat{\eta} = (\hat\theta, \hat\rho) \in \mathbb{R}^{2N}$,
the piecewise constant approximation of $\eta = (\theta,\rho)$, as well as the corresponding energy $\hat{\mathcal{E}}_\mu$, along with $\hat{\mathcal{E}}^\theta$, $\hat{\mathcal{E}}_\mu^\rho$ and $\hat{\mathcal{G}}$.
More generally, in what follows, a hat marks a space discrete quantity.
We denote the
set of admissible solutions by $\{ \hat{\eta} : \hat{\mathcal{G}}[\hat{\eta}] = 0 \in \mathbb{R}^d\}$, where $d$ is the number of constraints, so that the approximated minimization problem \eqref{eq:minprob} can by written as
\begin{equation}
    \min_{\hat{\mathcal{G}}[\hat{\eta}] = 0} \hat{\mathcal{E}}_\mu[\hat{\eta}]\,.
    \label{eq:discrete minprob}
\end{equation}
The first order optimality conditions are given by
\begin{equation}
\left\{
    \begin{aligned}
        \nabla \hat{\mathcal{E}}_\mu[\hat{\eta}] +  D\hat{\mathcal{G}}[\hat{\eta}]^{\top} \Lambda &= 0\,,
        \\
        \hat{\mathcal{G}}[\hat{\eta}] &= 0\,,
    \end{aligned}
\right.
\end{equation}
where $\Lambda \in \mathbb{R}^d$ are the corresponding Lagrange multipliers.
We can solve this system iteratively:
assuming that the tuple $(\hat{\eta}^j, \Lambda^j)$ is known, we linearize $\hat{\mathcal{E}}_\mu$ and $\hat{\mathcal{G}}$ around $(\hat{\eta}^j, \Lambda^j)$ and get the following system, which is linear in $(\hat{\eta}^{j+1} - \hat{\eta}^j, \Lambda^{j+1} - \Lambda^{j})$:
\begin{equation}
\left\{
    \begin{aligned}
        \left(\nabla^2 \hat{\mathcal{E}}_\mu[\hat{\eta}^{j}] + \Lambda^j D^2 \hat{\mathcal{G}}[\hat{\eta}^{j}]\right)(\hat{\eta}^{j+1} - \hat{\eta}^{j}) &
        \\
        + D \hat{\mathcal{G}}[\hat{\eta}^{j}]^\top(\Lambda^{j+1} - \Lambda^{j}) &=
        - \nabla \hat{\mathcal{E}}_\mu[\hat{\eta}^{j}]
        - D \hat{\mathcal{G}}[\hat{\eta}^{j}]^\top\Lambda^{j}  
        \\
         D \hat{\mathcal{G}}[\hat{\eta}^j] \, (\hat{\eta}^{j+1} - \hat{\eta}^{j}) &= 0\,.
    \end{aligned}
\right. .
\end{equation}

\emph{Extension to the time-dependent problem.}
Here, we use the same underlying idea and De Giorgi's minimizing movements to solve the corresponding $L^2$-gradient flow \eqref{eq:flow equation} numerically.
To do this, we perform a time discretization with time step $\tau$, and consider the corresponding time discrete solution $\hat{\eta}_\tau^{n} = \hat{\eta}_\tau(n\tau)$, which is updated as follows:
\begin{equation}
    \hat{\eta}_\tau^{n+1} \in \argmin\limits_{\hat{\mathcal{G}}(\hat{\eta}_\tau) = 0} \frac{1}{2\tau} \|\hat{\eta}_\tau - \hat{\eta}_\tau^{n}\|_{L^2}^2 + \hat{\mathcal{E}}_\mu[\hat{\eta}_\tau]\,.
    \label{eq:minimov}
\end{equation}

This new minimization problem has the same structure as that of \eqref{eq:discrete minprob}, we can solve it with the method sketched above.
\hl{In terms of time integration, this approach amounts to a one-step implicit Euler method, which in general has good stability properties compared to explicit schemes for parabolic problems. The corresponding nonlinear system is solved using Newton's method, which is a generic approach.}
For $\tau$ and $n$ given, this \hl{yields the following system}, where the index $n$ corresponds to the discretization in time and the index $j$ to the discretization in space:
\begin{equation}
    \small
\left\{
    \begin{aligned}
        \left(\frac{I_{2N}}{\tau}+\nabla^2 \hat{\mathcal{E}}_\mu+ \Lambda_\tau^{n+1,j} D^2 \hat{\mathcal{G}}[\hat{\eta}_\tau^{n+1,j}]\right)(\hat{\eta}_\tau^{n+1,j+1} - \hat{\eta}_\tau^{n+1,j})&
        \\
        + D \hat{\mathcal{G}}[\hat{\eta}_\tau^{n+1,j}]^\top (\Lambda_\tau^{n+1,j+1} - \Lambda_\tau^{n+1,j}) &=
        \\
        - \Big(\frac{\hat{\eta}_\tau^{n+1,j} - \hat{\eta}_\tau^{n}}{\tau}
            + \nabla \hat{\mathcal{E}}_\mu[\hat{\eta}_\tau^{n+1,j}]
                                                                                               &+ D \hat{\mathcal{G}}[\hat{\eta}_\tau^{n+1,j}]^\top \Lambda_\tau^{n+1,j} \Big)
        \\
        D \hat{\mathcal{G}}[\hat{\eta}_\tau^{n+1,j}] \, (\hat{\eta}_\tau^{n+1,j+1} - \hat{\eta}_\tau^{n+1,j})   &= 0\,,
    \end{aligned}
\right.
\end{equation}
where $I_{2N}$ is the identity matrix of size $2N$.
This is a linear system with unknown $(\hat\eta_\tau^{n+1,j+1} - \hat\eta_\tau^{n+1,j}, \Lambda_\tau^{n+1,j+1} - \Lambda_\tau^{n+1,j})$.
The inner loop (i.e.\ the loop in $j$) is initialized by setting $\hat{\eta}_\tau^{n+1,0} = \hat{\eta}_\tau^n$ and finalized with $\hat{\eta}_\tau^{n+1} = \hat{\eta}_\tau^{n+1,j_\infty}$ (and similarly for $\Lambda_\tau^{n+1}$). The iteration is stopped at $j_\infty$, corresponding to $(\hat{\eta}_\tau^{n+1,j_\infty}, \Lambda_\tau^{n+1,j_\infty})$ fulfilling a convergence criterion, typically based on the $L^2$-norm of the residual, i.e.\ the right-hand side.
This system has the form
\begin{multline}
    \small
\left(
\begin{bmatrix}
    I_{2N}  & 0 \\ 0 & 0
\end{bmatrix}
+ \tau
\begin{bmatrix}
    \nabla^2 \hat{\mathcal{E}}_\mu[\hat{\eta}_\tau^{n+1,j}] + \Lambda_\tau^{n+1,j} D^2 \hat{\mathcal{G}}[\hat{\eta}_\tau^{n+1,j}] & D \hat{\mathcal{G}}[\hat{\eta}_\tau^{n+1,j}]^\top \\ D \hat{\mathcal{G}}[\hat{\eta}_\tau^{n+1,j}] & 0
\end{bmatrix}
\right)
\begin{bmatrix}
    \hat{\eta}_\tau^{n+1,j+1} - \hat{\eta}_\tau^{n+1,j}
    \\
    \Lambda_\tau^{n+1,j+1} - \Lambda_\tau^{n+1,j}
\end{bmatrix}
\\
=
-\begin{bmatrix}
    \hat{\eta}_\tau^{n+1,j} - \hat{\eta}_\tau^{n}
    \\
    0
\end{bmatrix}
- \tau
\begin{bmatrix}
    \nabla \hat{\mathcal{E}}_\mu[\hat{\eta}_\tau^{n+1,j}]
    + D \hat{\mathcal{G}}[\hat{\eta}_\tau^{n+1,j}]^\top \Lambda_\tau^{n+1,j}
    \\
    0
\end{bmatrix}.
\nonumber
\end{multline}
\hl{The number of Newton iterations is not fixed beforehand. Rather, the loop is halted when either the norm of the residual (the second term on the right-hand side) or the change in the residual becomes small enough.
In practice, the number of iterations ranges from one to a couple of tens.}

\hl{The time-step $\tau$ is adapted during the (time) iteration. More precisely, $\tau$ is multiplied (resp.\ divided) by a given larger-than-one factor if the difference between two successive iterates (i.e.\ $\| \hat\eta_\tau^{n+1} - \hat\eta_\tau^{n} \|_\infty$) falls below (resp.\ exceeds) a given threshold.}

\subsection{Discretization}

As alluded to earlier, we consider a homogeneous discretization of $[0,L]$ of size $N$, with $s_i = i L/N = i \Delta s$ for $0 \le i < N$. 
We can then define the space discrete functions $\hat{\eta}_{\tau,i} = (\hat{\theta}_{\tau,i},\hat{\rho}_{\tau,i}) = (\hat{\theta}_{\tau}(s_i), \hat{\rho}_\tau(s_i))$.
From the periodicity conditions, we extend the definition to $-1 \le i \le N$ by defining
$\hat{\theta}_{\tau,-1} = \hat{\theta}_{\tau,N-1} - 2\omega\pi$, $\hat{\theta}_{\tau,N} = \hat{\theta}_{\tau,0} + 2\omega\pi$, as well as
$\hat{\rho}_{\tau,-1} = \hat{\rho}_{\tau,N-1}$, $\hat{\rho}_{\tau,N} = \hat{\rho}_{\tau,0}$.
Because of the discontinuity in $\theta$, we also need to define the forward finite difference operator $D_+\hat{\theta}_{\tau}$ with
\begin{equation}
    (D_+\hat{\theta}_\tau)_i =
    \begin{cases}
        \hat{\theta}_{\tau,i+1} - \hat{\theta}_{\tau,i} & 0 \le i < N-1
        \\
        \hat{\theta}_{\tau,0} + 2\pi\omega - \hat{\theta}_{\tau,N-1} & \text{otherwise}\,.
    \end{cases}
\end{equation}
We define the backward (resp. centered) finite difference operator $D_-\hat{\theta}_\tau$ (resp.\ $D_c \hat{\theta}_\tau$) in the same fashion.
The corresponding energy
is
$\hat{\mathcal{E}}_\mu[\hat{\theta}_{\tau},\hat{\rho}_{\tau}]$:
\begin{equation}
    \hat{\mathcal{E}}_\mu[\hat{\theta}_{\tau}, \hat{\rho}_{\tau}]
    = \frac{\Delta s}{2}\sum_{0 \le i < N} \beta(\hat{\rho}_{\tau,i}) \left(\frac{(D_c\hat{\theta}_{\tau})_i}{2 \Delta s}-c_0\right)^2 + \mu \left(\frac{\hat{\rho}_{\tau,i+1} - \hat{\rho}_{\tau,i}}{\Delta s}\right)^2\,,
    \label{eq:discrete energy}
\end{equation}
and its
gradient $\nabla \hat{\mathcal{E}}_\mu = [\nabla_\theta \hat{\mathcal{E}}_\mu\;\; \nabla_\rho \hat{\mathcal{E}}_\mu]^T$ is approximated by:
\begin{equation}
    \small
\begin{aligned}
    \left(\nabla_\theta \hat{\mathcal{E}}_\mu[\hat{\theta}_{\tau}, \hat{\rho}_{\tau}]\right)_i
    &= \Delta s\left[\frac{\beta(\hat{\rho}_{\tau,i-1})+\beta(\hat{\rho}_{\tau,i})}{2} \left(\frac{(D_-\hat{\theta}_\tau)}{\Delta s}-c_0\right) - \frac{\beta(\hat{\rho}_{\tau,i})+\beta(\hat{\rho}_{\tau,i+1})}{2} \left(\frac{(D_+\hat{\theta}_\tau)}{\Delta s}-c_0\right)\right] / \Delta s
    \\
    \left(\nabla_\rho \hat{\mathcal{E}}_\mu[\hat{\theta}_{\tau}, \hat{\rho}_{\tau}]\right)_i
    & =  \Delta s\left[\frac{\beta'(\hat{\rho}_{\tau,i})}{2} \left(\frac{(D_c\hat{\theta}_{\tau})_i}{2 \Delta s}-c_0\right)^2 - \mu \frac{\hat{\rho}_{\tau,i-1} - 2\hat{\rho}_{\tau,i} + \hat{\rho}_{\tau,i+1}}{(\Delta s)^2}\right]\,,
    \label{eq:discrete energy gradient}
\end{aligned}
\end{equation}
where the expression for $\nabla_\theta \hat{\mathcal{E}}_\mu$ is itself a finite difference, so that the divergence structure of the system is preserved at the discrete level.
The constraints are written as:
\begin{equation}
    \hat{\mathcal{G}}(\hat{\theta}_{\tau}, \hat{\rho}_{\tau}) =
    \Delta s 
    \begin{bmatrix}
        \sum_{0 \le i < N} \hat{\rho}_{\tau,i} - \nu L
        \\
        \sum_{0 \le i < N} \sin \hat{\theta}_{\tau,i}
        \\
        \sum_{0 \le i < N} \cos \hat{\theta}_{\tau,i}
    \end{bmatrix}\,,
    \qquad
    \nabla \hat{\mathcal{G}}(\hat{\theta}_{\tau}, \hat{\rho}_{\tau}) =
    \Delta s 
    \begin{bmatrix}
        \hat{\rho}_{\tau}
        &
        \cos \hat{\theta}_{\tau}
        &
        - \sin \hat{\theta}_{\tau}
    \end{bmatrix}\,.
    \label{eq:discrete constraints}
\end{equation}

For the sake of readability, we do not write the Hessian matrices of $\hat{\mathcal{E}}_\mu$ and $\hat{\mathcal{G}}$.

\subsection{Stabilization of \texorpdfstring{$k$}{k}-fold rotationally symmetric solutions}
As shown in Section~\ref{sec:rotsym}, in the case $\omega=1$, ${k}$-fold rotational symmetry is preserved along the flow.
For $k > 1$, such solutions are not numerically stable in general, and roundoff errors might lead to an incorrect asymptotic profile.

More precisely, let us define the real Fourier coefficients for $u = \theta - \phi$ (which we identify with its periodic $L$-extension to $\mathbb{R}$):
\[ a_0^u = \frac{1}{L} \int_0^L u(s) \; \mathrm{d}s\,, \quad a_i^u = \frac{2}{L} \int_0^L u(s) \cos\left(\frac{2 \pi}{L} i s\right) \; \mathrm{d}s \,, \quad b_i^u = \frac{2}{L} \int_0^L u(s) \sin\left(\frac{2 \pi}{L} is\right) \; \mathrm{d}s \,,\]
where $i \ge 1$. The real Fourier coefficients $a_0^\rho$, $a_i^\rho$ and $b_i^\rho$ are defined similarly.

A solution is $k$-fold rotationally symmetric if $u$ and $\rho$ are $L/k$-periodic, i.e.\ if 
\begin{equation}
a_i^u = b_i^u = a_i^\rho = b_i^\rho = 0 \text{ if } i \ge 1 \text{ is not a multiple of } k\,.
\label{eq:kfold fourier condition}
\end{equation}

In these terms, for the solution associated with $k$-fold rotationally symmetric initial datum, it can happen
that the coefficient for some mode $\ell < k$, which is zero initially, becomes nonzero because of roundoff errors.
If this mode is numerically unstable for the choice of parameters considered (esp.\ $\mu$), this mode will grow and break the solution's symmetry.

To address this issue, we define the spaces which satisfy condition \eqref{eq:kfold fourier condition}:
\[ \tilde{V}_{{k}} := \Bigl\{ \vspan\Bigl((\cos(2\ell \pi s_i / L))_i, (\sin(2\ell \pi s_i /L))_i\Bigr) : \exists q \in \mathbb{N} \text{ s.t. } \ell = q \, k \text{ and } \ell \le \lfloor N/2 \rfloor \Bigr\}^2\,, \]
\[ V_{{k}} := \tilde{V}_{{k}} + \text{constants.} \]
Then, in the Newton iteration, instead of setting $\hat{\eta}_\tau^{n+1} = \hat{\eta}_\tau^{n,j_\infty}$, we set
\[ \hat{\eta}_\tau^{n+1} = \hat{\eta}_{\tau,p}^{n,j_\infty} := \Pi_{\tilde{V}_{k}} \hat{\eta}_{\tau}^{n,j_\infty}\,, \]
where the discrete $L^2$-orthogonal projection is done using the Fast Fourier Transform.
Note that we exclude the space spanned by constants
from the projection space, i.e.\ we do not project on $V_{k}$ but on $\tilde{V}_{k}$ (which does not contain constants), since the integrals of both $\theta$ and $\rho$ are preserved along the flow, so that the first Fourier
coefficient of the increment $\hat{\eta}_{\tau}^{n+1} - \hat{\eta}_{\tau}^{n}$ is always zero.

\subsection{Results}
\label{eq:numerical results}

\begin{table}[h]
    \footnotesize
    \begin{center}
    \begin{tabular}{lccccccccc}
        \toprule
        Description & Figure & $N$ & $L$ & $\mu$ & $c_0$ & $\beta(x)$ & $\nu$ & $\omega$ & $k$\\
        \midrule
        Loss of convexity   & \ref{fig:loss of convexity} & 1440 & $2\pi$ &$10^{-1}$ & 1 & $e^{-x}$ & $0$  & $1$ & $2$ \\
        Loss of embeddedness for $c_0 > 2\pi/L$ & \ref{fig:energy high c0}         & 720& $2\pi$ &  $10^{-3}$ & 3 & $e^x$ & 0 & $1$ & $2$ \\
        Loss of embeddedness for $c_0 = 0$& \ref{fig:loss embeddedness} & 1440 & $2\pi$&  $1$ & $0$ & $0.03 + b x^2$ & $1/\pi$  & $1$ & $2$ \\
        $\hat{\mathcal{E}}_\mu(t)$, $c_0 = 0$, low $\|\hat\kappa_0 - 1\|_{\infty}$ & \ref{fig:energy low curvature}   & 720& $2\pi$ &  $10^{-3}$ & 0 & $e^x$ & 0 & $1$ & $5$ \\
        $\hat{\mathcal{E}}_\mu(t)$, $c_0 = 0$, high $\|\hat\kappa_0 - 1\|_{\infty}$ & \ref{fig:energy high curvature} & 1440& $2\pi$ &  $10^{-3}$ & 0 & $e^x$ & 0 & $1$ & $5$ \\
        $\hat\theta_0$ $\frac{L}{10}$-periodic, $\hat\rho_0$ $\frac{L}{20}$-periodic & \ref{fig:10-fold wave number rho 20} & 1440& $2\pi$ &  $10^{-3}$ & 0 & $e^x$ & 0 & $1$ & $10$ \\
        $(\hat\theta_0, \hat\rho_0)$ as $\mu$ increases & \ref{fig:omega 2} & 420& $2\pi$ &  $10^{-2}$ to $5$ & 0 & $e^{x}$ & $0$ & $2$ & - \\
        $(\hat\theta_0, \hat\rho_0)$ as $\beta'$ decreases & \ref{fig:omega 2} & 420& $2\pi$ &  $10^{-2}$ & 0 & $e^{a x}$ & $0$ & $2$ & - \\
        Figure eight & \ref{fig:omega 0 figure 8} & 720& $2\pi$ &  $10^{-1}$ & 2 & $0.1 + x^2$ & $0$ & $0$ & - \\
        $2$-fold figure eight & \ref{fig:omega 0 figure 8 double} & 1440 & $2\pi$&  $10^{-1}$ & 0 & $0.1 + x^2$ & $0$ & $0$ & - \\
        \bottomrule
    \end{tabular}
    \end{center}
    \caption{Parameters used in the figures below. In the last column, $k$ is given when the initial datum $(\rho_0, c_0)$ is $k$-fold rotational symmetric.}
    \label{tab:parameters}
\end{table}

Here we give some example behavior of the solutions,
\hl{the implementation and configuration files used for most figures are available \href{https://github.com/gjankowiak/dAJLR.2024}{online}\footnote{\texttt{https://github.com/gjankowiak/dAJLR.2024} (code licensed under the GPLv3)}}.
First, in the case $\omega = 1$, we look at the possible loss of convexity and simplicity of the corresponding curve.
Then, we give examples of the time evolution of the energies $\hat{\mathcal{E}}$, $\hat{\mathcal{E}}^\theta$ and $\hat{\mathcal{E}}_\mu^\rho$ for small $\mu$, where we observe metastable energy plateaus.
In this case, the bending energy $\hat{\mathcal{E}}^\theta$ makes up most of the total energy $\hat{\mathcal{E}}_\mu$.
Then, in the case $\omega = 2$, we illustrate how the choice of $\beta$ and $\mu$ can impact the limiting profile $(\hat{\theta}_\infty, \hat{\rho}_\infty)$.
Finally, for $\omega = 0$ and nonzero $c_0$, we look at the convergence of two curves, the first to the figure eight and the second to the $2$-fold covering of the figure eight.
All examples here correspond to $L = 2\pi$. A quick overview of the corresponding figures is given in \Cref{tab:parameters}.
\subsubsection{Loss of convexity}
\label{sec:loss convexity}

\medskip

In this section, we present a simple example illustrating the loss of convexity discussed in \Cref{ex:counterconvex}.
As initial datum, we consider $\hat\theta_0$ to be the discretization of a stadium of aspect ratio roughly equal to 1:5. $\hat{\rho}_0$ is (the discretization of) a cosine function of amplitude $1$.
We take $\beta : x \mapsto e^{-x}$, and fix the parameters $c_0 = 1$ and $\mu = 10^{-1}$.
This situation is illustrated in \Cref{fig:loss of convexity}, where the loss of convexity is visible at time $t \approx 0.05$.
We must note that, here, unlike in \Cref{ex:counterconvex},
the initial datum $\hat{\rho}_0$ is not linear on the flat sides of the stadium.

\begin{figure}[h]
    \begin{center}
        \begin{tabular}{lc}
            \raisebox{0.5cm}{$t = 0$} & \includegraphics[width=6cm]{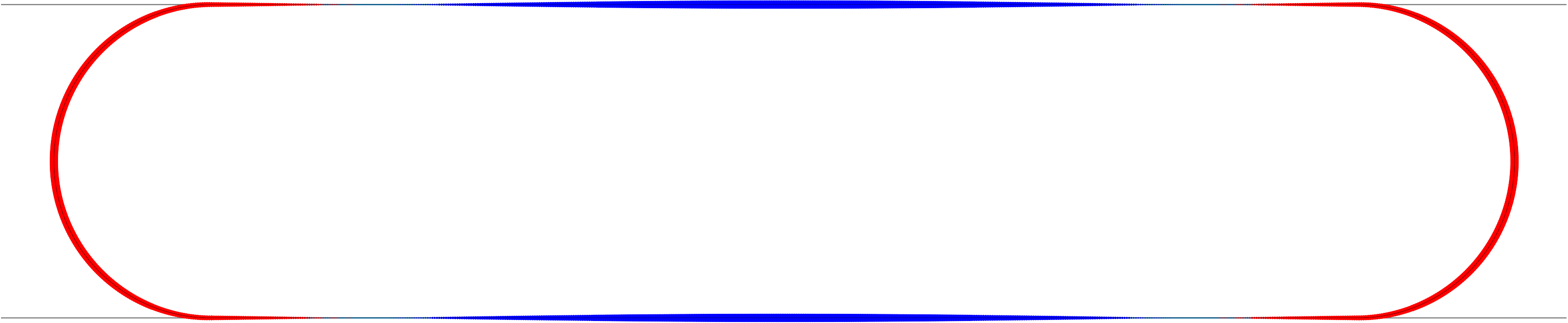} \\
            \raisebox{2.1cm}{$t \approx 0.05$} & \includegraphics[width=6cm]{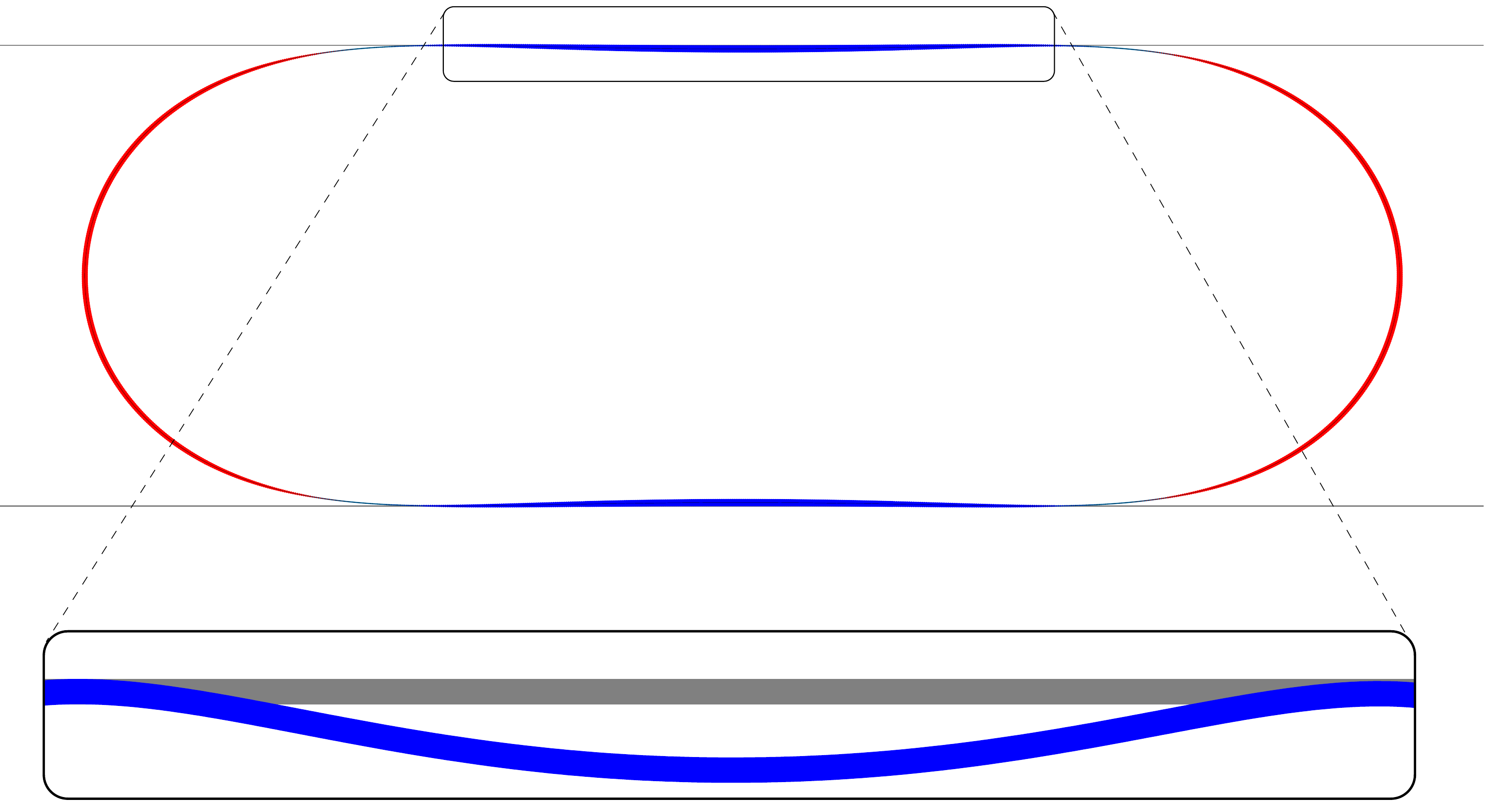} \\
        \end{tabular}
    \end{center}
    \caption{Loss of convexity of a stadium with sides parallel to the $x$-axis. Here and in subsequent figures, the width of the stroke increases with $|\hat\rho|$. For the sake of readability, the y-scale is amplified 10 times and the curve is shown with constant width in the inset. The gray lines are the tangents parallel to the $x$-axis, for reference.
    Here and in all the following figures, positive values of $\rho$ are shown in blue, negative values in red.
    It is not shown here, but the curve becomes convex again at later times.}
    \label{fig:loss of convexity}
\end{figure}

By taking a very elongated stadium, it is reasonable to believe that the corresponding curve will not only lose convexity but also simplicity. In practice, it is difficult to show this behavior because of our choice of discretization, which enforces a homogeneous distribution of the nodes.
An elongated stadium would require a very large $N$ to resolve the rounded ends in a satisfying way. Instead, in the following, we choose a different initial condition, which also leads to loss of simplicity.

\begin{figure}[h]
    \begin{minipage}{0.19\textwidth}
    \begin{center}
        \includegraphics[width=\textwidth]{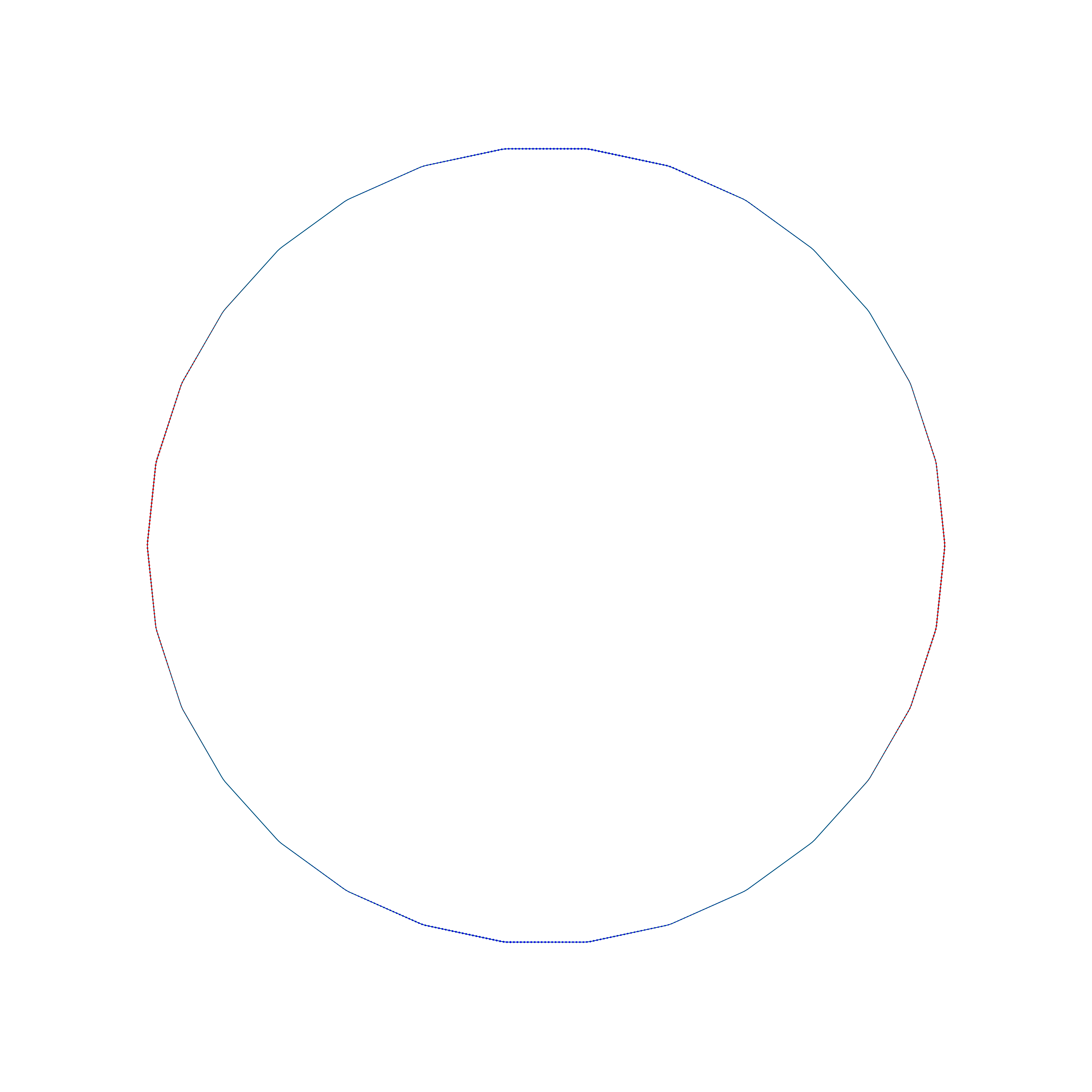}
        $t = 0$
    \end{center}
    \end{minipage}
    \begin{minipage}{0.19\textwidth}
    \begin{center}
        \includegraphics[width=\textwidth]{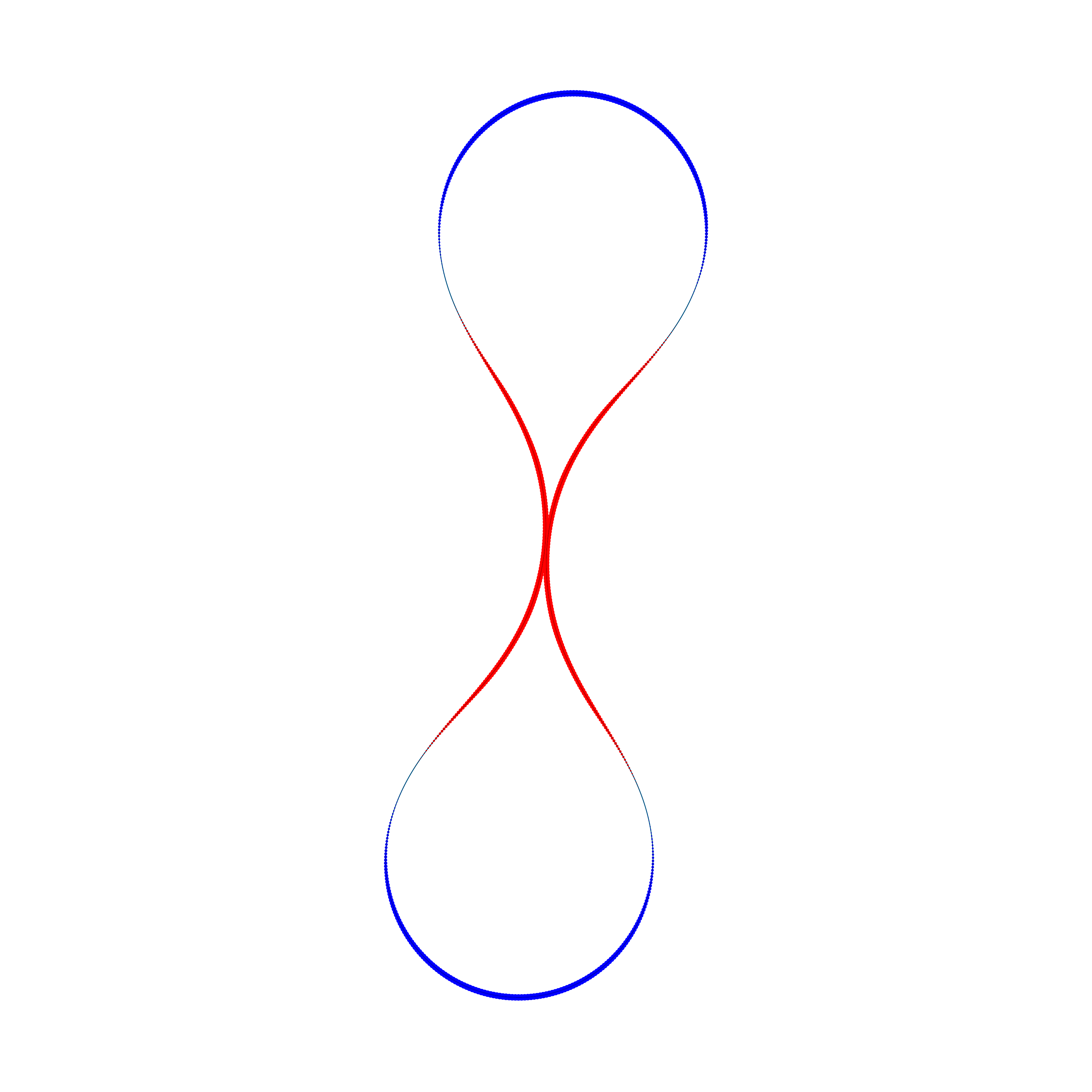}
        $t \approx 1.82$
    \end{center}
    \end{minipage}
    \begin{minipage}{0.19\textwidth}
    \begin{center}
        \includegraphics[width=\textwidth]{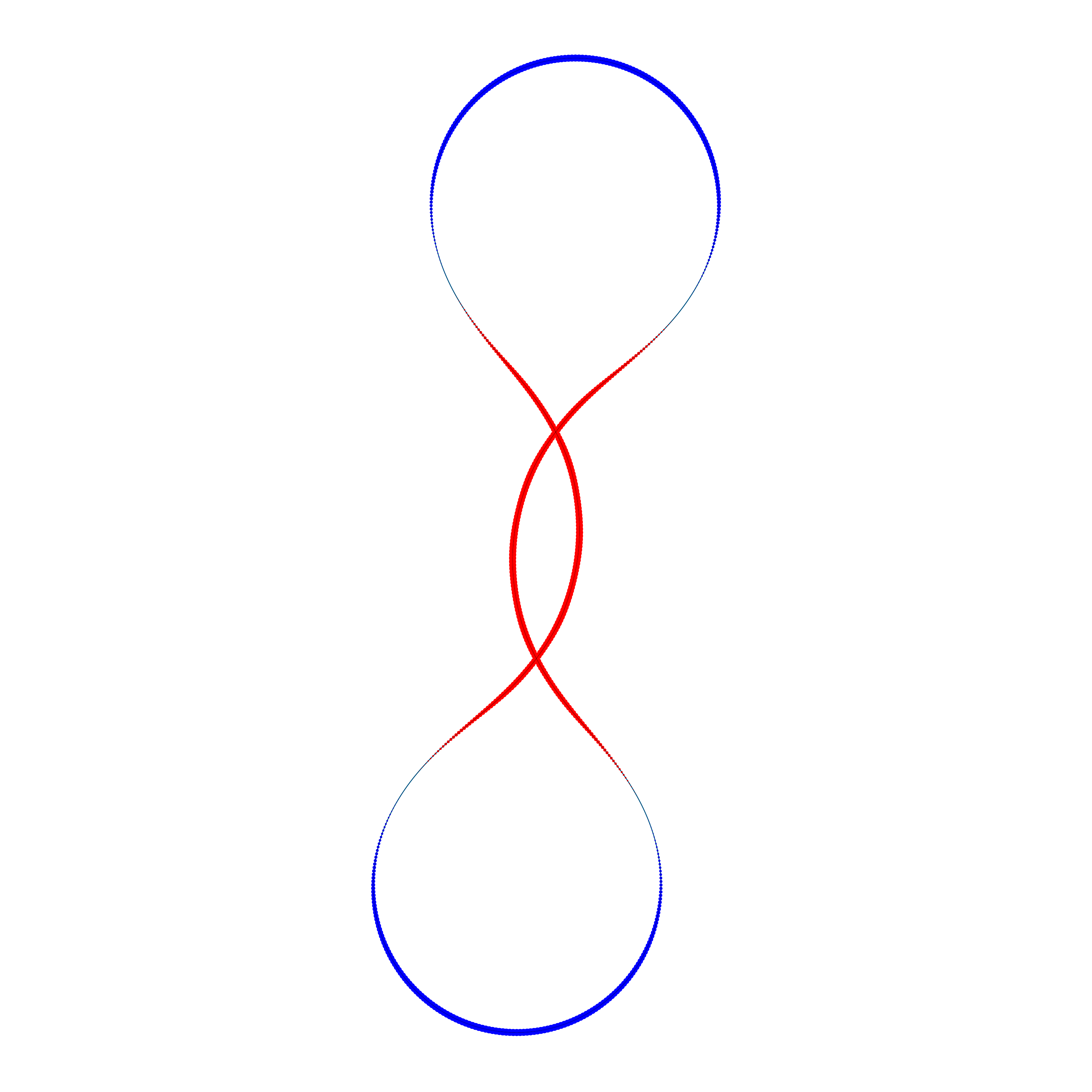}
        $t \approx 2$
    \end{center}
    \end{minipage}
    \begin{minipage}{0.19\textwidth}
    \begin{center}
        \includegraphics[width=\textwidth]{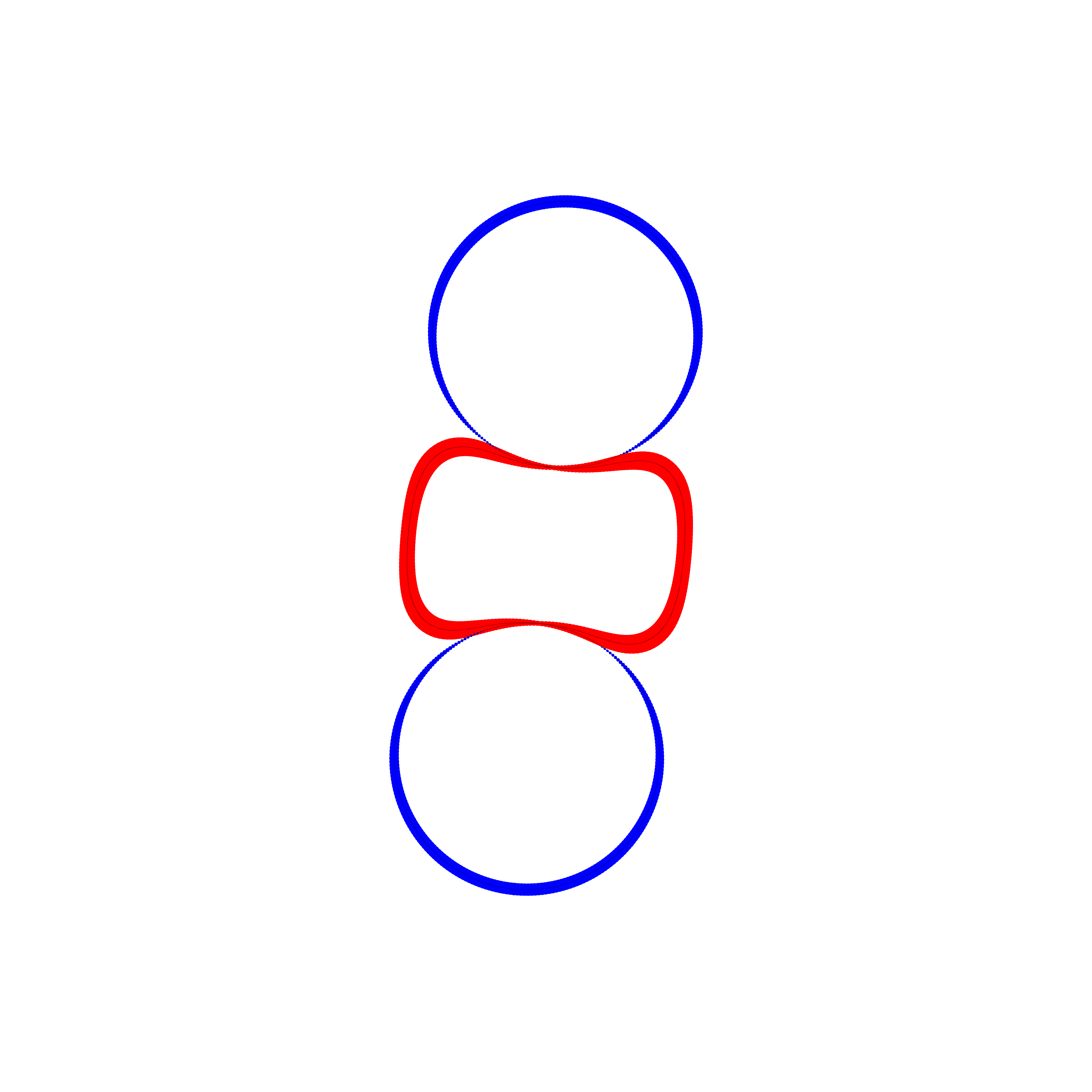}
        $t \approx 8$
    \end{center}
    \end{minipage}
    \begin{minipage}{0.19\textwidth}
    \begin{center}
        \includegraphics[width=\textwidth]{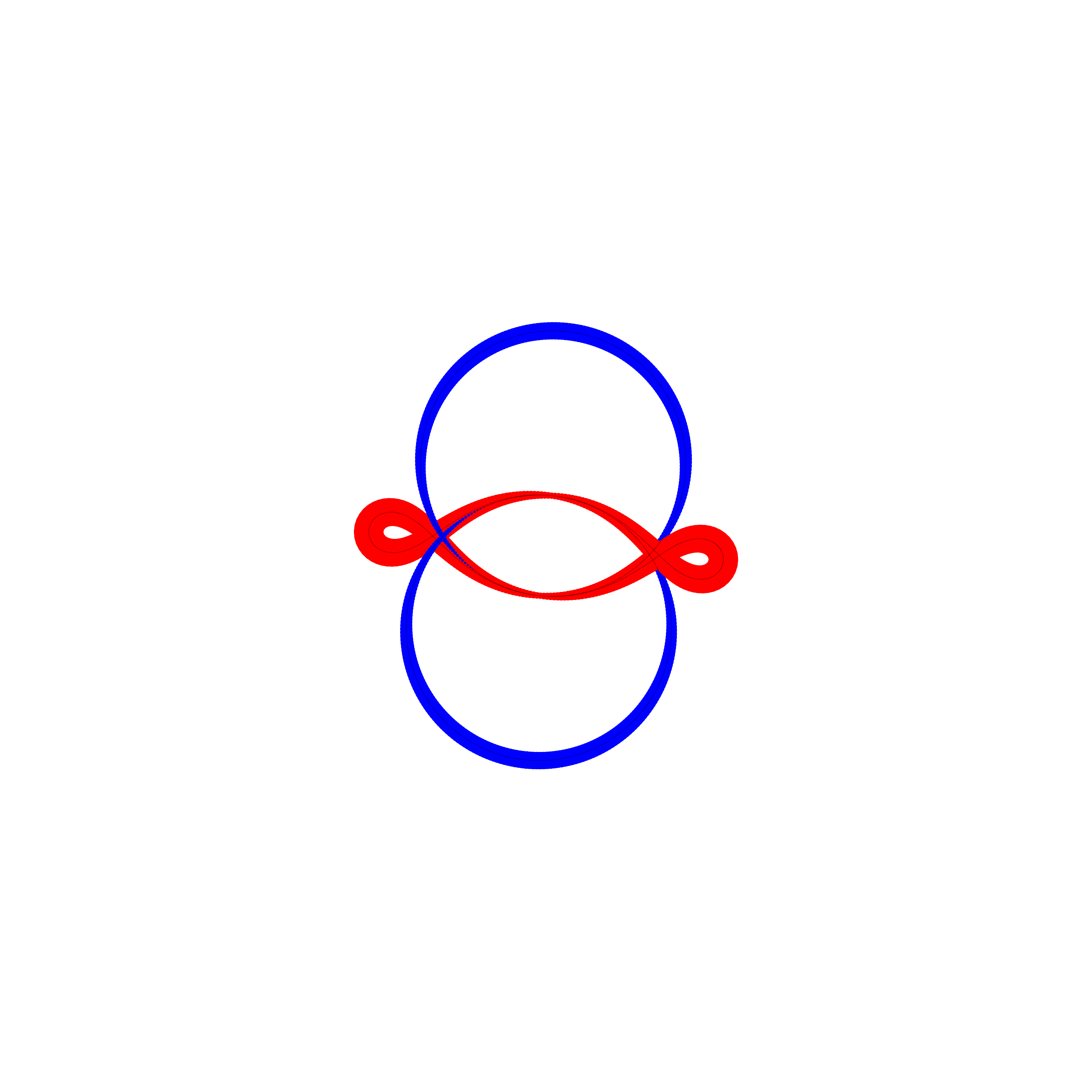}
        $t \approx 6500$
    \end{center}
    \end{minipage}
    \caption{Snapshots of the solution showing the loss of embeddedness for $c_0 = 3$, starting from a convex initial datum.}
    \label{fig:energy high c0}
\end{figure}

\bigskip

\subsubsection{Simplicity (or embeddedness) along the flow}
\label{sec:loss embeddedness}

In the case $\omega=1$, we now investigate the possible loss of embeddedness of the curve along the evolution. We look at two different situations: first, starting with an initial datum corresponding to a convex curve, with $c_0 > 2\pi / L$.
Second, for $c_0 = 0$, we carefully choose the initial datum such that $\hat\theta_0$ corresponds to an embedded curve with a narrow neck. As the curve evolves, the sides of this neck come closer together and eventually cross.

\medskip
\indent\emph{Loss of embeddedness, first case.} We start with $c_0 = 3 > 2\pi / L$, $\hat\theta_0$ close, but not equal, to $\hat\theta_c$, and
$\hat\rho_0$ corresponding to $\cos (4\pi / L s)$, so that the initial datum is the discretization of a $2$-fold rotationally symmetric curve. The solution at different times is drawn in \Cref{fig:energy high c0}. The associated curve loses convexity and then embeddedness at $t \approx 1.82$, and the solution stays $2$-fold rotationally symmetric, which is expected from the results of \Cref{sec:rotsym}.

\medskip

\indent\emph{Loss of embeddedness, second case.} In \Cref{prop:li_yau_emb}, the preservation of embeddedness of the curve described by $\theta$ is proven for $\omega = 1$, provided the initial energy is small enough.
Here, we provide a numerical example for which the discrete energy $\hat{\mathcal{E}}_{\mu}(\hat\theta_0, \rho_0)$ is above the threshold given by \Cref{prop:li_yau_emb}, and for which embeddedness is lost along the flow.

To do so, we consider the choice of parameters $\beta : x\mapsto 0.03 + b x^2, b > 0$ and $c_0=0$. The energy threshold in \Cref{prop:li_yau_emb} is then
\[ \varepsilon := \frac{\inf \beta}{2} \left(\frac{C_{2T}}{L} - 4\pi c_0 + L c_0^{2}\right) \approx 0.35\,, \]
where we recall that we chose $L = 2\pi$, and that it holds $C_{2T} \approx 146.628$.

As initial datum, we pick a curve which can be described as consisting of two lateral drop-shaped lobes which are connected by a long, narrow neck.
The inital datum $\hat\rho_0$ is chosen positive and distributed in the concave parts
of the lobes. Heuristically, the concave parts of the lobes concentrate a large part of the energy, and will quickly be ``flattened'' by the flow,
making the two sides of the middle channel cross. We show numerically that this crossing does occur for $b = 8$.

A representation of the corresponding curve and initial distribution $\hat{\rho}_0$ is shown in Figure~\ref{fig:simplicity init} and~\ref{fig:simplicity init zoom}.
For $b=8$, we have $\hat{\mathcal{E}}_\mu(\hat\theta_0, \hat\rho_0) \approx 16 > \varepsilon$. Loss of embeddedness occurs at $t_1 \approx 4.1 \times 10^{-3}$ with $\hat{\mathcal{E}}_\mu(\hat\theta(t_1), \hat\rho(t_1)) \approx 10$.
      The curve becomes simple again at $t_2 \approx 2.91 \times 10^{-2}$ with $\hat{\mathcal{E}}_\mu(\hat\theta(t_2), \hat\rho(t_2)) \approx 4.9 > \varepsilon$. See Figure~\ref{fig:simplicity uncross zoom}.
        Note that at time $t_2$ the energy is still one order of magnitude larger than the threshold given by Proposition~\ref{prop:li_yau_emb}.
        Simplicity is then kept for $t > t_2$.

\begin{figure}[h!]
    \centering
    \begin{minipage}{0.43\textwidth}
    \begin{subfigure}[l]{\textwidth}
        \subcaption{$t=0$.}
        \vspace{0.2cm}
        \includegraphics[width=\textwidth]{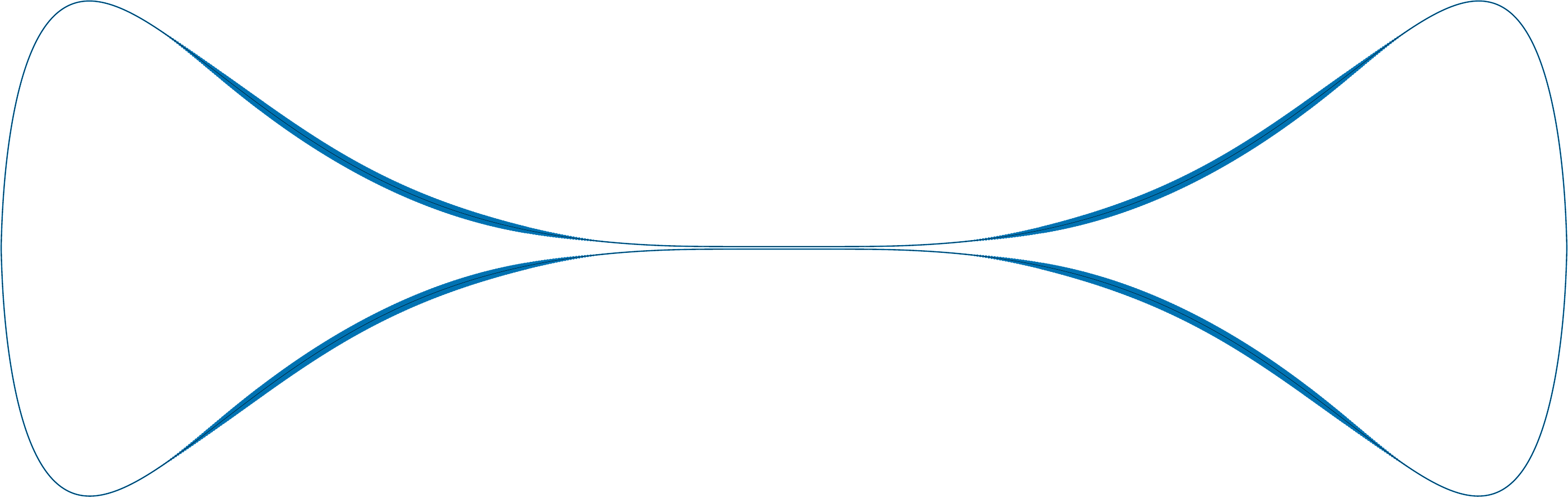}
        \label{fig:simplicity init}
    \end{subfigure}
    \begin{subfigure}{\textwidth}
        \subcaption{$t=0$, detail, $y$-scale amplified, $\hat{\rho}$ not shown. In the middle, the vertical size of the gap is $2p = 10^{-2}$.}
        \vspace{0.2cm}
        \includegraphics[width=\textwidth, height=0.68cm]{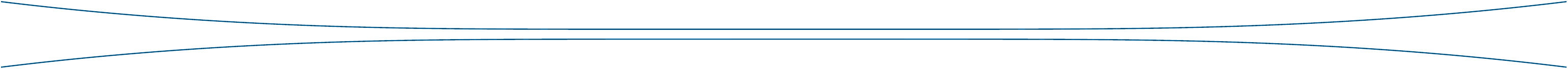}
        \label{fig:simplicity init zoom}
    \end{subfigure}
    \end{minipage}
    \hfill
    \begin{minipage}{0.43\textwidth}
    \begin{subfigure}{\textwidth}
        \subcaption{$t\approx0.01 > t_1$, same scale as (B).}
        \vspace{0.2cm}
        \includegraphics[width=\textwidth,height=1cm]{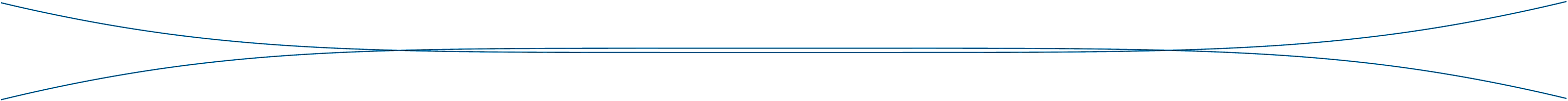}
        \label{fig:simplicity cross zoom}
    \end{subfigure}
    \begin{subfigure}{\textwidth}
        \subcaption{$t\approx0.04 > t_2$, same scale as (B).}
        \vspace{0.2cm}
        \includegraphics[width=\textwidth,height=2.6cm]{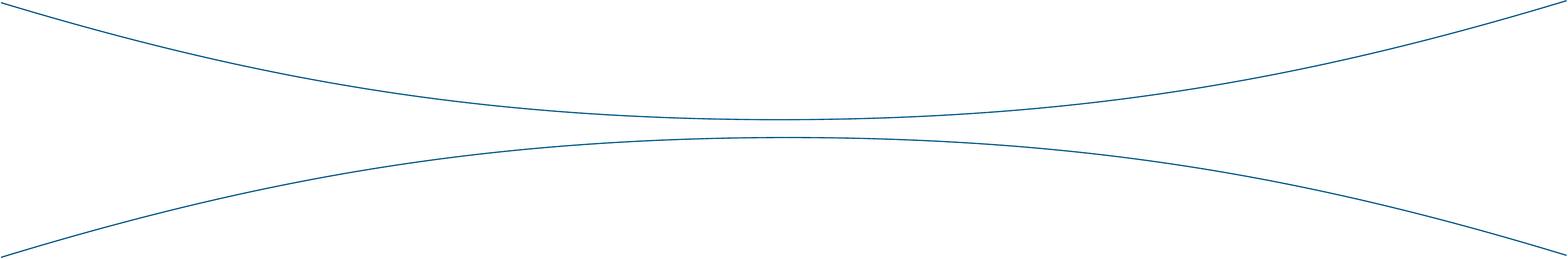}
        \label{fig:simplicity uncross zoom}
    \end{subfigure}
    \end{minipage}
    \caption{Snapshots showing that a simple curve does not need to stay simple along the flow, even for $c_0 = 0$.
    The choice of parameters is that of Section~\ref{sec:loss embeddedness}, here with $b=8$.}
    \label{fig:loss embeddedness}
\end{figure}


\subsubsection{Evolution of the energy for $\omega=1$, conservation of symmetry}
\label{sec:evolution energy}

Here, we provide some examples of the time evolution of the different components of the energy, along with characteristic shapes of the solution as well as the associated density distribution $\hat\rho$.
We are interested in cases where the solutions display a relatively rich behavior, so we choose $\mu$ small, namely $\mu = 10^{-3}$. The bending stiffness $\beta$ is chosen as $\beta(x) = e^x$, and we take zero total mass $\nu L$.

We consider two choices for the initial datum:
\begin{itemize}
    \item The first with $c_0 = 0$, with low initial energy $\hat{\mathcal{E}}_\mu$, with $\hat\kappa_0$ close to $2\pi/L = 1$ and $\hat\rho_0$ almost constant, see \Cref{fig:energy low curvature}.
  \item The second with $c_0 = 0$, with high initial energy, with $\hat\kappa_0$ and $\hat\rho_0$ oscillating significantly, see \Cref{fig:energy high curvature}.
\end{itemize}

In the two cases, both $\hat\theta_0$ and $\hat\rho_0$ are $L/5$-periodic, so that the initial datum is $5$-fold rotationally symmetric, and the symmetry preserving results of~\Cref{sec:rotsym} apply.
In the second case, $\hat\theta_0$ is not only $L/5$-periodic, but also $L/10$-periodic.
The corresponding results are shown in~\Cref{fig:energy low curvature} and \Cref{fig:energy high curvature}, respectively.
Because of the metastable nature of the evolution, both the time and the energy scales are logarithmic.
Since $\mu$ is small, with our choice of parameters, the main contribution to the initial energy $\hat{\mathcal{E}}_\mu(\hat\theta_0, \hat\rho_0)$  comes from the bending energy $\hat{\mathcal{E}}^\theta$.

\begin{figure}
    \includegraphics[
    width=0.9\textwidth]{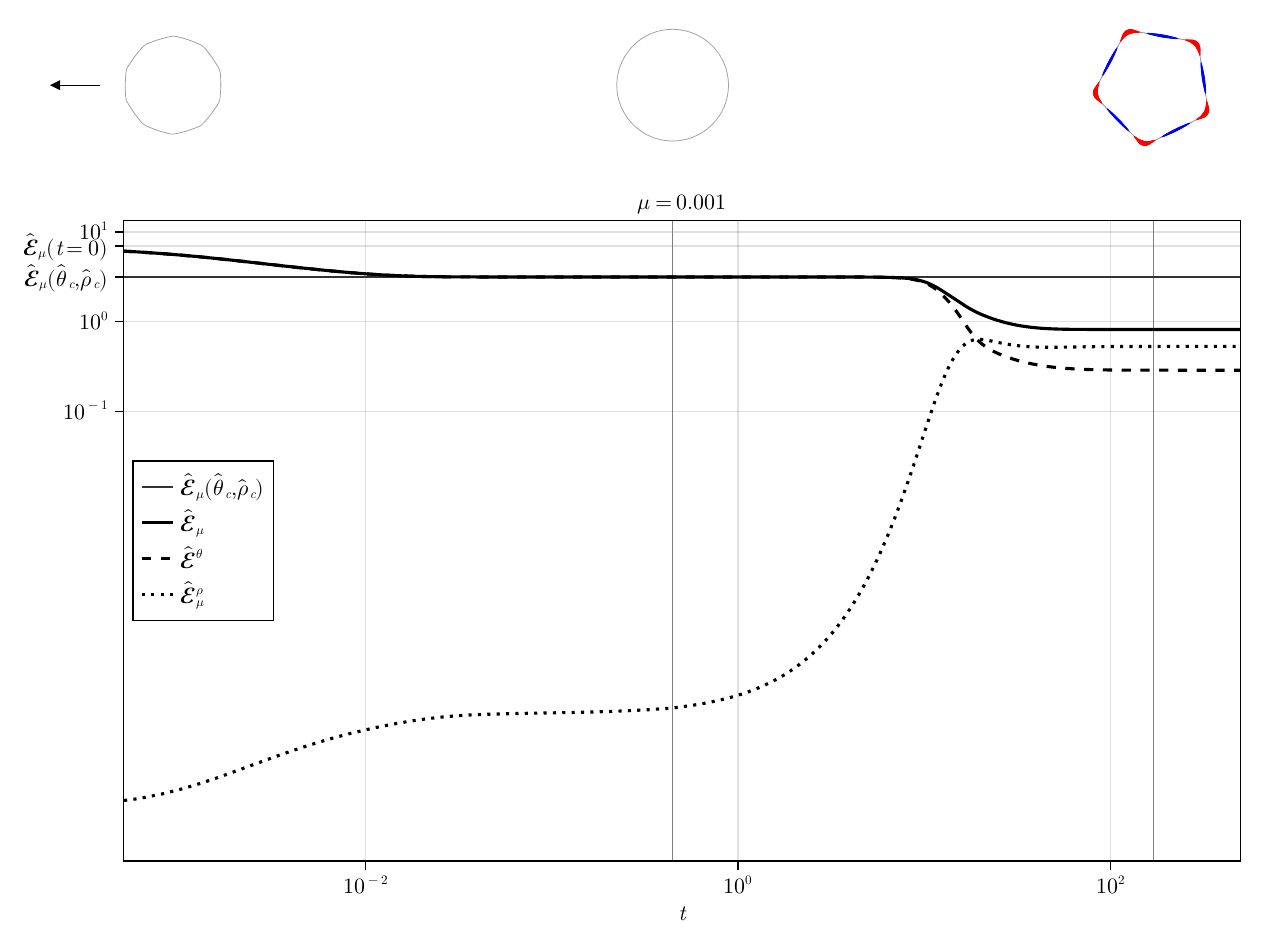}
    \caption{Energy evolution for $c_0 = 0$, starting with relatively low $\|\hat\kappa - 1\|_\infty$ and $5$-fold rotational symmetry.}
    \label{fig:energy low curvature}
\end{figure}

\begin{figure}
    \includegraphics[
    width=0.9\textwidth]{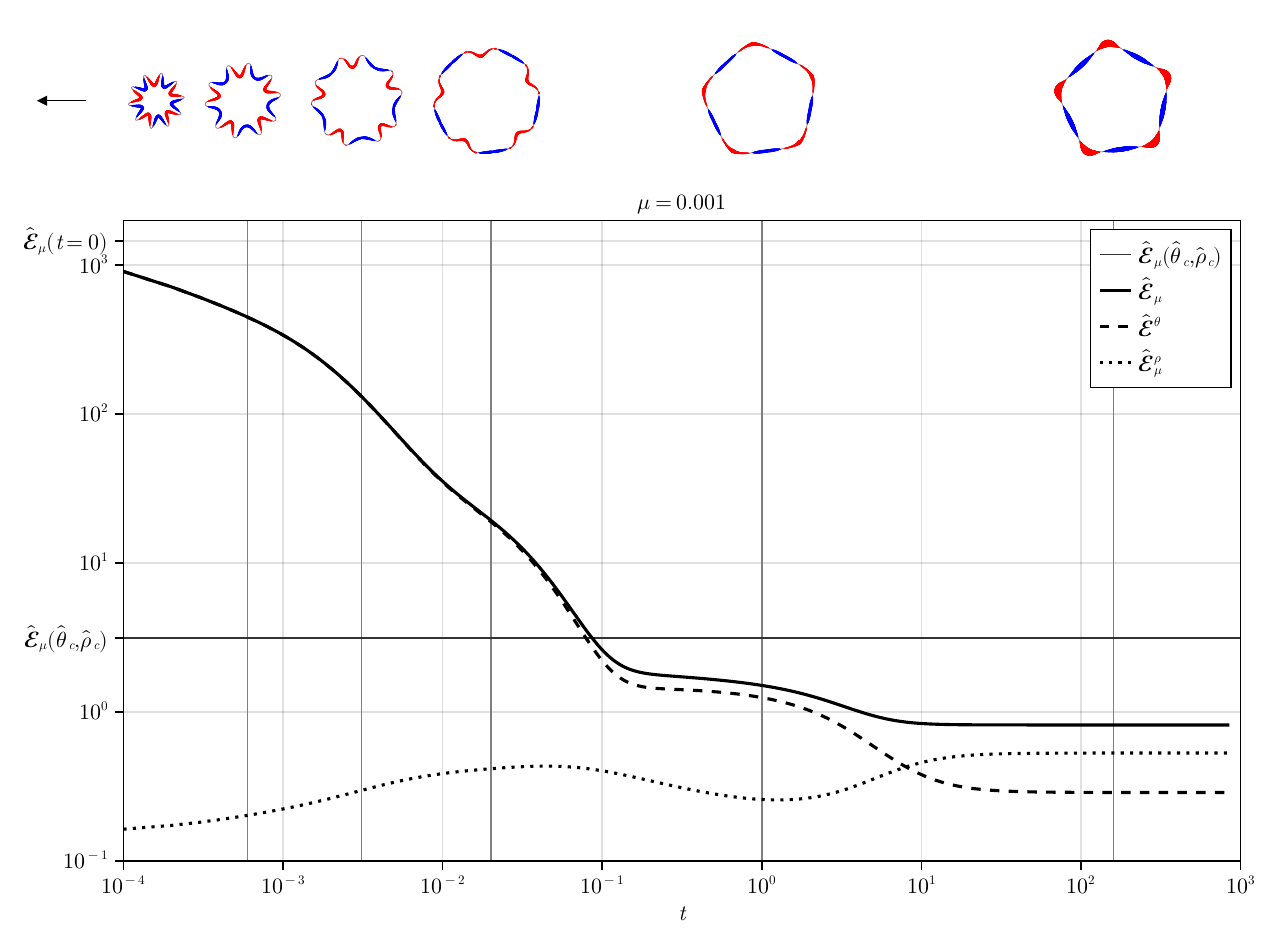}
    \caption{Energy evolution for $c_0 = 0$, starting with relatively large $\|\hat\kappa - 1\|_\infty$ and $5$-fold/$10$-fold rotational symmetry.}
    \label{fig:energy high curvature}
\end{figure}

The shapes of the corresponding curves are naturally rather different: in the first case, the solution goes close to the trivial state $(\hat\theta_c, \hat\rho_c)$
and spends some time there before changing to a pentagon-like curve, with positive values of $\hat\rho$ on the flat ``sides'' and negative values on the rounded ``corners'',
which is in accordance with $\nu L = 0$ and $\beta$ monotone increasing. 

In the second case, the solution does not come close to the trivial state, and the ``dents'' of the initial conditions coarsen, so that $\hat\theta$ goes from being $L/10$-periodic to $L/5$-periodic.
Eventually, $(\hat\theta_\infty, \hat\rho_\infty)$ is identical to the first case, up to rotation.

This seems to suggests that, for this choice of parameters, the periodicity of the limiting profile is dictated by the choice of $\hat\rho_0$. However, one can take $\hat\rho_0$ to be $L/20$-fold periodic by keeping all other parameters as in \Cref{fig:energy high curvature}, so that the initial datum is then $10$-fold rotationally symmetric. The final profile is observed to be also $L/10$-symmetric, as shown in \Cref{fig:10-fold wave number rho 20}, matching the periodicity of $\hat\theta_0$ and not that of $\hat\rho_0$.

\begin{figure}
    \centering
    \includegraphics[width=0.2\textwidth]{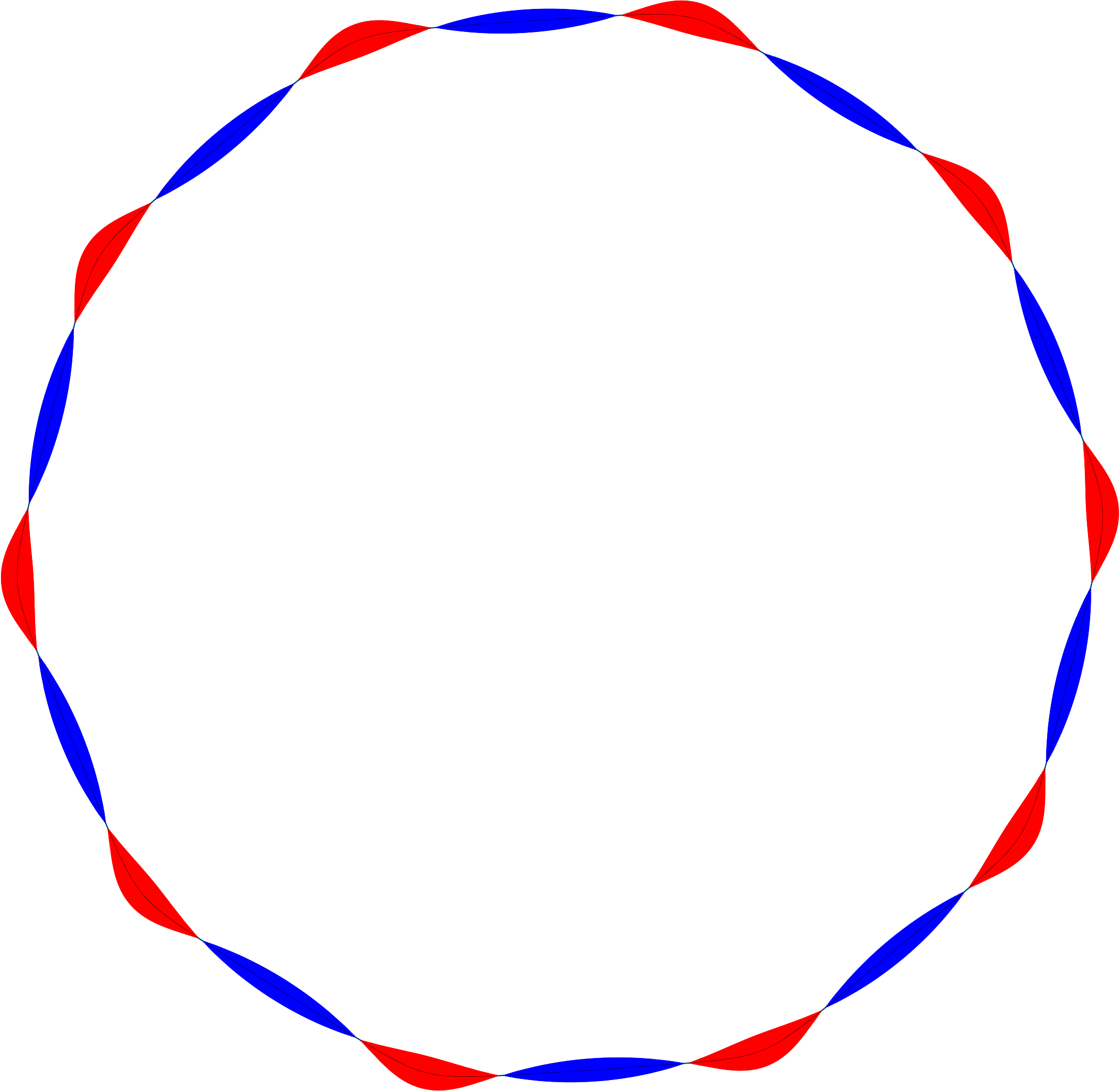}
    \caption{Representation of the $L/10$-periodic limiting profile of the solution with parameters identical to those of \Cref{fig:energy high curvature}, except that $\hat\rho_0$ is $L/20$-periodic.
    As in \Cref{fig:energy high curvature}, $\hat\theta_0$ is $L/10$-periodic.}
    \label{fig:10-fold wave number rho 20}
\end{figure}

This shows numerically that for $\mu$ small enough, there are $k$-fold rotationally symmetric critical points different from the homogeneous elastica, extending the picture drawn by \Cref{prop:convlargemu}.

\subsubsection{Influence of the model parameters for $\omega = 2$}

In~\Cref{fig:omega 2}, starting from the same initial datum, we illustrate how the shapes assumed by the limit $(\hat\theta_\infty, \hat\rho_\infty)$
change as the model parameters $\mu$ and $\beta$ change. The spontaneous curvature $c_0$ is taken to be zero, and $\beta$ is of the form $\beta(x) = e^{ax}$.

For relatively small values of $\mu$ and large values of $a$, the solution is cigar- (or stadium-) shaped, with a small additional loop at one end, which accounts for $\omega = 2$, see~\Cref{fig:omega 2}A and~\ref{fig:omega 2}D.
In those cases, $\hat{\mathcal{E}}^\theta$ can be made small for a curve with flat sections corresponding to large values of $\beta(\hat\rho)$ (and highly curved sections corresponding to small values) without making $\hat{\mathcal{E}}_\mu^\rho$ large thanks to the small value of $\mu$.

As $\mu$ increases, large values in the gradient of $\hat\rho$ are penalized and the geometric part dominates, so that the curve becomes rounder as a result, cf.\ \Cref{fig:omega 2}B.

Increasing $\mu$ further, in view of \Cref{prop:convlargemu}, it seems plausible that the solution should converge to $(\hat\theta_c, \hat\rho_c)$, although here
the density $\hat\rho_0$ is not constant. This is what can be observed in~\Cref{fig:omega 2}C.

For what concerns $a$, i.e.\ $\beta'$, as it gets smaller (with $\mu$ kept small), the gain in $\hat{\mathcal{E}}^\theta$ coming from a given oscillation of density distribution diminishes, 
so the oscillation increases, see \Cref{fig:omega 2}E. Eventually, as $a$ becomes very small, this is balanced by the increase in $\hat{\mathcal{E}}_\mu^\rho$, and the solution converges to the trivial state $(\hat\theta_c, \hat\rho_c)$, see \Cref{fig:omega 2}F.

\begin{figure}

    \begin{tikzpicture}[scale=0.9]

        \coordinate (center) at (0,0);

        \coordinate (a) at (120:4cm);
        \coordinate (b) at (180:4cm);
        \coordinate (c) at (240:4cm);

        \coordinate (d) at (60:4cm);
        \coordinate (e) at (0:4cm);
        \coordinate (f) at (300:4cm);

        \node[draw=none,fill=none] at (center){\includegraphics[angle=30,width=5cm]{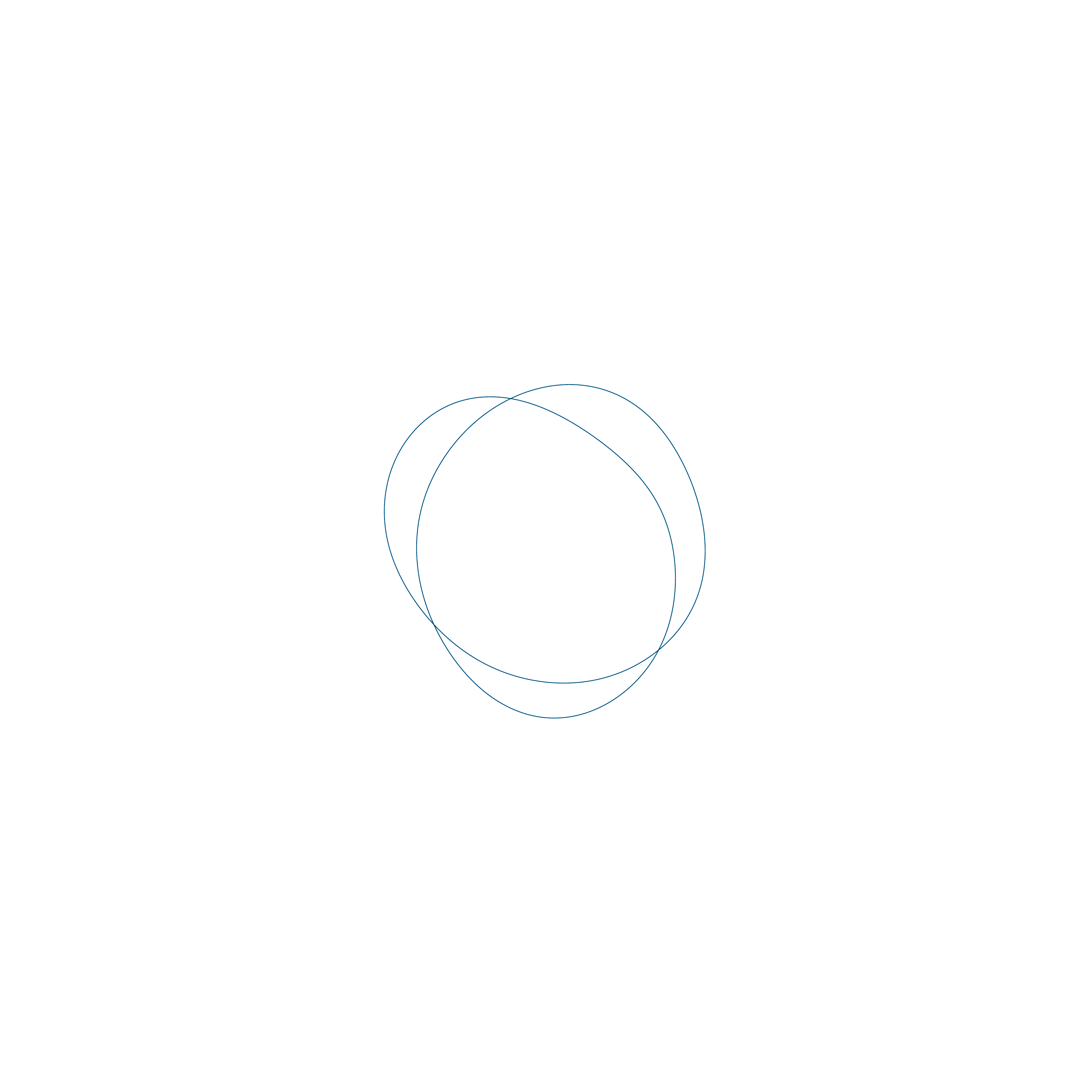}};
        \node[draw=none,fill=none] at (a){\includegraphics[angle=30,width=5cm]{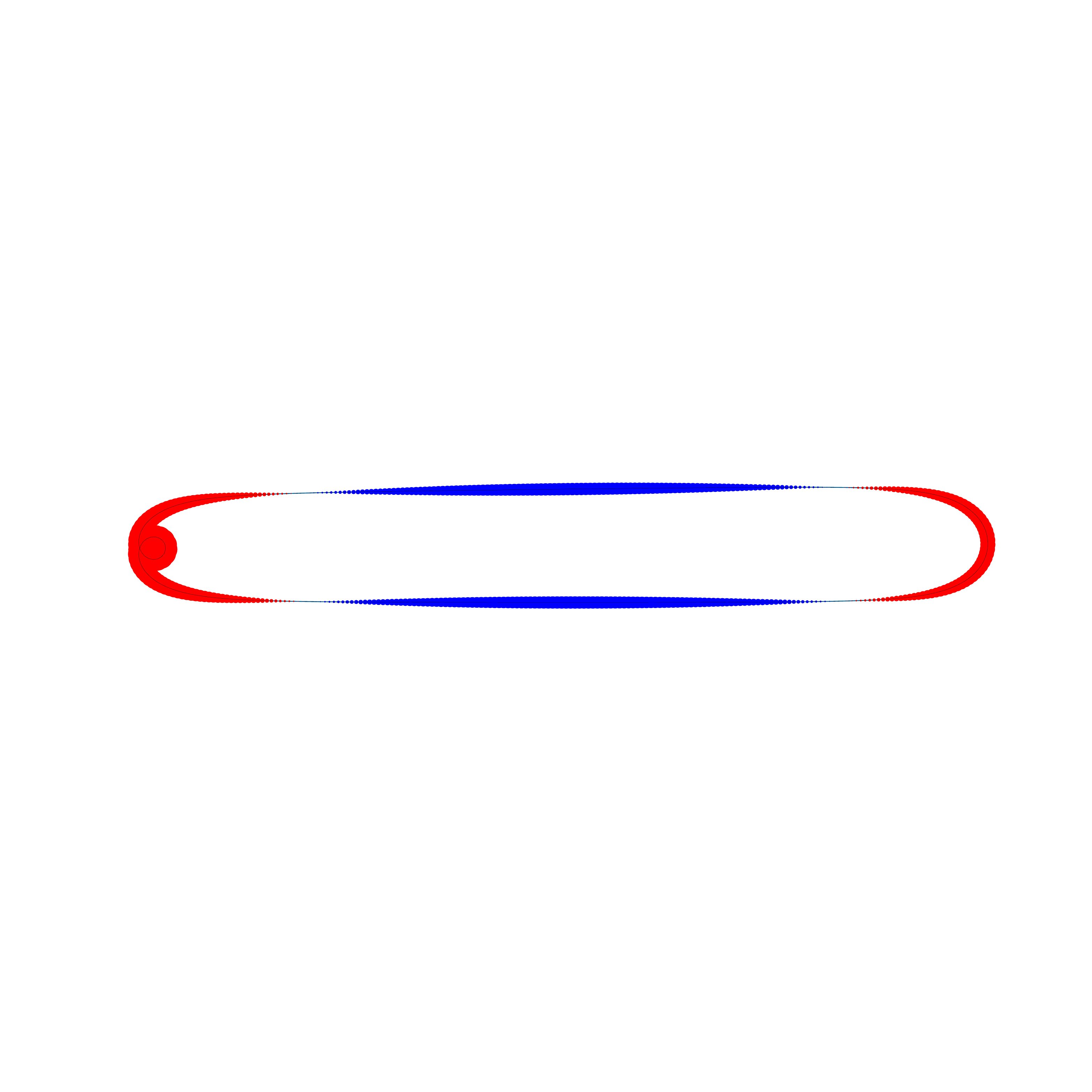}};
        \node[draw=none,fill=none] at (b){\includegraphics[angle=30,width=5cm]{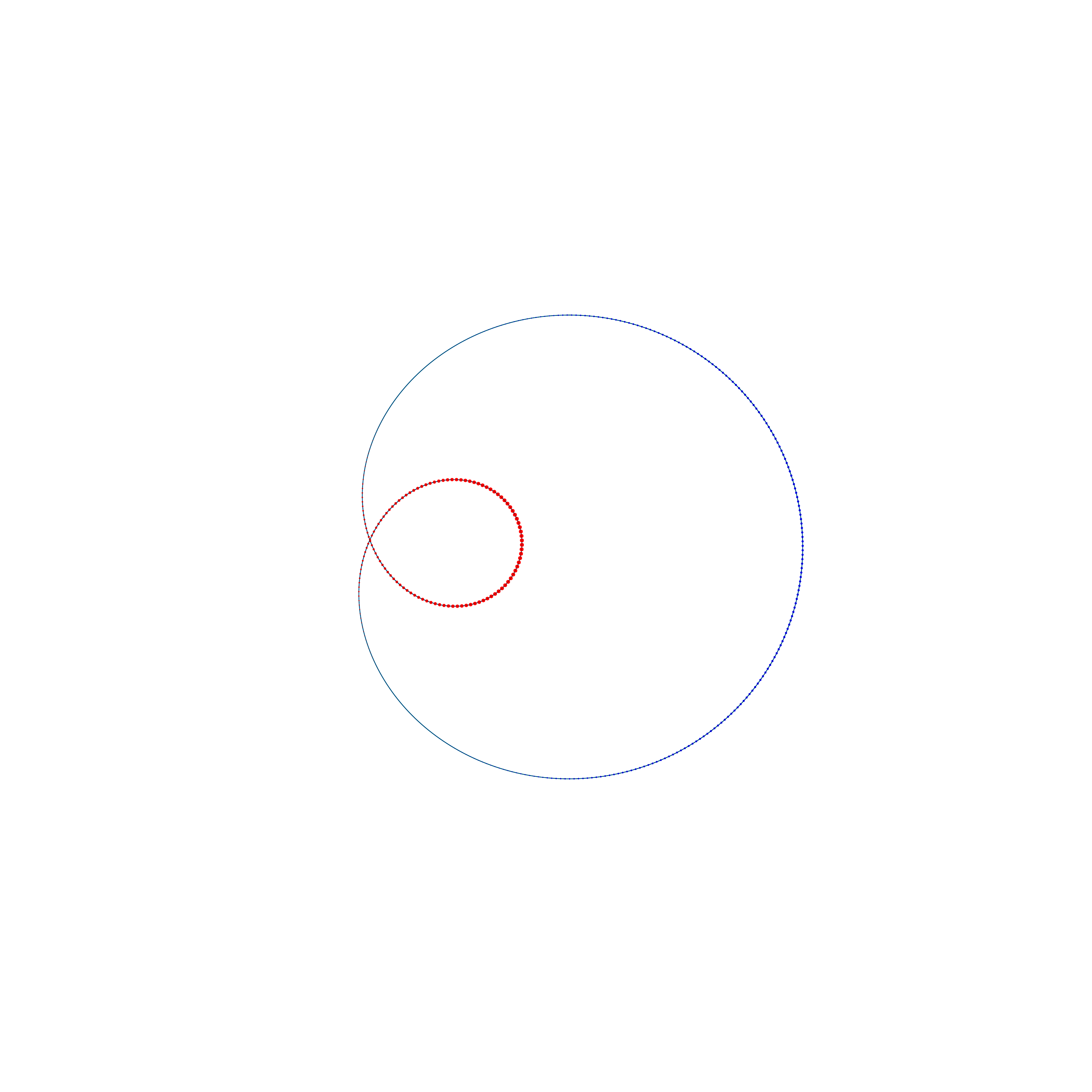}};
        \node[draw=none,fill=none] at (c){\includegraphics[angle=30,width=5cm]{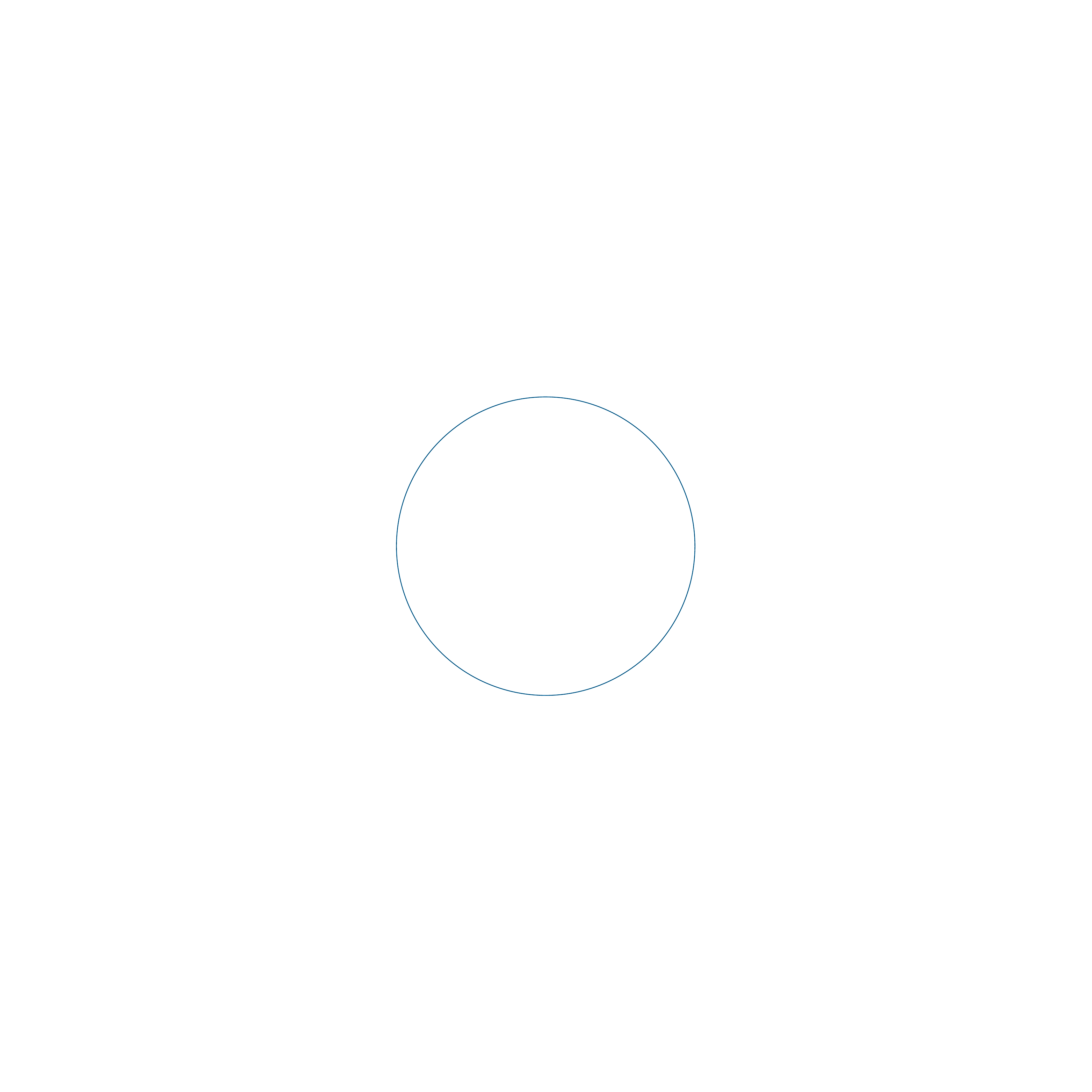}};
        \node[draw=none,fill=none] at (d){\includegraphics[angle=30,width=5cm]{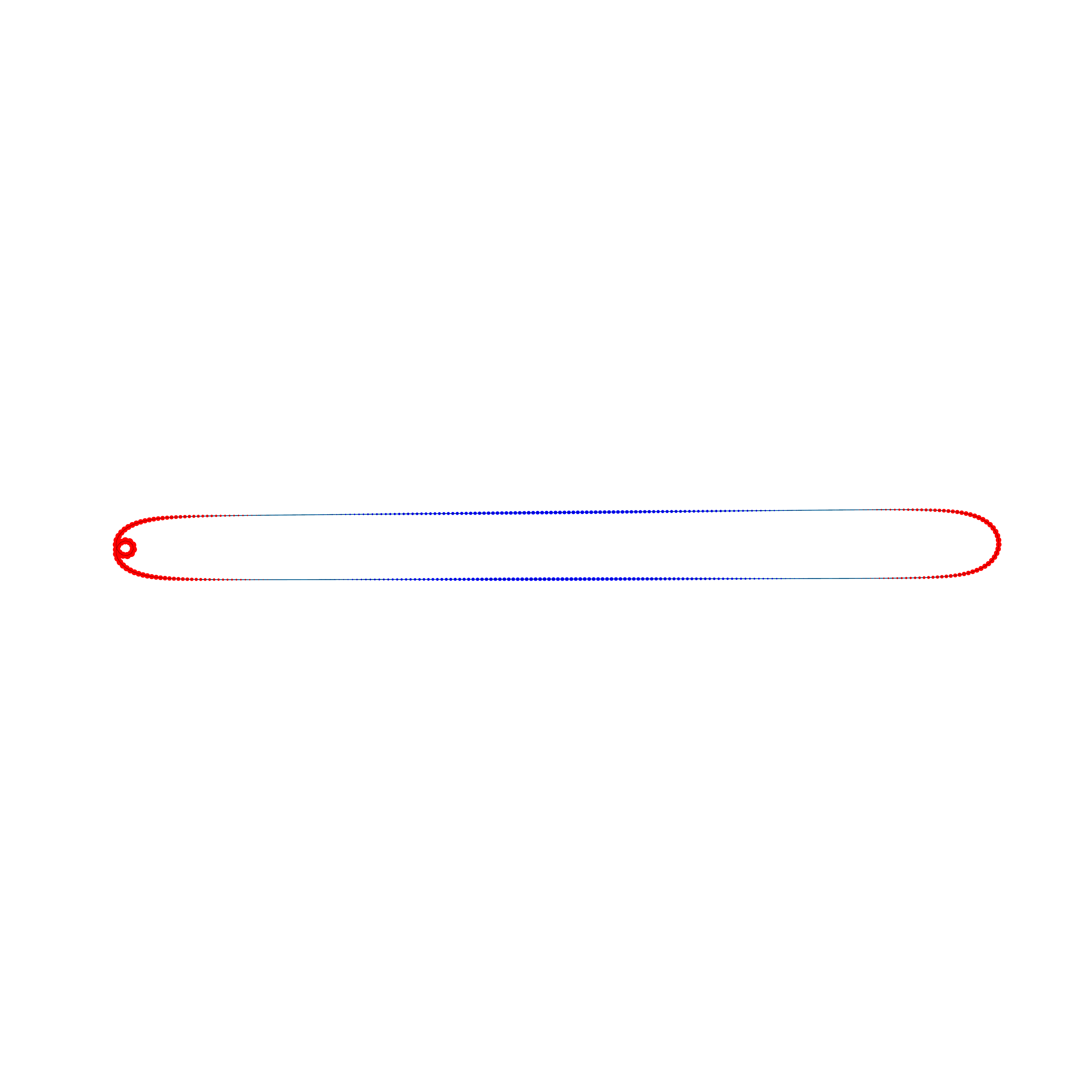}};
        \node[draw=none,fill=none] at (e){\includegraphics[angle=30,width=5cm]{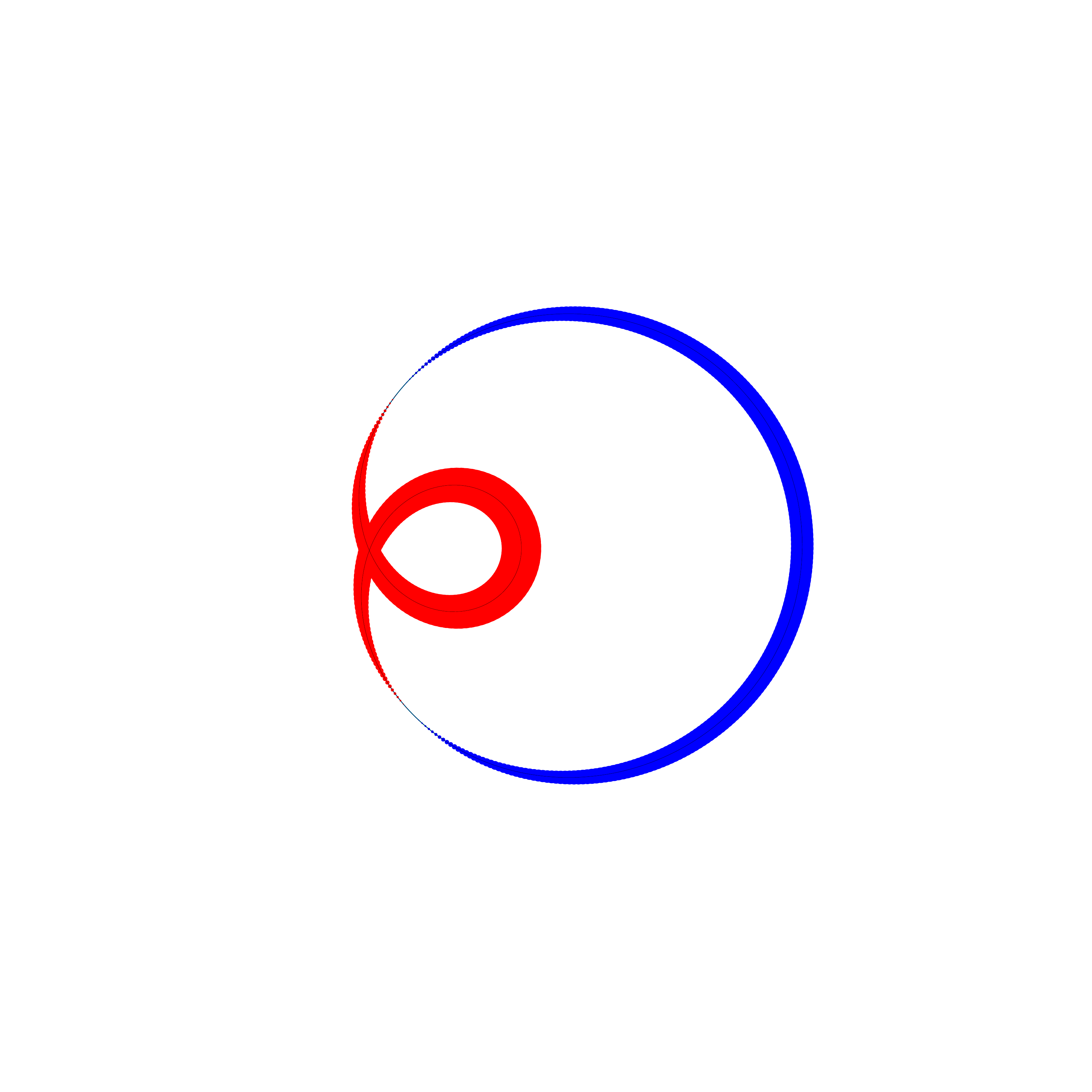}};
        \node[draw=none,fill=none] at (f){\includegraphics[angle=30,width=5cm]{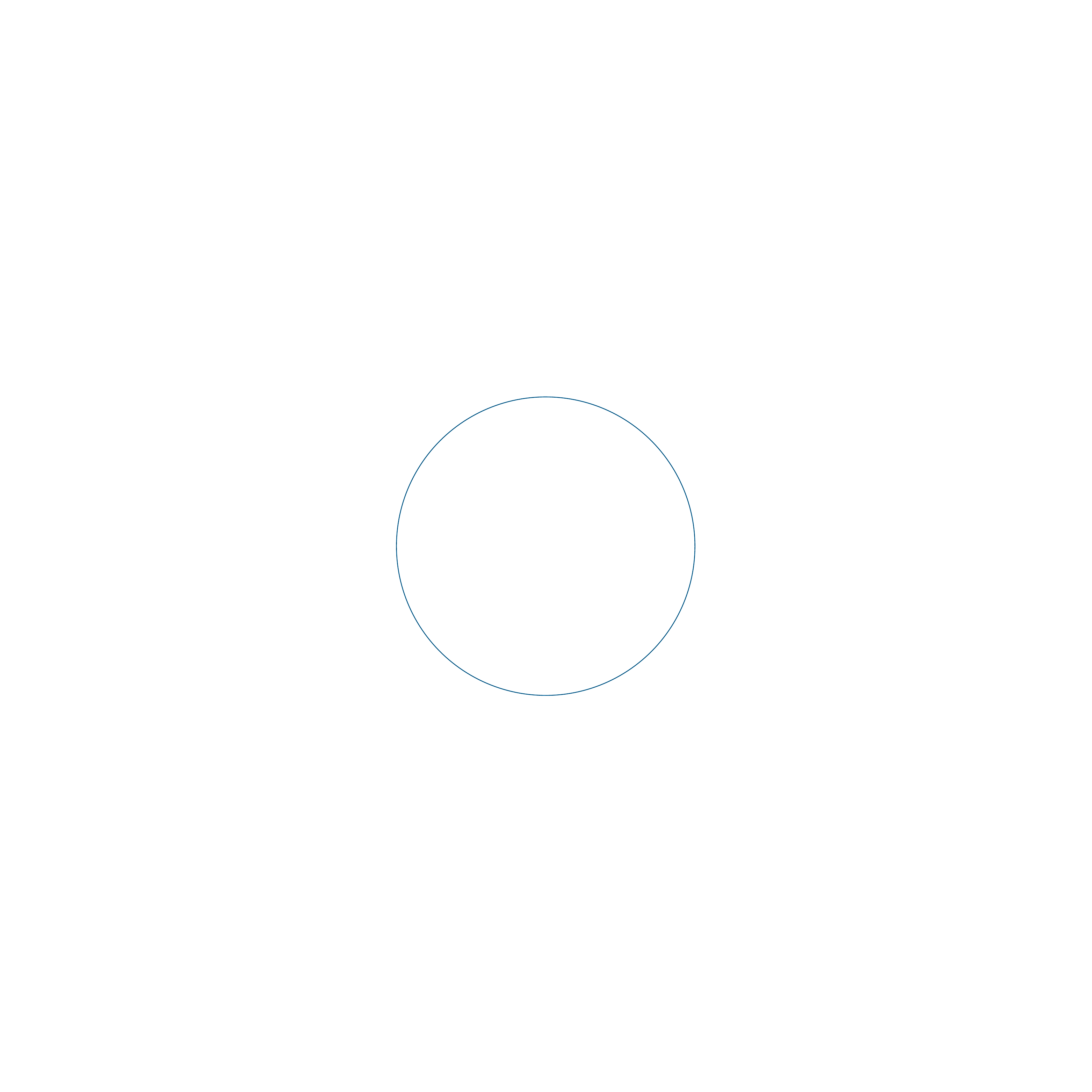}};

        \node[align=center] at ($(center) - (0,1.2)$) {$(\hat\theta_0, \hat\rho_0)$};

        \node[align=center] at ($ (a) - (2.5,0) $) {(A) \\ $\mu = 10^{-2}$ \\ $a = 1$};
        \node[align=center] at ($ (b) - (2.0,0) $) {(B) \\ $\mu = 1$ \\ $a = 1$};
        \node[align=center] at ($ (c) - (2.0,0) $) {(C) \\ $\mu = 5$ \\ $a = 1$};
        \node[align=center] at ($ (d) + (2.5,0) $) {(D) \\ $\mu = 10^{-2}$ \\ $a = 5$};
        \node[align=center] at ($ (e) + (2.0,0) $) {(E) \\ $\mu = 10^{-2}$ \\ $a = 10^{-1}$};
        \node[align=center] at ($ (f) + (2.0,0) $) {(F) \\ $\mu = 10^{-2}$ \\ $a = 10^{-2}$};

        \draw[color=black] (center) circle (2.1);

        \centerarc[thick, ->](0,0)(150:220:7.5cm)(increasing $\mu$);
        \centerarc[thick, ->](0,0)(30:-40:7.5cm)(decreasing $a$);

        \foreach \p in {a,...,f}{
           \draw[arrows = {-Stealth[]}] ($(center)!0.525!(\p)$) -- ($(center)!0.6!(\p)$);
        }

    \end{tikzpicture}
    \caption{Illustration of the dependency of $(\hat\theta_\infty, \hat\rho_\infty)$ on $\mu$ and $\beta$.
      In both cases, $\omega = 2$, $c_0 = 0$ and $\beta(x) = e^{a x}$.
      Left: $(\hat\theta_\infty, \hat\rho_\infty)$ as $\mu$ increases in $\{10^{-2}, 1, 5\}$ with $a = 1$.
      Right: $(\hat\theta_\infty, \hat\rho_\infty)$ as $a$ decreases in $\{5, 10^{-1}, 10^{-2}\}$ with $\mu = 10^{-2}$.
      The initial datum is represented in the center.
    \label{fig:omega 2}
    }
\end{figure}
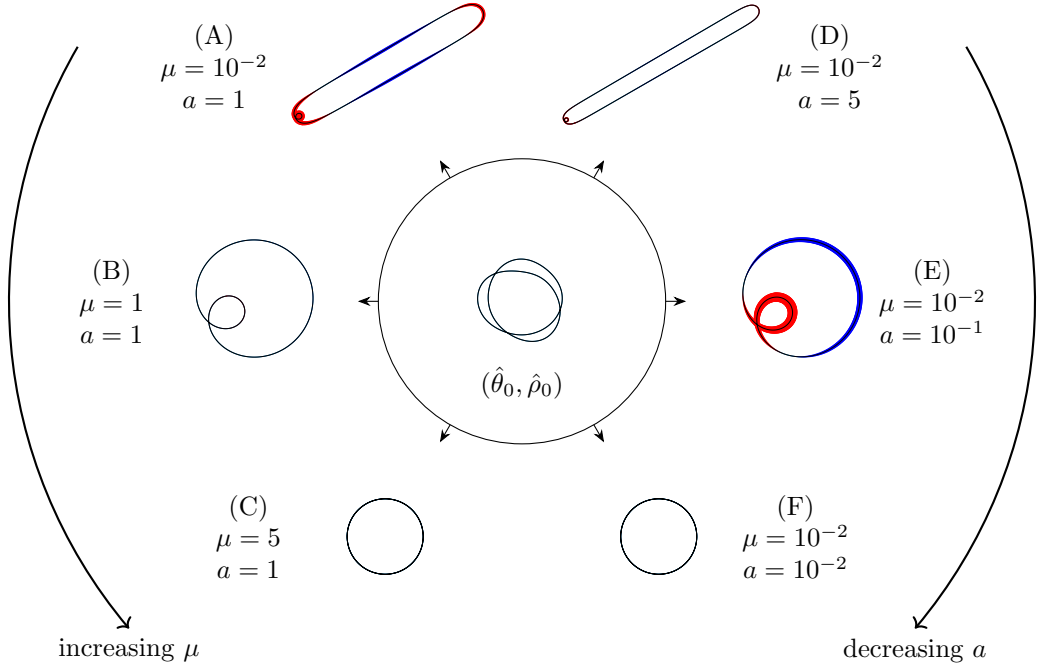

\subsubsection{Convergence to the figure eight for $\omega = 0$}

To conclude this numerical overview, we look at the case $\omega = 0$, for which we recall that the only closed elasticae are multiple coverings of the figure eight, see \Cref{lem:charelasticae}.

More specifically, we consider two cases:
\begin{itemize}
    \item First, $\hat\theta_0$ is given by a hand-drawn curve, with nonconstant $\hat\rho_0$. Here, $c_0 = 2$. See~\Cref{fig:omega 0 figure 8}.
    \item Second, in a very rough attempt to look at the stability of the $2$-fold covering of the figure eight, we consider $\hat\theta_0$ given by two slightly offset figure eights. We pick $c_0 = 0$.
    See~\Cref{fig:omega 0 figure 8 double}.
        
\end{itemize}

In both situations we take $\nu = 0$ and $\beta(x) = 0.1 + x^2$, a choice which fits the assumptions of~\Cref{thm:conv special beta neu} with $\bar{C} = 0$, so that the limit is necessarily a (potentially multiple) covering of the figure eight. We observe convergence to the $1$-fold covering of the homogeneous figure eight in both situations, which is expected in the first case. In the second case, this suggests that the multiple coverings of the figure eight are not stable under the flow, at least for this choice of $\beta$.

\begin{figure}
    \begin{minipage}{0.19\textwidth}
    \begin{center}
        \includegraphics[width=\textwidth]{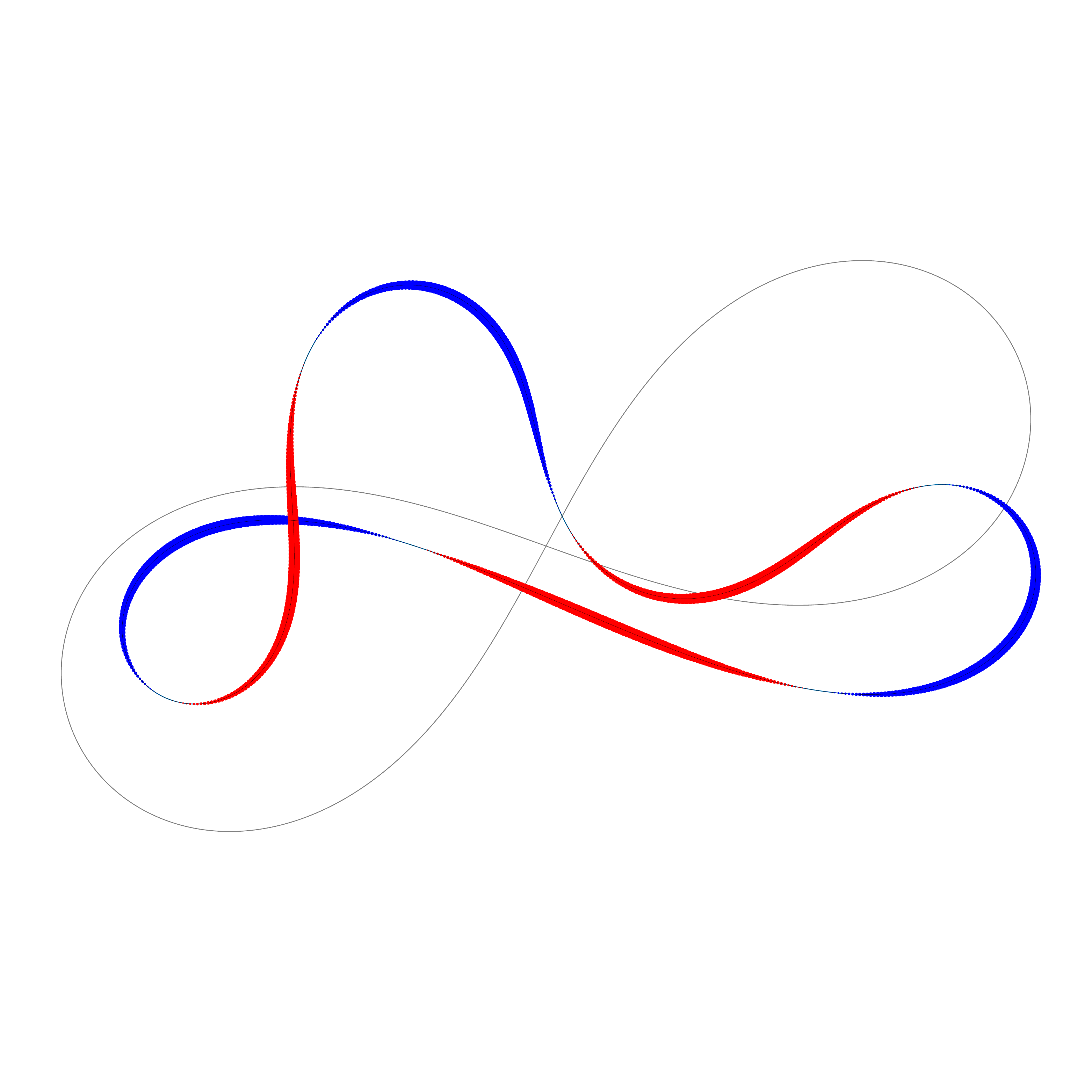}
        $t = 0$
    \end{center}
    \end{minipage}
    \begin{minipage}{0.19\textwidth}
    \begin{center}
        \includegraphics[width=\textwidth]{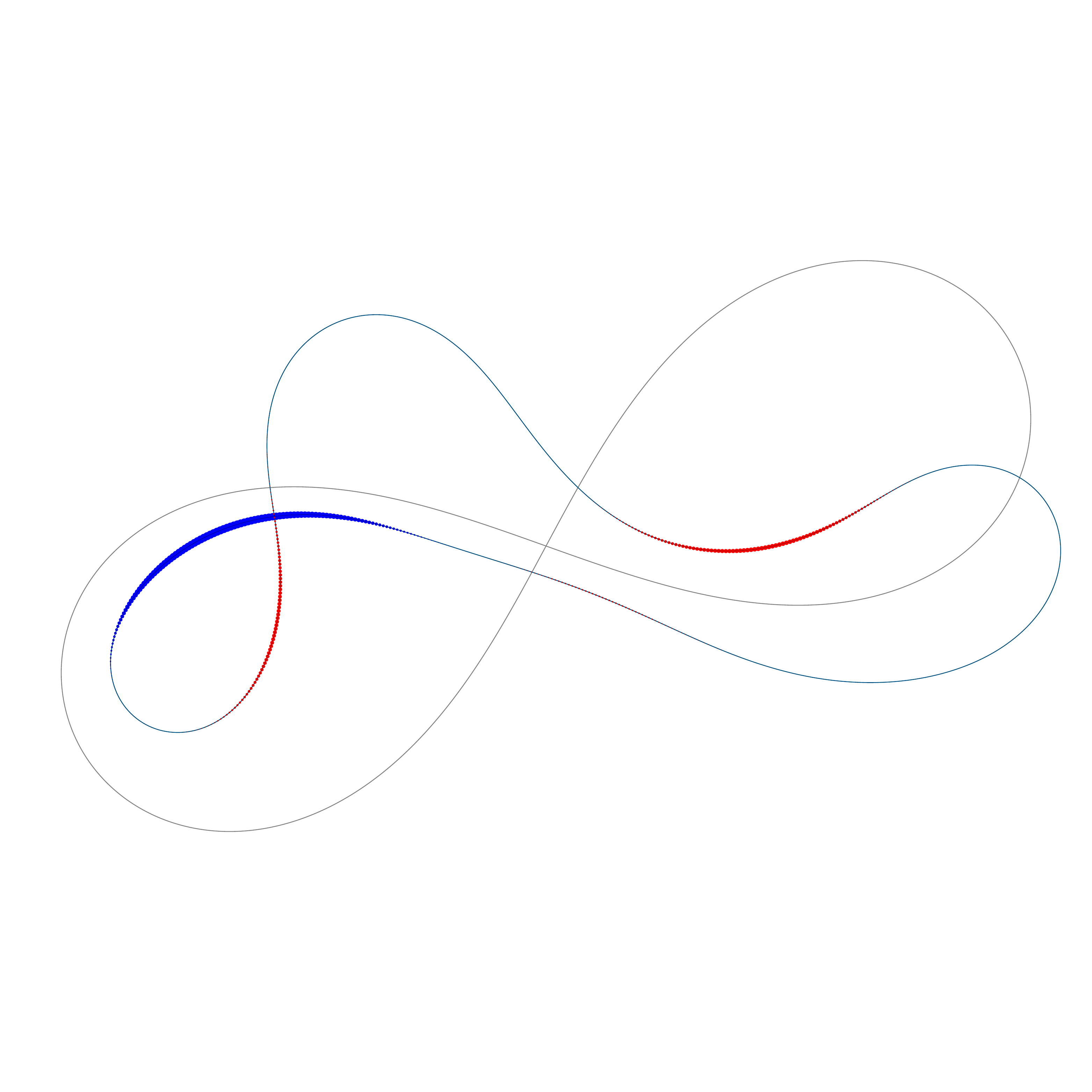}
        $t\approx0.5$
    \end{center}
    \end{minipage}
    \begin{minipage}{0.19\textwidth}
    \begin{center}
        \includegraphics[width=\textwidth]{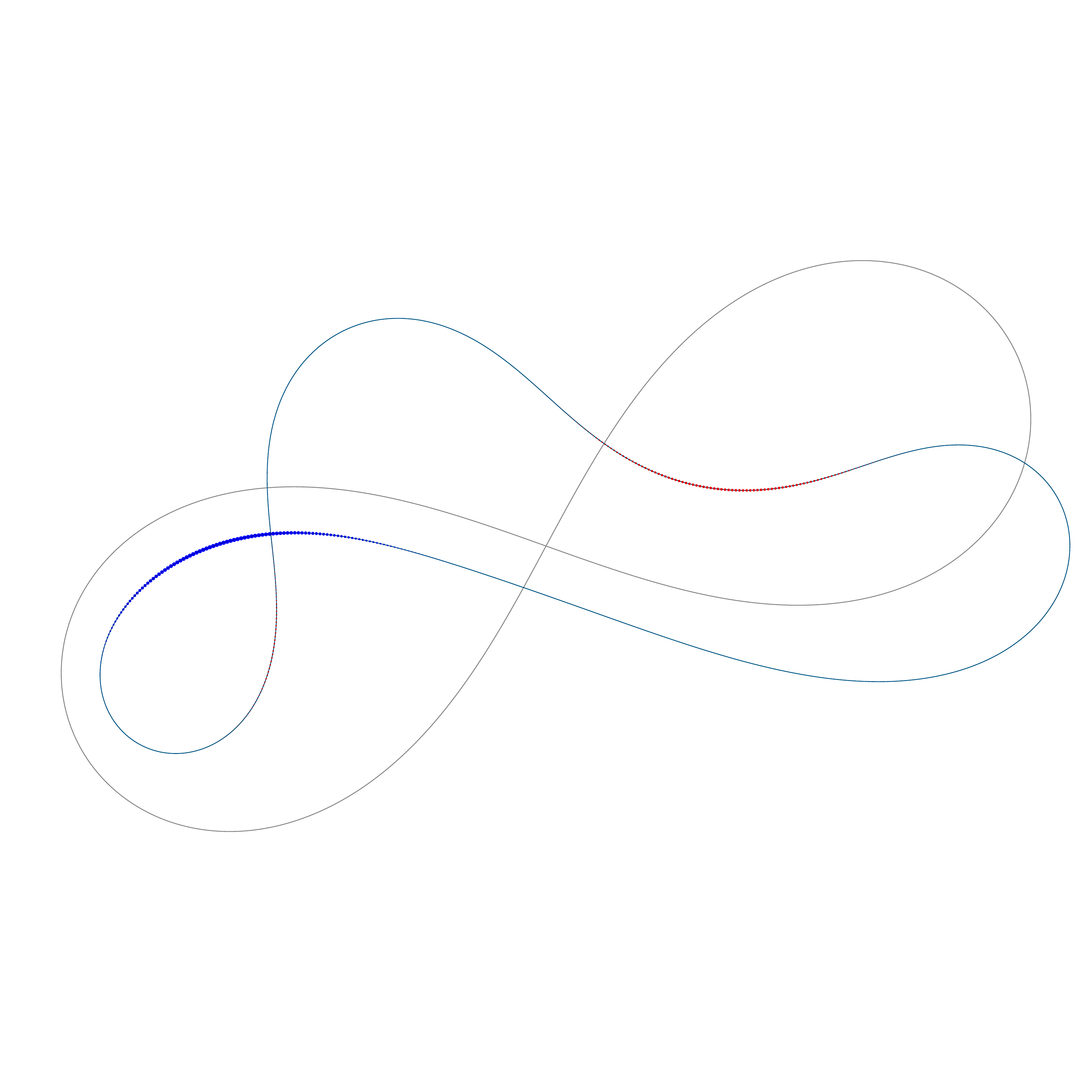}
        $t\approx1$
    \end{center}
    \end{minipage}
    \begin{minipage}{0.19\textwidth}
    \begin{center}
        \includegraphics[width=\textwidth]{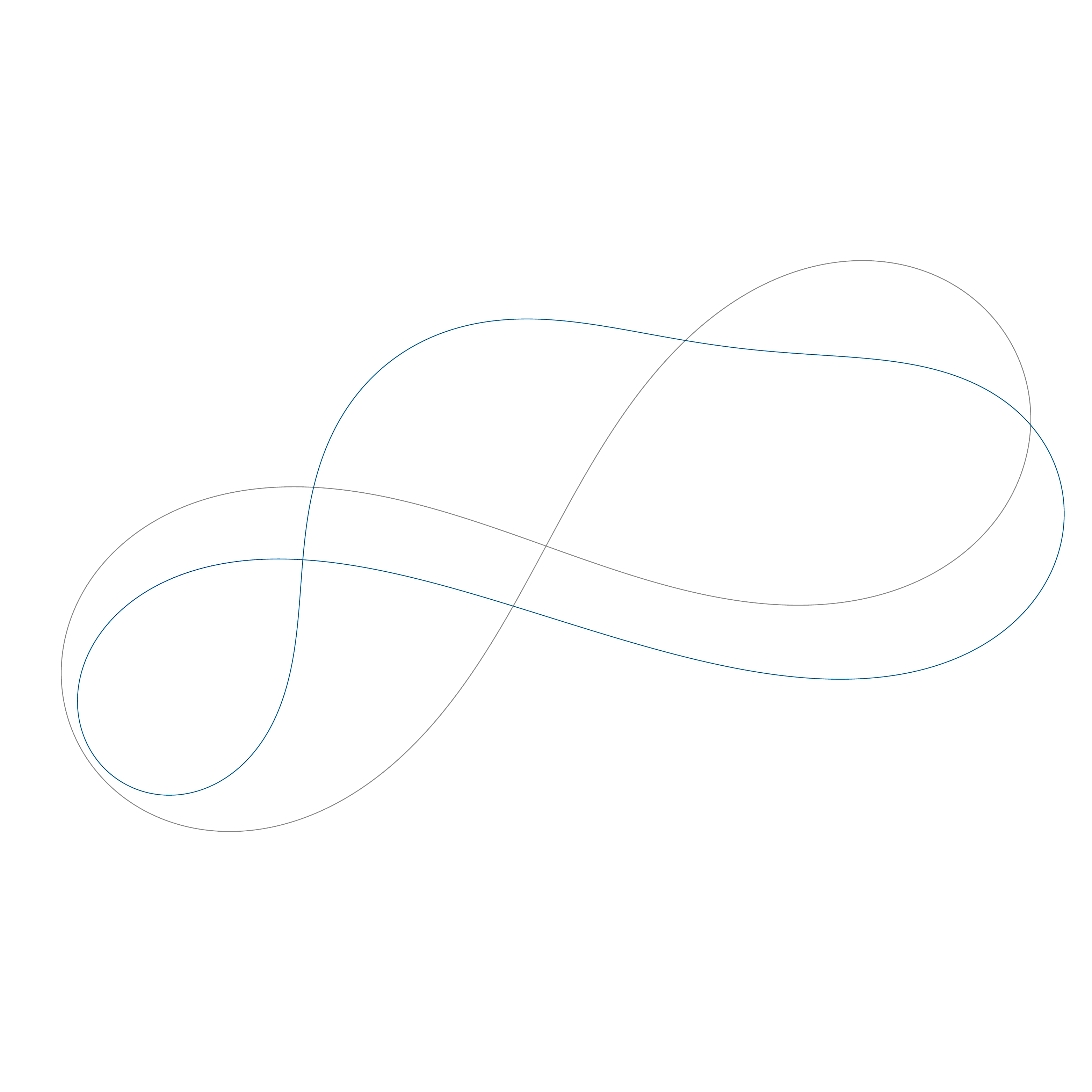}
        $t\approx2$
    \end{center}
    \end{minipage}
    \begin{minipage}{0.19\textwidth}
    \begin{center}
        \includegraphics[width=\textwidth]{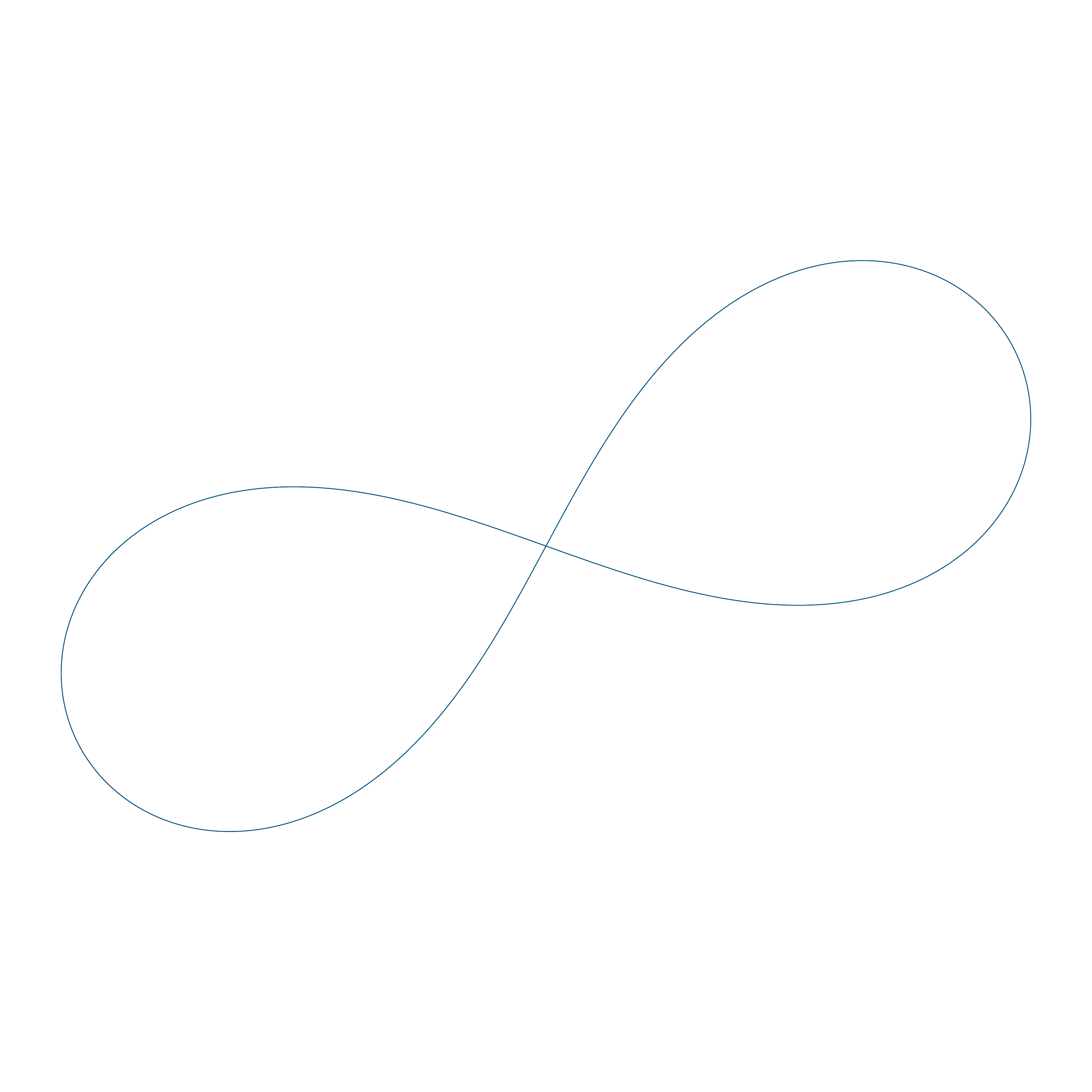}
        $t\approx135$
    \end{center}
    \end{minipage}
    \caption{Convergence of an hand-drawn curve (black) to the figure eight (gray).
    $\mu = 10^{-1}$, $c_0 = 2$ and $\beta(x) = 0.1 + x^2$ is quadratic and positive. We observe convergence to a homogeneous elastica.
    The figure eight in the background is the $1$-fold covering of the figure eight whose integral of the corresponding tangential angle $\theta$ matches that of the initial curve.}
    \label{fig:omega 0 figure 8}
\end{figure}

\begin{figure}
    \begin{minipage}{0.20\textwidth}
    \begin{center}
        \includegraphics[width=\textwidth]{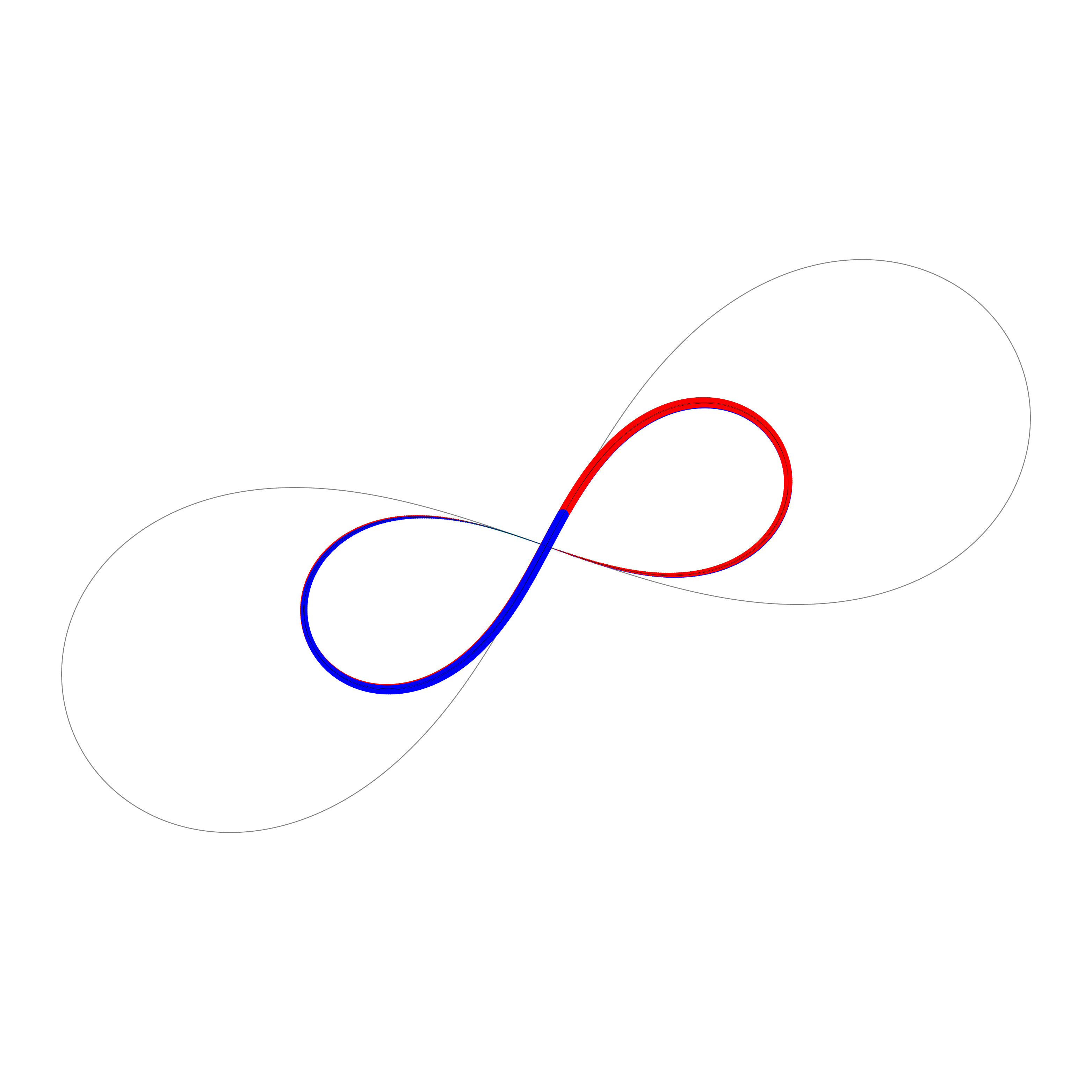}
        $t = 0$
    \end{center}
    \end{minipage}
    \hspace{-0.025\textwidth}
    \begin{minipage}{0.20\textwidth}
    \begin{center}
        \includegraphics[width=\textwidth]{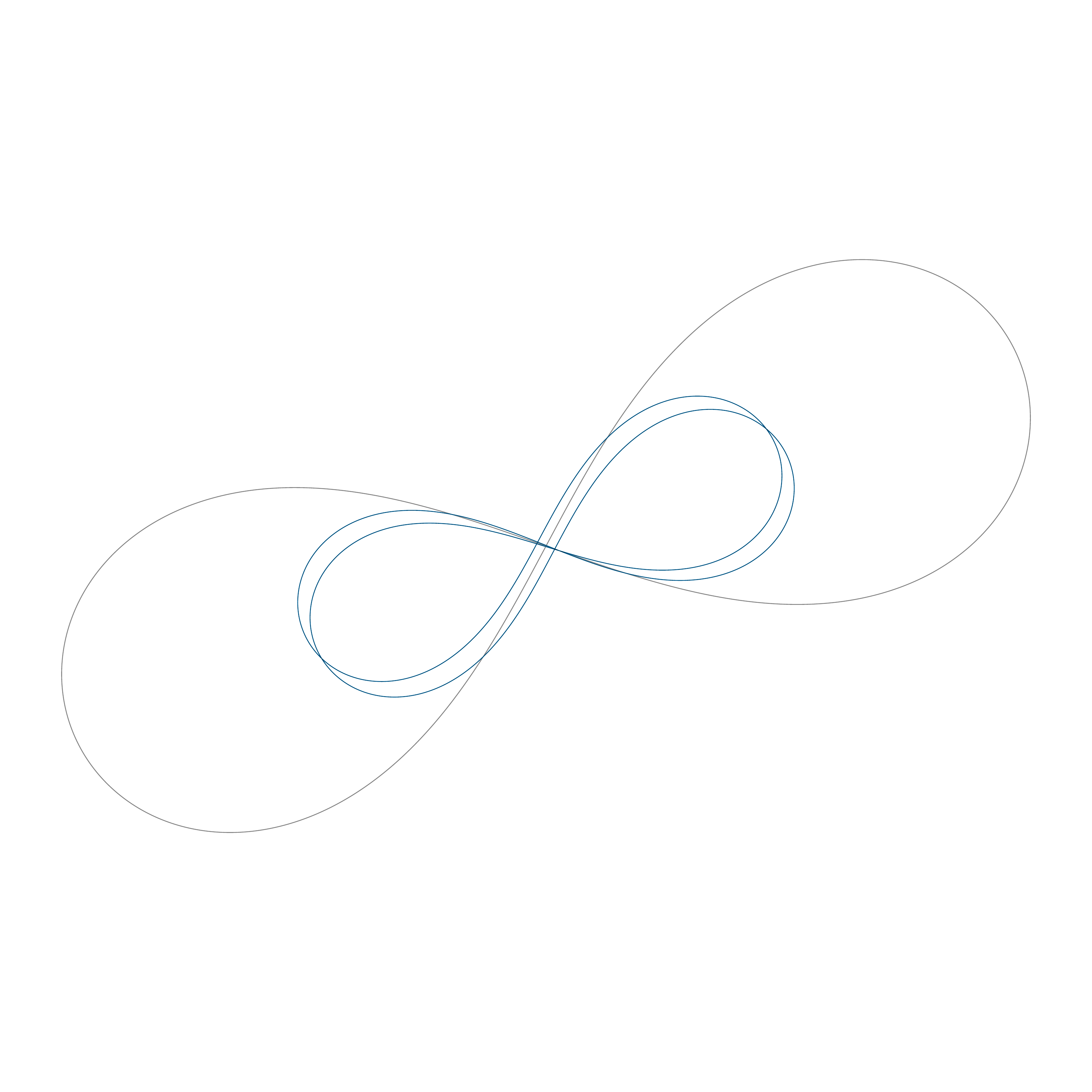}
        $t \approx 36$
    \end{center}
    \end{minipage}
    \hspace{-0.025\textwidth}
    \begin{minipage}{0.20\textwidth}
    \begin{center}
        \includegraphics[width=\textwidth]{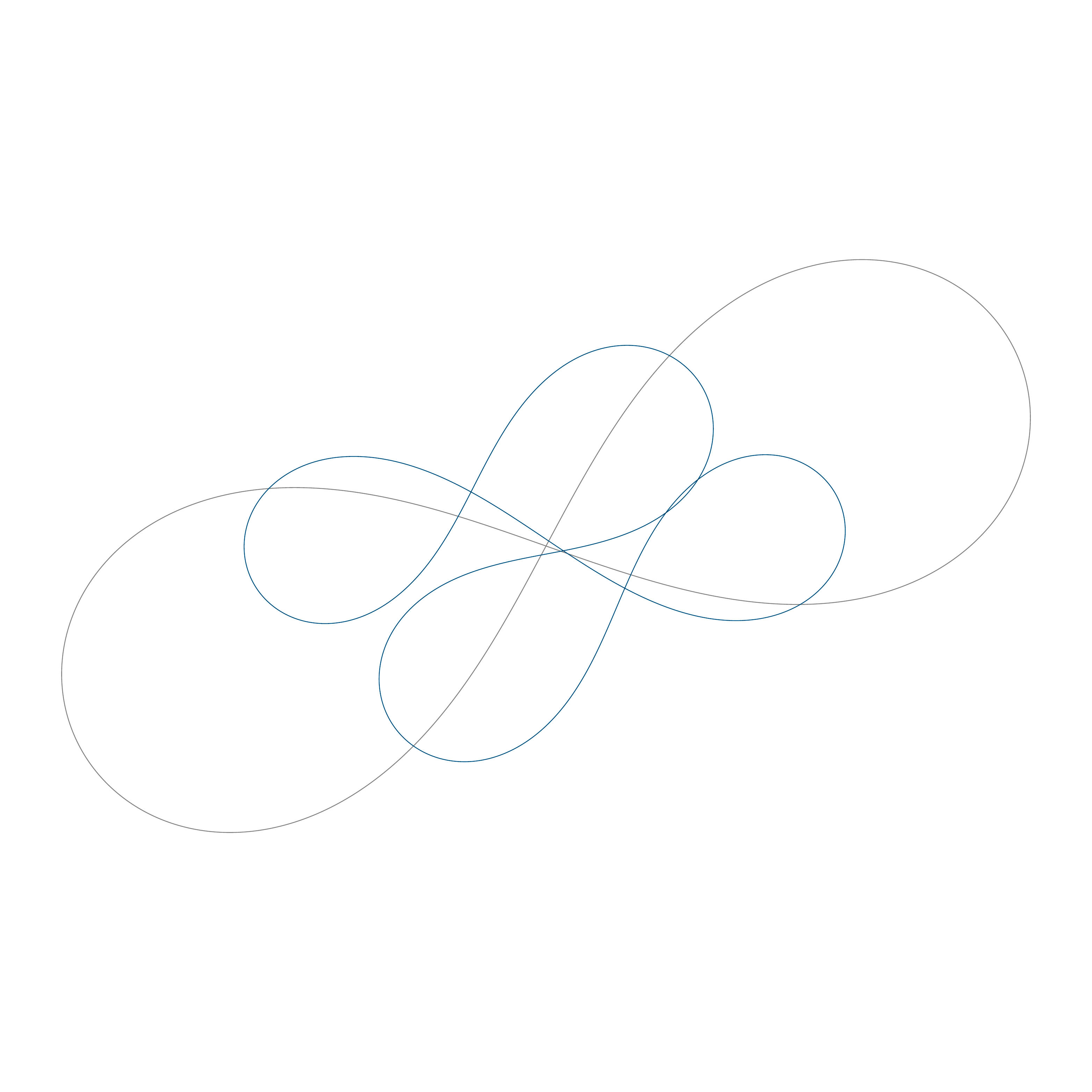}
        $t \approx 60$
    \end{center}
    \end{minipage}
    \hspace{-0.025\textwidth}
    \begin{minipage}{0.20\textwidth}
    \begin{center}
        \includegraphics[width=\textwidth]{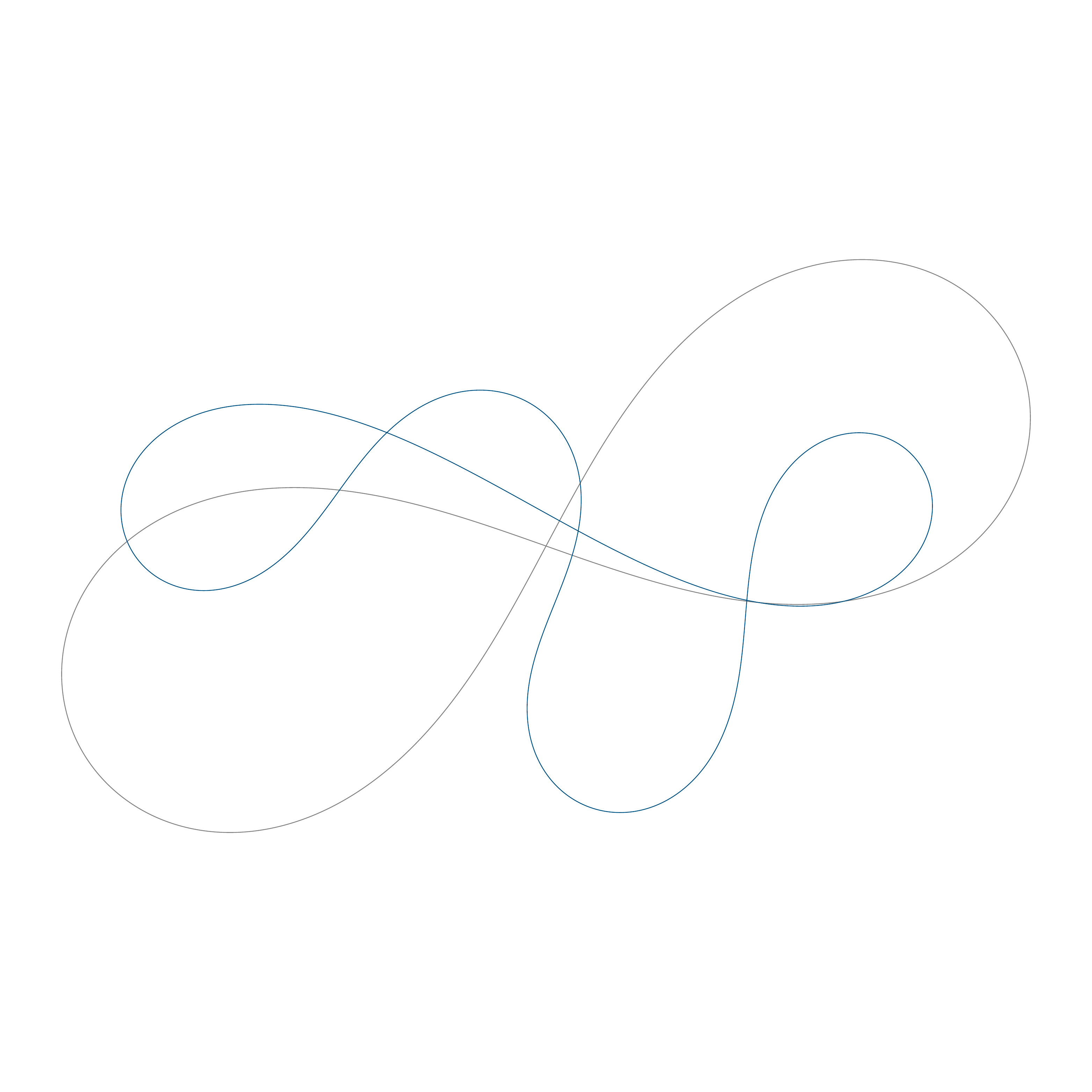}
        $t \approx 70$
    \end{center}
    \end{minipage}
    \hspace{-0.025\textwidth}
    \begin{minipage}{0.20\textwidth}
    \begin{center}
        \includegraphics[width=\textwidth]{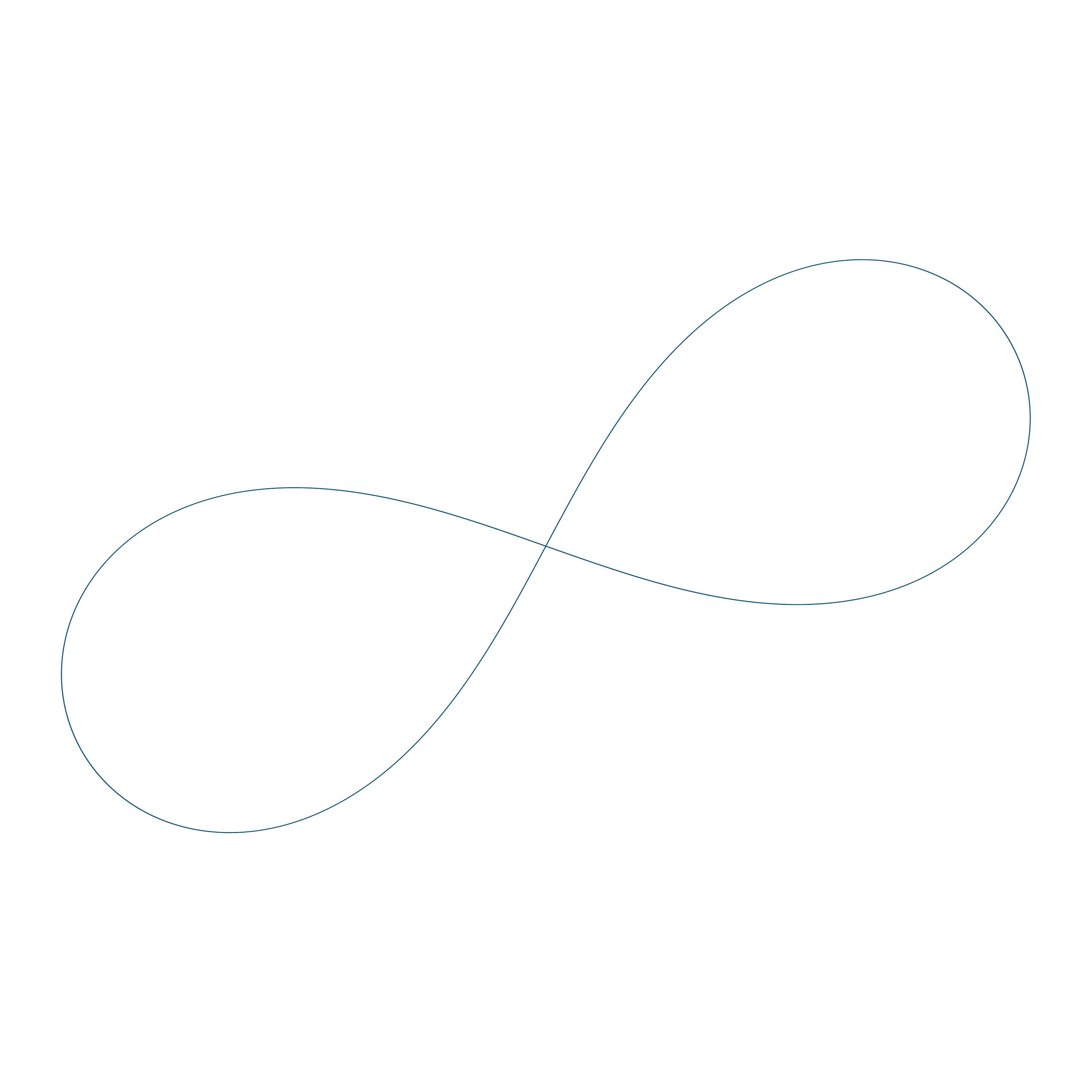}
        $t \approx 100$
    \end{center}
    \end{minipage}
    \caption{Evolution of the perturbation of a $2$-fold covering of the figure eight, where the second is rotated by an angle of $2\pi / 1000$ with respect to the first one.
    Other parameters are $\mu = 10^{-1}$, $c_0 = 0$ and $\beta(x) = 0.1 + x^2$.
The figure eight in the background is the $1$-fold covering of the figure eight whose integral of the corresponding tangential angle $\theta$ matches that of the initial curve.}
    \label{fig:omega 0 figure 8 double}
\end{figure}

\section*{Acknowledgements} The authors acknowledge support by the DFG (German Research Foundation), project no.\ 404870139. The fourth author is additionally supported by the Austrian Science Fund (FWF) project/grant P 32788-N.
\bibliographystyle{abbrv}
\bibliography{biblio}

\begin{thebibliography}{10}

\bibitem{ABG2022}
H.~Abels, F.~B\"{u}rger, and H.~Garcke.
\newblock Qualitative properties for a system coupling scaled mean curvature
  flow and diffusion.
\newblock {\em J. Differential Equations}, 349:236--268, 2023.

\bibitem{MR4530425}
H.~Abels, F.~B\"{u}rger, and H.~Garcke.
\newblock Short time existence for coupling of scaled mean curvature flow and
  diffusion.
\newblock {\em J. Evol. Equ.}, 23(1):Paper No. 14, 46, 2023.

\bibitem{A1988}
S.~B. Angenent.
\newblock The zero set of a solution of a parabolic equation.
\newblock {\em Journal f{\"u}r die reine und angewandte Mathematik (Crelles
  Journal)}, 1988:79 -- 96, 1988.

\bibitem{BHW2003}
T.~Baumgart, S.~T. Hess, and W.~W. Webb.
\newblock Imaging coexisting fluid domains in biomembrane models coupling
  curvature and line tension.
\newblock {\em Nature}, 425:821--824, 2003.

\bibitem{Blatt}
S.~Blatt.
\newblock Loss of convexity and embeddedness for geometric evolution equations
  of higher order.
\newblock {\em J. Evol. Equ.}, 10(1):21--27, 2010.

\bibitem{BJSS2020}
K.~Brazda, G.~Jankowiak, C.~Schmeiser, and U.~Stefanelli.
\newblock Bifurcation of elastic curves with modulated stiffness.
\newblock {\em European Journal of Applied Mathematics}, page 1–27, 2022.

\bibitem{BLS20}
K.~Brazda, L.~Lussardi, and U.~Stefanelli.
\newblock Existence of varifold minimizers for the multiphase
  {C}anham-{H}elfrich functional.
\newblock {\em Calc. Var. Partial Differential Equations}, 59(3):Paper No. 93,
  26, 2020.

\bibitem{Canham}
P.~Canham.
\newblock The minimum energy of bending as a possible explanation of the
  biconcave shape of the human red blood cell.
\newblock {\em J. Theor. Biol.}, 26(1):61--81, 1970.

\bibitem{CMV13}
R.~Choksi, M.~Morandotti, and M.~Veneroni.
\newblock Global minimizers for axisymmetric multiphase membranes.
\newblock {\em ESAIM Control Optim. Calc. Var.}, 19(4):1014--1029, 2013.

\bibitem{DLR2022}
A.~Dall'Acqua, L.~Langer, and F.~Rupp.
\newblock A dynamic approach to heterogeneous elastic wires.
\newblock {\em Journal of Differential Equations}, 392:1--42, 2024.

\bibitem{DLP17}
A.~Dall'Acqua, C.-C. Lin, and P.~Pozzi.
\newblock A gradient flow for open elastic curves with fixed length and clamped
  ends.
\newblock {\em Ann. Sc. Norm. Super. Pisa Cl. Sci. (5)}, 17(3):1031--1066,
  2017.

\bibitem{DP14}
A.~Dall'Acqua and P.~Pozzi.
\newblock A {W}illmore-{H}elfrich {$L^2$}-flow of curves with natural boundary
  conditions.
\newblock {\em Comm. Anal. Geom.}, 22(4):617--669, 2014.

\bibitem{DPS16}
A.~Dall'Acqua, P.~Pozzi, and A.~Spener.
\newblock The {{\L}}ojasiewicz-{S}imon gradient inequality for open elastic
  curves.
\newblock {\em J. Differential Equations}, 261(3):2168--2209, 2016.

\bibitem{DHMV2008}
P.~A. Djondjorov, M.~T. Hadzhilazova, I.~M. Mladenov, and V.~M. Vassilev.
\newblock Explicit parameterization of {Euler’s} elastica.
\newblock {\em Geometry, integrability and quantization}, pages 175 -- 186,
  2008.

\bibitem{dondl2023gammaconvergence}
P.~Dondl, C.~A. Hounkpe, and M.~Jesenko.
\newblock {$\Gamma$}-convergence of a discrete {K}irchhoff rod energy.
\newblock {\em arXiv}, 2306.10936, 2023.

\bibitem{DKS2002}
G.~Dziuk, E.~Kuwert, and R.~Sch\"{a}tzle.
\newblock Evolution of elastic curves in {$\Bbb R^n$}: existence and
  computation.
\newblock {\em SIAM J. Math. Anal.}, 33(5):1228--1245, 2002.

\bibitem{elliott_numerical_2022}
C.~M. Elliott, H.~Garcke, and B.~Kovács.
\newblock Numerical analysis for the interaction of mean curvature flow and
  diffusion on closed surfaces.
\newblock {\em Numerische Mathematik}, 151(4):873--925, Aug. 2022.

\bibitem{EI2005}
J.~Escher and K.~Ito.
\newblock Some dynamic properties of volume preserving curvature driven flows.
\newblock {\em Math. Ann.}, 333(1):213--230, 2005.

\bibitem{Gage_Hamilton_86}
M.~Gage and R.~S. Hamilton.
\newblock The heat equation shrinking convex plane curves.
\newblock {\em J. Differential Geom.}, 23(1):69--96, 1986.

\bibitem{Helfrich}
W.~Helfrich.
\newblock Elastic properties of lipid bilayers: Theory and possible
  experiments.
\newblock {\em Zeitschrift für Naturforschung C}, 28(11):693--703, 1973.

\bibitem{Helmers11}
M.~Helmers.
\newblock Snapping elastic curves as a one-dimensional analogue of
  two-component lipid bilayers.
\newblock {\em Math. Models Methods Appl. Sci.}, 21(5):1027--1042, 2011.

\bibitem{Helmers15}
M.~Helmers.
\newblock Convergence of an approximation for rotationally symmetric two-phase
  lipid bilayer membranes.
\newblock {\em Q. J. Math.}, 66(1):143--170, 2015.

\bibitem{JulicherLipowsky93}
F.~J\"ulicher and R.~Lipowsky.
\newblock Domain-induced budding of vesicles.
\newblock {\em Phys. Rev. Lett.}, 70:2964--2967, May 1993.

\bibitem{LS1984}
J.~Langer and D.~A. Singer.
\newblock {The total squared curvature of closed curves}.
\newblock {\em Journal of Differential Geometry}, 20(1):1 -- 22, 1984.

\bibitem{Lin}
C.-C. Lin.
\newblock {$L^2$}-flow of elastic curves with clamped boundary conditions.
\newblock {\em J. Differential Equations}, 252(12):6414--6428, 2012.

\bibitem{LLS15}
C.-C. Lin, Y.-K. Lue, and H.~R. Schwetlick.
\newblock The second-order {$L^2$}-flow of inextensible elastic curves with
  hinged ends in the plane.
\newblock {\em J. Elasticity}, 119(1-2):263--291, 2015.

\bibitem{L1989}
A.~Linn\'{e}r.
\newblock Some properties of the curve straightening flow in the plane.
\newblock {\em Trans. Amer. Math. Soc.}, 314(2):605--618, 1989.

\bibitem{Linner1996}
A.~Linn\'{e}r.
\newblock Unified representations of nonlinear splines.
\newblock {\em J. Approx. Theory}, 84(3):315--350, 1996.

\bibitem{MR4277362}
C.~Mantegazza, A.~Pluda, and M.~Pozzetta.
\newblock A survey of the elastic flow of curves and networks.
\newblock {\em Milan J. Math.}, 89(1):59--121, 2021.

\bibitem{MantegazzaPozzetta21}
C.~Mantegazza and M.~Pozzetta.
\newblock The {{\L}}ojasiewicz-{S}imon inequality for the elastic flow.
\newblock {\em Calc. Var. Partial Differential Equations}, 60(1):Paper No. 56,
  17, 2021.

\bibitem{MG2005}
H.~T. McMahon and J.~L. Gallop.
\newblock Membrane curvature and mechanisms of dynamic cell membrane
  remodelling.
\newblock {\em Nature}, 438:590--596, 2005.

\bibitem{MMR2021}
T.~Miura, M.~Müller, and F.~Rupp.
\newblock Optimal thresholds for preserving embeddedness of elastic flows.
\newblock {\em To appear in Amer. J. Math.}, arXiv:2106.09549, 2021.

\bibitem{MR2021}
M.~M\"{u}ller and F.~Rupp.
\newblock A {L}i-{Y}au inequality for the 1-dimensional {W}illmore energy.
\newblock {\em Adv. Calc. Var.}, 16(2):337--362, 2023.

\bibitem{MR4080255}
M.~M\"{u}ller and A.~Spener.
\newblock On the convergence of the elastic flow in the hyperbolic plane.
\newblock {\em Geom. Flows}, 5(1):40--77, 2020.

\bibitem{NP2020}
M.~Novaga and P.~Pozzi.
\newblock A second order gradient flow of {$p$}-elastic planar networks.
\newblock {\em SIAM J. Math. Anal.}, 52(1):682--708, 2020.

\bibitem{OPW}
S.~Okabe, P.~Pozzi, and G.~Wheeler.
\newblock A gradient flow for the {$p$}-elastic energy defined on closed planar
  curves.
\newblock {\em Math. Ann.}, 378(1-2):777--828, 2020.

\bibitem{ConstrLoja}
F.~Rupp.
\newblock On the {{\L}}ojasiewicz-{S}imon gradient inequality on submanifolds.
\newblock {\em J. Funct. Anal.}, 279(8):108708, 33, 2020.

\bibitem{W1993}
Y.~Wen.
\newblock {$L^2$ flow of curve straightening in the plane}.
\newblock {\em Duke Mathematical Journal}, 70(3):683 -- 698, 1993.

\bibitem{Z1985}
E.~Zeidler.
\newblock {\em Nonlinear Functional Analysis and its Applications - III:
  Variational Methods and Optimization. Translated by L. F. Boron}.
\newblock Springer-Verlag, New York, 1985.

\end{thebibliography}

\end{document}